\expandafter\def\csname ver@fixltx2e.sty\endcsname{}
 \documentclass[journal,10pt]{IEEEtran}

\let\proof\relax 
\let\endproof\relax
\usepackage{amsthm}
\usepackage{amsfonts}
\usepackage{amssymb,bm}
\usepackage{mathrsfs}
\usepackage{mathtools}
\usepackage{amsmath}

\usepackage[switch]{lineno}
\usepackage[named]{algo}
\usepackage[noend]{algpseudocode}
\usepackage{algorithm}

\usepackage{cite}
\usepackage{balance}

\let\labelindent\relax
\usepackage[inline]{enumitem}

\usepackage{verbatim}

\usepackage{dblfloatfix}
\usepackage{url}
\usepackage{graphicx}
\usepackage{textcomp}

\usepackage{accents}

\newlength{\dhatheight}
\newcommand{\doublehat}[1]{%
    \settoheight{\dhatheight}{\ensuremath{\hat{#1}}}%
    \addtolength{\dhatheight}{-0.25ex}%
    \hat{\vphantom{\rule{1pt}{\dhatheight}}%
    \smash{\hat{#1}}}}

\usepackage{array}
\newcolumntype{P}[1]{>{\centering\arraybackslash}p{#1}}

\usepackage{xcolor}

\newtheorem{theorem}{Theorem}
\newtheorem{definition}{Definition}
\newtheorem{lemma}{Lemma}
\newtheorem{claim}{Claim}

\newtheorem{remark}{Remark}
\newenvironment{claimproof}{\proof}{\endproof}

\def\BibTeX{{\rm B\kern-.05em{\sc i\kern-.025em b}\kern-.08em
    T\kern-.1667em\lower.7ex\hbox{E}\kern-.125emX}}

\begin{document}
\title{Inverse Unscented Kalman Filter}

\author{Himali Singh, Kumar Vijay Mishra$^\ast$ and Arpan Chattopadhyay$^\ast$\vspace{-24pt}
\thanks{$^\ast$K. V. M. and A. C. have made equal contributions.}
\thanks{H. S. and A. C. are with the Electrical Engineering Department, Indian Institute of Technology (IIT) Delhi, India. A. C. is also associated with the Bharti School of Telecommunication Technology and Management, IIT Delhi. Email: \{eez208426, arpanc\}@ee.iitd.ac.in.} 
\thanks{K. V. M. is with the United States DEVCOM Army Research Laboratory, Adelphi, MD 20783 USA. E-mail: kvm@ieee.org.}
\thanks{A. C. acknowledges support via the professional development fund and professional development  allowance from IIT Delhi, grant no. GP/2021/ISSC/022 from I-Hub Foundation for Cobotics and grant no. CRG/2022/003707 from Science and Engineering Research Board (SERB), India. H. S. acknowledges support via Prime Minister Research Fellowship. K. V. M. acknowledges support from the National Academies of Sciences, Engineering, and Medicine via Army Research Laboratory Harry Diamond Distinguished Fellowship.}
\thanks{The conference precursor of this work has been accepted for publication at the 2023 IEEE Conference on Decision and Control (CDC).}
}

\maketitle

\begin{abstract}
Rapid advances in designing cognitive and counter-adversarial systems have motivated the development of inverse Bayesian filters. In this setting, a cognitive `adversary' tracks its target of interest via a stochastic framework such as a Kalman filter (KF). The target or `defender' then employs another \textit{inverse} stochastic filter to infer the forward filter estimates of the defender computed by the adversary. For linear systems, the \textit{inverse Kalman filter} (I-KF) has been recently shown to be effective in these counter-adversarial applications. In the paper, contrary to prior works, we focus on non-linear system dynamics and formulate the inverse unscented KF (I-UKF) to estimate the defender's state based on the unscented transform, or equivalently, statistical linearization technique. We then generalize this framework to unknown systems by proposing reproducing kernel Hilbert space-based UKF (RKHS-UKF) to learn the system dynamics and estimate the state based on its observations. Our theoretical analyses to guarantee the stochastic stability of I-UKF and RKHS-UKF in the mean-squared sense show that, provided the forward filters are stable, the inverse filters are also stable under mild system-level conditions. We show that, despite being a suboptimal filter, our proposed I-UKF is a conservative estimator, i.e., I-UKF's estimated error covariance upper-bounds its true value. Our numerical experiments for several different applications demonstrate the estimation performance of the proposed filters using \textit{recursive} Cram\'{e}r-Rao lower bound and non-credibility index (NCI).
\end{abstract}


\section{Introduction}
\label{sec:introduction}
Inference and control form an integral part of many dynamic systems in various engineering applications, including navigation\cite{simon2006optimal}, guidance \cite{bar2004estimation}, and radar target tracking \cite{bell2015cognitive}. Often, these applications involve \textit{cognitive} agents that sense the environment and, based on the information gathered, adapt their actions to achieve optimal performance. In military surveillance, for instance, a cognitive radar\cite{mishra2020toward} adapts its transmit waveform and receive processing to improve target detection \cite{mishra2017performance} and tracking \cite{bell2015cognitive,sharaga2015optimal}. In this context, \textit{inverse cognition} has been recently proposed for a \textit{defender} agent to detect its cognitive adversarial \textit{attacker} agent and infer the information adversary has learned about the defender \cite{krishnamurthy2019how,krishnamurthy2020identifying}. This aids in designing counter-adversarial systems to assist or desist the adversary\cite{mattila2020hmm,krishnamurthy2019how}. For example, an intelligent target senses its adversarial radar's waveform adaptations and designs smart interference to force the latter to change its actions\cite{krishnamurthy2021adversarial,kang2023}. In \cite{krishnamurthy2019how}, optimal probe signals are designed to estimate the adversary's sensor characteristics. In \cite{mattila2017inverse}, observations of an automatic sleep-tracking system are estimated for its real-time fault diagnosis. Similar examples abound in interactive learning and cyber-physical security\cite{krishnamurthy2019how,mattila2020hmm}.

In order to predict the attacker's future actions, a defender requires an estimate of the attacker's inference. In this context, \textit{inverse Bayesian filtering}\cite{krishnamurthy2019how} has been proposed to infer the attacker's inference at the defender's end. An attacker employs a (forward) Bayesian filter to compute a posterior distribution for the defender's state given its noisy observations. A common example is the Kalman filter (KF), which recursively estimates the state in linear Gaussian state-space models and is optimal in the minimum mean-squared error sense. The attacker further cognitively adapts its actions based on this inference. An inverse Bayesian filter then provides a posterior distribution for the forward Bayesian filter's estimate given noisy measurements of the attacker's actions.

In this inverse filtering context, \cite{mattila2020hmm} proposed \textit{inverse hidden Markov model} to estimate the adversary's observations and observation likelihood. In \cite{krishnamurthy2019how}, inverse filtering was extended to general linear Gaussian state-space models. Here, the attacker employed a KF to estimate the defender's state. Then, an inverse KF (I-KF) was developed to estimate the \textit{attacker's estimate of the defender's state}. In practice, counter-adversarial systems are non-linear. Here, the adversary may also employ the extended KF (EKF), which extends the standard KF to the non-linear dynamics using Taylor series expansion. For this setting, our previous work \cite{singh2022inverse, singh2022inverse_part1} proposed an inverse EKF (I-EKF) for the defender. While EKF is a widely used non-linear filter, it is sensitive to initialization/modeling errors and performs poorly when the system is considerably non-linear \cite{li2017approximate}. In many practical applications, the computation of the Jacobian matrices required for EKF is also non-trivial. In inverse cognition, some of these drawbacks may be addressed through more advanced variants of I-EKF \cite{singh2022inverse_part2}.

In spite of its ease of implementation, EKF suffers from linearization errors in highly non-linear applications\cite{daum2005nonlinear}, which can be efficiently tackled through derivative-free sigma-point KFs (SPKFs)\cite{bhaumik2019nonlinear}, which are based on a weighted sum of function evaluations at a finite number of deterministic \textit{sigma-points}. These points aim to approximate the posterior probability density of a random variable under non-linear transformation. For example, the unscented KF (UKF) \cite{julier2004unscented} draws points using the unscented transform and delivers estimates that are exact in mean for monomials up to third-degree; the covariance computation is exact only for linear functions. The unscented transform is further equivalent to statistically linearizing the non-linear functions using specific regression points such that UKF is a special case of linear regression KF\cite{lefebvre2002comment}. To this end, unlike EKF, the UKF takes into account the increased uncertainty due to linearization errors, providing enhanced performance in many applications. Square-root UKF\cite{van2001square} and Gaussian-sum UKF\cite{kottakki2014state} have also been proposed for, respectively, enhanced numerical stability and non-Gaussian models. In this paper, we focus on inverse filtering based on the unscented transform technique.

\vspace{-8pt}
\subsection{Prior Art}
Conventionally, inverse filtering was limited to non-dynamic systems for applications such as system identification, fault detection, image deblurring and signal deconvolution\cite{idier2013bayesian,gustafsson2007statistical}. However, recent cognitive and counter-adversarial systems applications have motivated the design of inverse stochastic filters. These types of inverse problems may be traced to \cite{kalman1964linear} that aimed to find the cost criterion for a given control policy. An analogous formulation also appears in inverse reinforcement learning (IRL) where the associated reward function is learned passively\cite{ng2000algorithms}. The inverse cognition agent, on the other hand, actively probes its adversarial agent and, hence, can be considered as a generalization of IRL.

The UKF is also related to general approximate Bayesian inference methods that are encountered in machine learning literature\cite{opper1999bayesian,minka2013expectation}. In this context, the UKF belongs to the class of general Gaussian filters that assume a Gaussian posterior distribution for the underlying state. The assumed posterior's mean and covariance are then updated recursively using the observations. These Gaussian filtering techniques form a subset of assumed density filtering (ADF)\cite{maybeck1982stochastic} or online Bayesian learning\cite{opper1999bayesian}, which successively approximates the posterior distribution. Expectation propagation is an extension of ADF in which new observations are utilized to iteratively modify earlier estimates\cite{minka2013expectation}. On the contrary, sequential Monte-Carlo (MC)-based methods (e.g., particle filter (PF) \cite{li2017approximate}) do not assume any posterior form and are also applicable to non-Gaussian systems. The PF framework has also been used to realize Bernoulli filters\cite{ristic2013tutorial} (for randomly switching systems), possibility PFs\cite{ristic2019robust} (for mismatched models), and probability hypothesis density (PHD)\cite{mahler2003multitarget} filters (for high-dimensional multi-object Bayesian inference). However, in general, MC approaches are computationally expensive\cite{ristic2003beyond}.

The inverse stochastic filters developed in \cite{krishnamurthy2019how,singh2022inverse_part1} assume perfect system model information on the part of both attacker and defender. However, in many applications, the prior system model information is not available. In the inverse cognition scenario, the attacker may lack the defender's state evolution information, while the defender may be unaware of the attacker's forward filter and adaptation strategy. The uncertainty in system parameters limits the applicability of the inverse filters developed so far. Some prior works on (forward) non-linear filtering have addressed the unknown system case using kernel-based techniques. While \cite{zhu2011extended} suggested coupling kernel recursive least squares (KRLS)\cite{engel2004kernel} with EKF to learn an unknown non-linear measurement model, the state transition was assumed to be known and linear. A conditional embedding operator was used to reformulate the KF algorithm in \cite{zhu2013learning} for non-linear state-transition functions in reproducing kernel Hilbert space (RKHS), but with linear observations. In our prior work\cite{singh2022inverse_part2}, we adopted an iterative expectation maximization (EM) for non-linear parameter learning and developed RKHS-EKF. However, similar to EKF, the RKHS-EKF's performance also degrades because of the linearization of non-linear functions. In this paper, we further develop RKHS-based UKF for improved performance.

\vspace{-8pt}
\subsection{Our contributions}
Preliminary results of this work appeared in our conference publication \cite{singh2023counter}, where only inverse UKF (I-UKF) was formulated and proofs of stability guarantees were not included. Our main contributions in this paper are:\\
\textbf{1) Inverse UKF.} To address the limitations of I-EKF's linearization, we consider the unscented transform and develop I-UKF. Our I-UKF estimates an adversary's inference, who also deploys a forward UKF. Similar to the inverse cognition framework investigated in \cite{krishnamurthy2019how,mattila2020hmm,singh2022inverse_part1}, we assume perfect system model information, i.e., the attacker's forward filter is known to the defender. The inverse filter is formulated using augmented states to take into account the non-additive process noise terms in the forward filter's state estimate evolution. Our numerical experiments show that I-UKF provides reasonably accurate estimates even when the defender's forward UKF assumption is incorrect. We remark that the I-UKF is different from the \textit{inversion of UKF} \cite{zhengyu2021iterated}, which estimates the input based on the output. Clearly, such an inversion of UKF may not take the same mathematical form because the UKF is employed on the adversary's side. Hence, this formulation is unrelated to our inverse cognition problem. Note that our proposed I-UKF does not follow trivially from I-KF\cite{krishnamurthy2019how} or I-EKF\cite{singh2022inverse_part1}; see also Remark~\ref{remark:I-UKF with I-KF and I-EKF}.\\
\textbf{2) Generalizations of I-UKF.} In practice, the systems often involve continuous-time state evolution or are complex-valued such that suitable continuous-discrete or complex filters are required. Hence, we generalize our I-UKF theory to obtain continuous-discrete and complex I-UKFs for inverse filtering applications. We also consider maximum correntropy criterion (MCC)-based I-UKF to tackle non-Gaussian noises.\\
\textbf{3) RKHS-UKF.} When the defender lacks prior knowledge of the state-transition and observation models, we propose RKHS-UKF for the defender. Further, it may be employed by both attacker and defender to infer the defender's state and attacker's state estimate, respectively. While RKHS-UKF adopts the online approximate EM of RKHS-EKF\cite{singh2022inverse_part2} to learn the system parameters, the expectations under non-linear transformations are computed using the unscented transform, thereby avoiding Jacobian computations. Our numerical experiments show that RKHS-UKF outperforms RKHS-EKF in terms of estimation accuracy.\\
\textbf{4) Stability of I-UKF and RKHS-UKF.} In general, stability and convergence results are difficult to obtain for non-linear KFs. A bounded non-linearity approach was employed in \cite{reif1999stochastic} to prove EKF's stability in the exponential mean-squared boundedness sense, but for bounded initial estimation error. The unknown matrix approach introduced in \cite{xiong2006performance_ukf} for UKF's stability relaxed the bound on the initial error by introducing unknown instrumental matrices to model the linearization errors but the measurements were still linear. Besides providing sufficient conditions for error boundedness, this approach also rigorously justifies enlarging the noise covariance matrices to stabilize the filter. In this paper, we provide the stability conditions for I-UKF and RKHS-UKF based on the unknown matrix approach. In the process, we also obtain hitherto unreported general stability results for forward UKF. We then prove that I-UKF's recursive estimates are conservative, i.e., the true error covariance is upper bounded by the I-UKF's estimate of the same. We validate the estimation performance of all inverse filters through extensive numerical experiments with \textit{recursive} Cram\'{e}r-Rao lower bound (RCRLB) \cite{tichavsky1998posterior} and non-credibility index (NCI) as the performance metrics.

The rest of the paper is organized as follows. The next section describes the system model for inverse cognition problem
, whereas Section~\ref{sec:IUKF} presents I-UKFs for systems with prior model information. We address the unknown system model case through RKHS-UKF in Section~\ref{sec:RKHS-UKF}. We provide the performance guarantees in Section~\ref{sec:stability} and demonstrate the proposed filters' performance via numerical experiments in Section~\ref{sec:numerical}. We conclude in Section~\ref{sec:summary}.

Throughout the paper, we reserve boldface lowercase and uppercase letters for vectors (column vectors) and matrices, respectively, and $\lbrace a_{i}\rbrace_{i_{1}\leq i\leq i_{2}}$ denotes a set of elements indexed by an integer $i$. The notation $[\mathbf{a}]_{i}$ is used to denote the $i$-th component of vector $\mathbf{a}$ and $[\mathbf{A}]_{i,j}$ denotes the $(i,j)$-th component of matrix $\mathbf{A}$, with $[\mathbf{A}]_{(i,:)}$ and $[\mathbf{A}]_{(:,j)}$, respectively, denoting the $i$-th row and $j$-th column of the matrix. Also, $[\mathbf{A}]_{(i_{1}:i_{2},j_{1}:j_{2})}$ represents the sub-matrix of $\mathbf{A}$ consisting of rows $i_{1}$ to $i_{2}$ and columns $j_{1}$ to $j_{2}$ while $[\mathbf{a}]_{i_{1}:i_{2}}$ denotes the corresponding sub-vector. The transpose/Hermitian operation is $(\cdot)^{T/H}$; the $l_{2}$ norm and norm with respect to matrix $\mathbf{A}$ of a vector are $\|\cdot\|_{2}$ and $\|\cdot\|_{\mathbf{A}}$, respectively; and the notation $\textrm{det}(\mathbf{A})$, $\textrm{Tr}(\mathbf{A})$ and $\|\mathbf{A}\|$, respectively, denote the determinant, trace and spectral norm of $\mathbf{A}$. For matrices $\mathbf{A}$ and $\mathbf{B}$, the inequality $\mathbf{A}\preceq\mathbf{B}$ means that $\mathbf{B}-\mathbf{A}$ is a positive semidefinite (p.s.d.) matrix. For a function $f:\mathbb{R}^{n}\rightarrow\mathbb{R}^{m}$, $\frac{\partial f}{\partial\mathbf{x}}$ denotes the $\mathbb{R}^{m\times n}$ Jacobian matrix with $\mathbb{R}^{m\times n}$ denoting the set of all real-valued $m\times n$ matrices, while for function $f:\mathbb{R}^{n}\rightarrow\mathbb{R}$, it denotes the $\mathbb{R}^{n\times 1}$ gradient vector with respect to vector $\mathbf{x}$. Also, $\mathbf{I}_{n}$ and $\mathbf{0}_{n\times m}$ denote a `$n\times n$' identity matrix and a `$n\times m$' all zero matrix, respectively. The Gaussian random variable is represented as $\mathbf{x} \sim \mathcal{N}(\boldsymbol{\mu},\mathbf{Q})$ with mean  $\boldsymbol{\mu}$ and covariance matrix $\mathbf{Q}$ while the covariance of random variable $\mathbf{x}$ is denoted by $\textrm{Cov}(\mathbf{x})$. We denote the Cholesky decomposition of matrix $\mathbf{A}$ as $\mathbf{A}=\sqrt{\mathbf{A}}\sqrt{\mathbf{A}}^{T}$.

\section{System Model}
\label{sec:sys_mod}
Consider a discrete-time stochastic dynamical system as the defender's state evolution process $\{\mathbf{x}_{k}\}_{k\geq 0}$, where $\mathbf{x}_{k}\in\mathbb{R}^{n_{x}\times 1}$ is the defender's state at the $k$-th time instant. The defender's state, perfectly known to the defender, evolves as
\par\noindent\small
\begin{align}
    \mathbf{x}_{k+1}=f(\mathbf{x}_{k})+\mathbf{w}_{k},\label{eqn:state transition x}
\end{align}
\normalsize
where process noise $\mathbf{w}_{k}\sim\mathcal{N}(\mathbf{0}_{n_{x}\times 1},\mathbf{Q})$ with covariance matrix $\mathbf{Q}\in\mathbb{R}^{n_{x}\times n_{x}}$. At the $k$-th time instant, the attacker observes defender's state as $\mathbf{y}_{k}\in\mathbb{R}^{n_{y}\times 1}$ given by
\par\noindent\small
\begin{align}
    \mathbf{y}_{k}=h(\mathbf{x}_{k})+\mathbf{v}_{k},\label{eqn:observation y}
\end{align}
\normalsize
where $\mathbf{v}_{k}\sim\mathcal{N}(\mathbf{0}_{n_{y}\times 1},\mathbf{R})$ is the attacker's measurement noise with covariance matrix $\mathbf{R}\in\mathbb{R}^{n_{y}\times n_{y}}$. The attacker computes an estimate $\hat{\mathbf{x}}_{k}$ of the defender's state $\mathbf{x}_{k}$ given its observations $\{\mathbf{y}_{j}\}_{1\leq j\leq k}$ using a (forward) stochastic filter. The attacker then takes an action whose noisy observation by the defender is $\mathbf{a}_{k}\in\mathbb{R}^{n_{a}\times 1}$ as
\par\noindent\small
\begin{align}
    \mathbf{a}_{k}=g(\hat{\mathbf{x}}_{k})+\bm{\epsilon}_{k},\label{eqn:observation a}
\end{align}
\normalsize
where $\bm{\epsilon}_{k}\sim\mathcal{N}(\mathbf{0}_{n_{a}\times 1},\bm{\Sigma}_{\epsilon})$ is the defender's measurement noise with covariance matrix $\bm{\Sigma}_{\epsilon}\in\mathbb{R}^{n_{a}\times n_{a}}$. In this context, the function $g(\cdot)$ represents the combined effect of the attacker's action strategy and the defender's observation. For example, inspired by linear quadratic Gaussian (LQG) control problems, I-KF\cite{krishnamurthy2019how} selects $\mathbf{a}_{k}$ based on a linear relationship with $\hat{\mathbf{x}}_{k}$, adjusted by its estimated error covariance matrix $\bm{\Sigma}_{k}$. However, we consider a general non-linear observation as defined in \eqref{eqn:observation a}. Finally, the defender uses $\{\mathbf{a}_{j},\mathbf{x}_{j}\}_{1\leq j\leq k}$ to compute an estimate $\doublehat{\mathbf{x}}_{k}$ of the attacker's estimate $\hat{\mathbf{x}}_{k}$ using the inverse stochastic filter. The noise processes $\{\mathbf{w}_{k}\}_{k\geq 0}$, $\{\mathbf{v}_{k}\}_{k\geq 1}$ and $\{\bm{\epsilon}_{k}\}_{k\geq 1}$ are mutually independent and identically distributed across time. To maintain simplicity, we do not assume time-varying noise covariances, while $f(\cdot)$, $h(\cdot)$, and $g(\cdot)$ are chosen as appropriate non-linear functions. In the following, we operate under the assumption that both the attacker and defender possess perfect knowledge of these functions and the noise distributions. Later, we examine the problem without this assumption of perfect system knowledge in Section~\ref{sec:RKHS-UKF}. Furthermore, the attacker is unaware that the defender is observing the former. The case when the attacker deliberately changes its actions to guard against the defender requires an inverse-inverse reinforcement learning-based representation of the problem, which has been recently addressed in \cite{pattanayak2022inverse,pattanayak2022meta}.

\section{Inverse UKF}\label{sec:IUKF}
The UKF generates a set of `$2n_{x}+1$' sigma points deterministically from the previous state estimate including the previous estimate itself as one of the sigma points. The sigma points are then propagated through the non-linear system model, and the state estimates are obtained as a weighted sum of these propagated points. In I-UKF, we assume that the attacker is employing a forward UKF to compute its estimate $\hat{\mathbf{x}}_{k}$ with known state transition \eqref{eqn:state transition x} and observation \eqref{eqn:observation y}. The I-UKF then infers the estimate $\doublehat{\mathbf{x}}_{k}$ of $\hat{\mathbf{x}}_{k}$ using observation \eqref{eqn:observation a}.

\vspace{-8pt}
\subsection{Forward UKF}
\label{subsec:forward UKF}
Consider the scaling parameter $\kappa\in\mathbb{R}$ controlling the spread of the forward UKF's sigma points around the previous estimate. The sigma points $\{\widetilde{\mathbf{x}}_{i}\}_{0\leq i\leq 2n_{x}}$ are generated from state estimate $\hat{\mathbf{x}}$ and its error covariance matrix $\bm{\Sigma}$ as
\par\noindent\small
\begin{align}
  &\{\widetilde{\mathbf{x}}_{i}\}_{0\leq i\leq 2n_{x}}=S_{gen}(\hat{\mathbf{x}},\bm{\Sigma})\nonumber\\
  &=\begin{cases}
  \hat{\mathbf{x}},\;\;\;\;\;\;\;i=0,\\
 \hat{\mathbf{x}}+\left[\sqrt{(n_{x}+\kappa)\bm{\Sigma}}\right]_{(:,i)},\;i=1,2,\hdots,n_{x}\\
    \hat{\mathbf{x}}-\left[\sqrt{(n_{x}+\kappa)\bm{\Sigma}}\right]_{(:,i-n_{x})},\;i=n_{x}+1,n_{x}+2,\hdots,2n_{x}\end{cases},\label{eqn:sigma points generation}
\end{align}
\normalsize
with their weights $\omega_{i}=\begin{cases}\frac{\kappa}{n_{x}+\kappa} & i=0\\\frac{1}{2(n_{x}+\kappa)} & i=1,2,\hdots,2n_{x}\end{cases}$.

Denote the sigma points generated and propagated for the time update at $k$-th time instant by $\lbrace\mathbf{s}_{i,k}\rbrace$ and $\lbrace\mathbf{s}^{*}_{i,k+1|k}\rbrace$, respectively. Similarly, $\lbrace\mathbf{q}_{i,k+1|k}\rbrace$ and $\lbrace\mathbf{q}^{*}_{i,k+1|k}\rbrace$ are the sigma points, respectively, generated and propagated to predict observation $\mathbf{y}_{k+1}$ as $\hat{\mathbf{y}}_{k+1|k}$. The attacker's forward UKF recursions to compute state estimate $\hat{\mathbf{x}}_{k+1}$ and the associated error covariance matrix estimate $\bm{\Sigma}_{k+1}$ are \cite{simon2006optimal}
\par\noindent\small
\begin{align}
    &\textrm{Time update:}\;\;\lbrace\mathbf{s}_{i,k}\rbrace_{0\leq i\leq 2n_{x}}=S_{gen}(\hat{\mathbf{x}}_{k},\bm{\Sigma}_{k}),\label{eqn:forward ukf prediction sigma points}\\
    &\mathbf{s}^{*}_{i,k+1|k}=f(\mathbf{s}_{i,k})\;\;\;\forall i=0,1,\hdots,2n_{x},\nonumber\\
    &\hat{\mathbf{x}}_{k+1|k}=\sum_{i=0}^{2n_{x}}\omega_{i}\mathbf{s}^{*}_{i,k+1|k},\label{eqn:forward ukf x predict}\\
    &\bm{\Sigma}_{k+1|k}=\sum_{i=0}^{2n_{x}}\omega_{i}\mathbf{s}^{*}_{i,k+1|k}(\mathbf{s}^{*}_{i,k+1|k})^{T}-\hat{\mathbf{x}}_{k+1|k}\hat{\mathbf{x}}_{k+1|k}^{T}+\mathbf{Q},\nonumber\\
    &\textrm{Measurement update:}\;\;\lbrace\mathbf{q}_{i,k+1|k}\rbrace_{0\leq i\leq 2n_{x}}=S_{gen}(\hat{\mathbf{x}}_{k+1|k},\bm{\Sigma}_{k+1|k}),\label{eqn:forward UKF update sigma points}\\
    &\mathbf{q}^{*}_{i,k+1|k}=h(\mathbf{q}_{i,k+1|k})\;\;\;\forall i=0,1,\hdots,2n_{x},\nonumber\\
    &\hat{\mathbf{y}}_{k+1|k}=\sum_{i=0}^{2n_{x}}\omega_{i}\mathbf{q}^{*}_{i,k+1|k},\label{eqn:forward ukf y predict}\\
    &\bm{\Sigma}^{y}_{k+1}=\sum_{i=0}^{2n_{x}}\omega_{i}\mathbf{q}^{*}_{i,k+1|k}(\mathbf{q}^{*}_{i,k+1|k})^{T}-\hat{\mathbf{y}}_{k+1|k}\hat{\mathbf{y}}_{k+1|k}^{T}+\mathbf{R},\nonumber\\
    &\bm{\Sigma}^{xy}_{k+1}=\sum_{i=0}^{2n_{x}}\omega_{i}\mathbf{q}_{i,k+1|k}(\mathbf{q}^{*}_{i,k+1|k})^{T}-\hat{\mathbf{x}}_{k+1|k}\hat{\mathbf{y}}_{k+1|k}^{T}\nonumber,\\
    &\hat{\mathbf{x}}_{k+1}=\hat{\mathbf{x}}_{k+1|k}+\mathbf{K}_{k+1}(\mathbf{y}_{k+1}-\hat{\mathbf{y}}_{k+1|k}),\label{eqn:forward ukf x update}\\
    &\bm{\Sigma}_{k+1}=\bm{\Sigma}_{k+1|k}-\mathbf{K}_{k+1}\bm{\Sigma}^{y}_{k+1}\mathbf{K}_{k+1}^{T},\label{eqn:forward UKF sigma update}
\end{align}
\normalsize
where gain matrix $\mathbf{K}_{k+1}=\bm{\Sigma}^{xy}_{k+1}\left(\bm{\Sigma}^{y}_{k+1}\right)^{-1}$. The step to generate the second set of sigma points may be omitted, and $\lbrace\mathbf{q}^{*}_{i,k+1|k}\rbrace$ may be obtained by propagating $\lbrace\mathbf{s}^{*}_{i,k+1|k}\rbrace$ through the observation function $h(\cdot)$ to save computational efforts. However, this may degrade the performance of the classical UKF because the effect of process noise here is unaccounted for. On the other hand, pure-propagation UKF \cite{straka2014design} generates a modified sigma-point set with increased covariance only once without compromising the performance.

\vspace{-8pt}
\subsection{I-UKF}
\label{subsec:IUKF}
Under the known forward UKF assumption, substituting \eqref{eqn:observation y}, \eqref{eqn:forward ukf x predict}, and \eqref{eqn:forward ukf y predict} in \eqref{eqn:forward ukf x update}, yields the inverse filter's state transition as
\par\noindent\small
\begin{align}
    \hat{\mathbf{x}}_{k+1}&=\sum_{i=0}^{2n_{x}}\omega_{i}\left(\mathbf{s}^{*}_{i,k+1|k}-\mathbf{K}_{k+1}\mathbf{q}^{*}_{i,k+1|k}\right)+\mathbf{K}_{k+1}h(\mathbf{x}_{k+1})\nonumber\\
    &\;\;+\mathbf{K}_{k+1}\mathbf{v}_{k+1}.\label{eqn: IUKF state transition detail}
\end{align}
\normalsize
In this state transition, $\mathbf{x}_{k+1}$ is a known exogenous input, while $\mathbf{v}_{k+1}$ represents the process noise involved. Since the functions $f(\cdot)$ and $h(\cdot)$ are known, the propagated sigma points $\{\mathbf{s}^{*}_{i,k+1|k}\}$ and $\{\mathbf{q}^{*}_{i,k+1|k}\}$, and the gain matrix $\mathbf{K}_{k+1}$ are functions of the first set of sigma points $\{\mathbf{s}_{i,k}\}$. These sigma points, in turn, are obtained deterministically from the previous state estimate $\hat{\mathbf{x}}_{k}$ and covariance matrix $\bm{\Sigma}_{k}$ using \eqref{eqn:forward ukf prediction sigma points}. Hence, I-UKF's state transition, \emph{under the assumption that parameter $\kappa$ is known to the defender}, is
\par\noindent\small
\begin{align}
    \hat{\mathbf{x}}_{k+1}=\widetilde{f}(\hat{\mathbf{x}}_{k},\bm{\Sigma}_{k},\mathbf{x}_{k+1},\mathbf{v}_{k+1}).\label{eqn:inverse ukf state transition}
\end{align}
\normalsize
Note that the process noise $\mathbf{v}_{k+1}$ is non-additive because $\mathbf{K}_{k+1}$ depends on the previous estimates. Furthermore, $\bm{\Sigma}_{k}$ does not depend on the current forward filter's observation $\mathbf{y}_{k}$ but is evaluated recursively using the previous estimate. We approximate $\bm{\Sigma}_{k}$ as $\bm{\Sigma}_{k}^{*}$ by computing the covariance matrix using its own previous estimate, i.e. $\doublehat{\mathbf{x}}_{k}$, in the same recursive manner as the forward filter estimates at its estimate $\hat{\mathbf{x}}_{k}$ (using \eqref{eqn:forward UKF sigma update} after computing gain matrix $\mathbf{K}_{k+1}$ from the generated sigma points). Hence, the inverse filter treats $\bm{\Sigma}_{k}$ as a known exogenous input in the state transition \eqref{eqn:inverse ukf state transition}.

Since the state transition \eqref{eqn:inverse ukf state transition} involves non-additive noise term, we consider an augmented state vector $\mathbf{z}_{k}=[\hat{\mathbf{x}}_{k}^{T},\mathbf{v}_{k+1}^{T}]^{T}$ of dimension $n_{z}=n_{x}+n_{y}$ for I-UKF formulation such that the state transition \eqref{eqn:inverse ukf state transition} becomes $\hat{\mathbf{x}}_{k+1}=\widetilde{f}(\mathbf{z}_{k},\bm{\Sigma}_{k},\mathbf{x}_{k+1})$. Denote
\par\noindent\small
\begin{align}
    \hat{\mathbf{z}}_{k}=[\doublehat{\mathbf{x}}_{k}^{T},\mathbf{0}_{1\times n_{y}}]^{T},\;\;\overline{\bm{\Sigma}}^{z}_{k}=\begin{bsmallmatrix}\overline{\bm{\Sigma}}_{k}&\mathbf{0}_{n_{x}\times n_{y}}\\\mathbf{0}_{n_{y}\times n_{x}}&\mathbf{R}\end{bsmallmatrix}.\label{eqn:IUKF z hat and sigma z}
\end{align}
\normalsize
Considering the I-UKF's scaling parameter as $\overline{\kappa}\in\mathbb{R}$, the sigma points $\{\overline{\mathbf{s}}_{j,k}\}_{0\leq j\leq 2n_{z}}$ are generated from $\hat{\mathbf{z}}_{k}$ and $\overline{\bm{\Sigma}}^{z}_{k}$ similar to \eqref{eqn:sigma points generation} with weights $\overline{\omega}_{j}$. I-UKF then computes $\doublehat{\mathbf{x}}_{k}$ and its associated error covariance matrix $\overline{\bm{\Sigma}}_{k}$ recursively as
\par\noindent\small
\begin{align}
    &\textit{Time update:}\nonumber\\
    &\overline{s}^{*}_{j,k+1|k}=\widetilde{f}(\overline{\mathbf{s}}_{j,k},\bm{\Sigma}_{k}^{*},\mathbf{x}_{k+1})\;\;\;\forall j=0,1,\hdots,2n_{z},\label{eqn:IUKF f propagate}\\   
    &\doublehat{\mathbf{x}}_{k+1|k}=\sum_{j=0}^{2n_{z}}\overline{\omega}_{j}\overline{\mathbf{s}}^{*}_{j,k+1|k},\label{eqn:IUKF x predict}\\
    &\overline{\bm{\Sigma}}_{k+1|k}=\sum_{j=0}^{2n_{z}}\overline{\omega}_{j}\overline{\mathbf{s}}^{*}_{j,k+1|k}(\overline{\mathbf{s}}^{*}_{j,k+1|k})^{T}-\doublehat{\mathbf{x}}_{k+1|k}\doublehat{\mathbf{x}}_{k+1|k}^{T},\label{eqn:IUKF sig predict}\\
    &\textit{Measurement update:}\nonumber\\
    &\mathbf{a}^{*}_{j,k+1|k}=g(\overline{\mathbf{s}}^{*}_{j,k+1|k})\;\;\;\forall j=0,1,\hdots,2n_{z},\label{eqn:IUKF g propagate}\\
    &\hat{\mathbf{a}}_{k+1|k}=\sum_{j=0}^{2n_{z}}\overline{\omega}_{j}\mathbf{a}^{*}_{j,k+1|k},\label{eqn:IUKF a predict}\\
    &\overline{\bm{\Sigma}}^{a}_{k+1}=\sum_{j=0}^{2n_{z}}\overline{\omega}_{j}\mathbf{a}^{*}_{j,k+1|k}(\mathbf{a}^{*}_{j,k+1|k})^{T}-\hat{\mathbf{a}}_{k+1|k}\hat{\mathbf{a}}_{k+1|k}^{T}+\bm{\Sigma}_{\epsilon},\label{eqn:IUKF sig a}\\
    &\overline{\bm{\Sigma}}^{xa}_{k+1}=\sum_{j=0}^{2n_{z}}\overline{\omega}_{j}\overline{\mathbf{s}}^{*}_{j,k+1|k}(\mathbf{a}^{*}_{j,k+1|k})^{T}-\doublehat{\mathbf{x}}_{k+1|k}\hat{\mathbf{a}}_{k+1|k}^{T},\label{eqn:IUKF cross cov}\\
     &\overline{\mathbf{K}}_{k+1}=\overline{\bm{\Sigma}}^{xa}_{k+1}\left(\overline{\bm{\Sigma}}^{a}_{k+1}\right)^{-1},\label{eqn:IUKF gain}\\
    &\doublehat{x}_{k+1}=\doublehat{x}_{k+1|k}+\overline{\mathbf{K}}_{k+1}(\mathbf{a}_{k+1}-\hat{\mathbf{a}}_{k+1|k}),\label{eqn:IUKF state update}\\
    &\overline{\bm{\Sigma}}_{k+1}=\overline{\bm{\Sigma}}_{k+1|k}-\overline{\mathbf{K}}_{k+1}\overline{\bm{\Sigma}}^{a}_{k+1}\overline{\mathbf{K}}_{k+1}^{T}.\label{eqn:IUKF covariance update}
\end{align}
\normalsize
\begin{figure}
  \centering
  \includegraphics[width = 1.0\columnwidth]{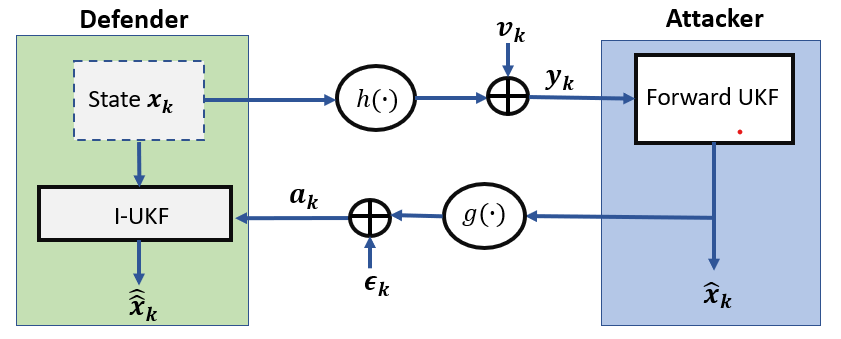}
  \caption{An illustration of I-UKF's recursion at $k$-th time step. The defender's true state at $k$-th time is $\mathbf{x}_{k}$, which the attacker observes as $\mathbf{y}_{k}$ through observation function $h(\cdot)$ with additive measurement noise $\mathbf{v}_{k}$. Forward UKF provides estimate $\hat{\mathbf{x}}_{k}$ of $\mathbf{x}_{k}$ using $\mathbf{y}_{k}$. The defender observes $\hat{\mathbf{x}}_{k}$ as $\mathbf{a}_{k}$ through observation function $g(\cdot)$ with additive measurement noise $\bm{\epsilon}_{k}$. Finally, with $\mathbf{a}_{k}$ and $\mathbf{x}_{k}$ as inputs, I-UKF computes estimate $\doublehat{\mathbf{x}}_{k}$ of $\hat{\mathbf{x}}_{k}$.}
 \label{fig:I-UKF's schematic}
\end{figure}

Fig.~\ref{fig:I-UKF's schematic} illustrates the I-UKF's recursions at $k$-th time step. These recursions follow from UKF's non-additive noise formulation\cite{wan2000unscented} with the sigma points generated in higher ($n_{z}=n_{x}+n_{y}$) dimensional state space than the $n_{x}$ dimensions in forward UKF. However, the latter requires a new set for the measurement update, while I-UKF generates these points only once. The defender's assumed observation \eqref{eqn:observation a} depends only on the state estimate $\hat{\mathbf{x}}_{k}$. However, in many applications, the attacker may also consider the estimated covariance $\bm{\Sigma}_{k}$ of $\hat{\mathbf{x}}_{k}$ in deciding its actions. In such cases, our proposed I-UKF can be trivially modified to include $\bm{\Sigma}_{k}$ in observations \eqref{eqn:observation a}. In particular, I-UKF's state transition \eqref{eqn:inverse ukf state transition} depends on forward UKF's $\bm{\Sigma}_{k}$ and hence, we compute the approximation $\bm{\Sigma}^{*}_{k}$. If observation \eqref{eqn:observation a} depends on $\bm{\Sigma}_{k}$, I-UKF's update step \eqref{eqn:IUKF g propagate} also uses $\bm{\Sigma}^{*}_{k}$. As mentioned earlier in Section~\ref{subsec:IUKF}, I-UKF computes $\bm{\Sigma}^{*}_{k}$ using its estimate $\doublehat{\mathbf{x}}_{k}$, in the same manner as the forward UKF computes $\bm{\Sigma}_{k}$ from its estimate $\hat{\mathbf{x}}_{k}$. In particular, we obtain $\mathbf{K}_{k+1}$ (and intermediately, $\bm{\Sigma}_{k+1|k}$ and $\bm{\Sigma}^{y}_{k+1}$) while propagating sigma-point $\overline{\mathbf{s}}_{j,k}$ through state transition \eqref{eqn:inverse ukf state transition} (equivalently, \eqref{eqn: IUKF state transition detail}). Hence, for each $\overline{\mathbf{s}}_{j,k}$, we obtain an estimate $\widetilde{\bm{\Sigma}}_{j,k+1}$ of forward UKF's covariance update $\bm{\Sigma}_{k+1}$ using \eqref{eqn:forward UKF sigma update}. We choose the average of these $\{\widetilde{\bm{\Sigma}}_{j,k+1}\}_{0\leq j\leq 2n_{z}}$ as the covariance approximation $\bm{\Sigma}^{*}_{k+1}$ for the next I-UKF recursion. 
 
\begin{remark}[Unknown $\kappa$]\label{remark:unknown kappa}
In the I-UKF formulation, we assumed that the parameter $\kappa$ of the forward UKF is known. In general, assuming a $\kappa$ different from the attacker's actual $\kappa$ may increase estimation errors in I-UKF. However, our numerical experiments in Section~\ref{sec:numerical} show that the I-UKF provides reasonable estimates even when assuming a different $\kappa$ from its true value. Further, the I-UKF's stability (Theorem~\ref{theorem:IUKF stability}) requires a stable forward UKF and is independent of the assumed $\kappa$. Note that the choice of forward UKF's $\kappa$ by the inverse filter is independent of its own control parameter $\overline{\kappa}$. Besides UKF, several numerical integration techniques-based SPKFs have also been developed in the literature. For example, cubature KF (CKF) \cite{arasaratnam2009cubature} and quadrature KF (QKF) \cite{ito2000gaussian,arasaratnam2007qkf} consider, respectively, the cubature and Gauss-Hermite quadrature rules. Our recent work \cite{singh2023inverse_ckf_qkf} proposed and analyzed these formulations for inverse CKF and QKF.
\end{remark}

\begin{remark}[Differences from I-KF and I-EKF]\label{remark:I-UKF with I-KF and I-EKF}
   Unlike I-KF \cite{krishnamurthy2019how} and I-EKF \cite{singh2022inverse_part1}, the forward gain matrix $\mathbf{K}_{k+1}$ is not treated as a time-varying parameter of I-UKF's state transition \eqref{eqn: IUKF state transition detail}. In KF, the gain matrix is fully deterministic given the model parameters and the initial covariance estimate $\bm{\Sigma}_{0}$. In EKF, it depends on the linearized model functions at the state estimates. However, the UKF gain matrix is computed from the covariance matrix estimates obtained as a weighted average of the generated sigma points, which are explicit functions of the state estimates. This prevents I-UKF from treating $\mathbf{K}_{k+1}$ as a parameter of \eqref{eqn: IUKF state transition detail}. 
\end{remark}
\begin{remark}[Non-Gaussian noise]\label{remark:non-Gaussian}
    UKF and hence, I-UKF assume Gaussian process and measurement noises. Recently, MCC-modified KFs have been developed to handle non-Gaussian noises\cite{izanloo2016kalman,liu2017maximum,wang2017maximum}. Our I-UKF can be trivially modified based on MCC for non-Gaussian system models. For instance, forward MCC-UKF proposed in \cite{wang2017maximum} introduces a scalar $l_{k+1}=G_{\sigma}(\|\mathbf{y}_{k+1}-\hat{\mathbf{y}}_{k+1|k}\|_{\widetilde{\mathbf{R}}_{k+1}^{-1}})$ with $G_{\sigma}(\cdot)$ as the Gaussian kernel and computes the gain matrix as $\mathbf{K}_{k+1}=\bm{\Sigma}_{k+1|k}l_{k+1}\widetilde{\mathbf{H}}_{k+1}^{T}(\widetilde{\mathbf{R}}_{k+1}+\widetilde{\mathbf{H}}_{k+1}\bm{\Sigma}_{k+1|k}l_{k+1}\widetilde{\mathbf{H}}_{k+1}^{T})^{-1}$. Here, $\widetilde{\mathbf{H}}_{k+1}$ and $\widetilde{\mathbf{R}}_{k+1}$, respectively, are the pseudo-measurement matrix and modified covariance matrix. The estimated error covariance then becomes $\bm{\Sigma}_{k+1}=(\mathbf{I}-\mathbf{K}_{k+1}\widetilde{\mathbf{H}}_{k+1})\bm{\Sigma}_{k+1|k}(\mathbf{I}-\mathbf{K}_{k+1}\widetilde{\mathbf{H}}_{k+1})^{T}+\mathbf{K}_{k+1}\widetilde{\mathbf{R}}_{k+1}\mathbf{K}_{k+1}^{T}$ while all other state prediction and update steps remain same as in forward UKF. These modifications need to be taken into account in the inverse filter's state transition equation while formulating the inverse filter. I-UKF's gain matrix $\overline{\mathbf{K}}_{k+1}$ and covariance estimate $\overline{\bm{\Sigma}}_{k+1}$ are also similarly modified using scalar $\overline{l}_{k+1}$ which is the counterpart of $l_{k+1}$ for the inverse filter's dynamics.
\end{remark}

\subsection{Continuous-time state evolution}\label{subsec:continuous-time}
The state-evolution \eqref{eqn:state transition x} and observations \eqref{eqn:observation y} represent discrete-time processes. In many practical applications, the defender's state evolves as a continuous-time process while the attacker observes the state at discrete-time instants. In such cases, (forward) continuous-discrete Kalman-Bucy filter \cite{kalman1961new} and its non-linear extensions \cite{jazwinski2007stochastic,kulikov2015accurate,sarkka2007unscented} are often employed for efficient state estimation. On the contrary, the inverse filtering problem still remains a discrete-time problem provided that the defender observes the attacker's actions (as $\mathbf{a}_{k}$) and estimates $\hat{\mathbf{x}}_{k}$ at the same discrete-time instants only. Our I-UKF can handle the continuous-time state evolution case with some trivial modifications. In particular, the forward continuous-discrete UKF's time update numerically integrates a pair of differential equations \cite[eq.~(34)]{sarkka2007unscented} to compute $\hat{\mathbf{x}}_{k+1|k}$ and $\bm{\Sigma}_{k+1|k}$ using estimates $\hat{\mathbf{x}}_{k}$ and $\bm{\Sigma}_{k}$ as the initial conditions. The measurement update steps are the same as in the forward UKF. I-UKF's state transition is modified to account for these differences. 

Denote the forward continuous-discrete UKF's time update (solutions of differential equations) as $\hat{\mathbf{x}}_{k+1|k}=\chi_{1}(\hat{\mathbf{x}}_{k})$ and $\bm{\Sigma}_{k+1|k}=\chi_{2}(\hat{\mathbf{x}}_{k},\bm{\Sigma}_{k})$. I-UKF's state transition \eqref{eqn: IUKF state transition detail} becomes
\par\noindent\small
\begin{align}
        \hat{\mathbf{x}}_{k+1}&=\chi_{1}(\hat{\mathbf{x}}_{1})-\sum_{i=0}^{2n_{x}}\omega_{i}\mathbf{K}_{k+1}\mathbf{q}^{*}_{i,k+1|k}+\mathbf{K}_{k+1}h(\mathbf{x}_{k+1})\nonumber\\
        &\;\;\;+\mathbf{K}_{k+1}\mathbf{v}_{k+1}.\label{eqn:state transition continuous discrete}
\end{align}
\normalsize
Here, the propagated points $\{\mathbf{q}^{*}_{i,k+1|k}\}$ are again obtained deterministically from the predicted state $\hat{\mathbf{x}}_{k+1|k}$ and covariance estimate $\bm{\Sigma}_{k+1|k}$, which in turn are functions of $\hat{\mathbf{x}}_{k}$ and $\bm{\Sigma}_{k}$ via solutions $\chi_{1}(\cdot)$ and $\chi_{2}(\cdot)$. Hence, \eqref{eqn:state transition continuous discrete} simply becomes $\hat{\mathbf{x}}_{k+1}=\widetilde{f}(\hat{\mathbf{x}}_{k},\bm{\Sigma}_{k},\mathbf{x}_{k+1},\mathbf{v}_{k+1})$ with $\widetilde{f}(\cdot)$ now denoting the modified state transition function.

\subsection{Complex-valued systems}\label{subsec:complex}
In many applications like frequency estimation and neural network training, the state $\mathbf{x}_{k}$ and its observations are not real but complex-valued such that complex KFs are employed\cite{dini2011widely,dini2012class}. While the simplest complex KFs assume second-order circularity and only use covariance matrix information, the recent widely linear filters\cite{dini2011widely,dini2012class,zhang2022unscented} consider general non-circular cases and enhance their accuracy using the pseudo-covariance matrix. Our I-UKF can be appropriately modified to obtain widely linear complex I-UKFs for complex-valued inverse filtering problems.

Consider the defender's state $\mathbf{x}_{k}\in\mathbb{C}^{n_{x}\times 1}$ with $\bm{\Sigma}_{k}$ and $\bm{\Sigma}^{p}_{k}$ denoting its covariance and pseudo-covariance matrices, respectively. In \cite{dini2011widely}, the forward widely linear complex UKF defines an augmented state $\bm{\xi}_{k}\doteq[\mathbf{x}_{k}^{T},\mathbf{x}^{H}_{k}]^{T}$ and its covariance matrix $\bm{\Sigma}^{\xi}_{k}\doteq\begin{bsmallmatrix}
    \bm{\Sigma}_{k}&\bm{\Sigma}^{p}_{k}\\
    (\bm{\Sigma}^{p}_{k})^{H}&\bm{\Sigma}_{k}^{H}
\end{bsmallmatrix}$. The forward filter recursions to estimate $\bm{\xi}_{k}$ then follow from the standard UKF with $(\cdot)^{T}$ replaced by $(\cdot)^{H}$. Note that the sigma points are then generated using estimate $\hat{\bm{\xi}}_{k}$ (of $\bm{\xi}_{k}$) and $\bm{\Sigma}^{\xi}_{k}$ which includes pseudo-covariance $\bm{\Sigma}^{p}_{k}$. While formulating the inverse filter, we need to consider the forward filter's augmented state $\bm{\xi}_{k}$ and modify the state transition \eqref{eqn: IUKF state transition detail}. Finally, \textit{mutatis mutandis}, the general complex I-UKF's recursions follow from the augmented complex UKF of \cite{dini2011widely} treating the I-UKF's modified state transition as the state evolution process and $\mathbf{a}_{k}$ as observations. On the other hand, the (forward) complex UKF proposed in \cite{zhang2022unscented} introduces modified sigma points and state updates using both innovation and its conjugate. These changes are similarly accommodated in the inverse filter's formulation to obtain an alternative complex I-UKF.

\section{RKHS-UKF}
\label{sec:RKHS-UKF}
In the previous section, we assumed perfect system information on both the attacker's and defender's sides to formulate the inverse filter. However, in many applications, the agent (attacker and/or defender) employing the stochastic filter may lack information about state evolution, observation, or both. Here, we develop RKHS-UKF to jointly estimate the desired state and learn the unknown system model. To this end, we consider an RKHS-based function approximation to represent the unknown non-linear functions\cite{aronszajn1950theory}. From the representer theorem \cite{scholkopf2001generalized}, the optimal approximation of a function $s(\cdot):\mathbb{R}^{n}\to\mathbb{R}$ in the RKHS induced by a kernel $K(\cdot,\cdot):\mathbb{R}^{n}\times\mathbb{R}^{n}\to\mathbb{R}$ takes the form $s(\cdot)\approx\sum_{i=1}^{M}a_{i}K(\widetilde{\mathbf{x}}_{i},\cdot)$ where $\{\widetilde{\mathbf{x}}_{i}\}_{1\leq i\leq M}$ are the $M$ input training samples or dictionary, and $\{a_{i}\}_{1\leq i\leq M}$ are the unknown coefficients to be learnt. This kernel function approximation has been used widely for non-linear state-space modeling \cite{tobar2015unsupervised,ralaivola2003dynamical} and recursive least-squares algorithms with unknown non-linear functions\cite{engel2004kernel,liu2009extended,van2006sliding}. Being a universal kernel \cite{steinwart2001influence}, a Gaussian kernel of kernel width $\sigma>0$ with $K(\mathbf{x}_{i},\mathbf{x}_{j})=\exp{\left(-\frac{\|\mathbf{x}_{i}-\mathbf{x}_{j}\|^{2}_{2}}{\sigma^2}\right)}$ is the most commonly used kernel for function approximation.

Consider a general non-linear system model with both state transition and observation models unknown to an agent. Our RKHS-UKF can be trivially simplified if the agent has perfect prior information about the state evolution, observation function, and/or noise covariance matrices. In particular, our RKHS-UKF couples the UKF to obtain state estimates with an approximate online EM algorithm - a popular choice to compute maximum likelihood estimates in the presence of missing data \cite{hajek2015random} - to learn the unknown system parameters. Further, the RKHS-UKF can be employed by both attacker and defender to infer, respectively, the defender's state (as a forward filter) and the attacker's state estimate (as an inverse filter). The inverse filters developed so far assumed a specific forward filter employed by the attacker to obtain their state transition equation. Since RKHS-UKF learns its state transition based on the available observations itself, we do not require any prior forward filter information to employ RKHS-UKF as the defender's inverse filter.

\textbf{System models for unknown dynamics:} We examine state transition \eqref{eqn:state transition x} and observation \eqref{eqn:observation y} when the functions $f(\cdot)$ and $h(\cdot)$, including the noise covariances $\mathbf{Q}$ and $\mathbf{R}$, are unknown. Consider a kernel function $K(\cdot,\cdot)$ and a dictionary $\{\widetilde{\mathbf{x}}_{l}\}_{1\leq l\leq L}$ of size $L$. Define $\bm{\Phi}(\mathbf{x})\doteq[K(\widetilde{\mathbf{x}}_{1},\mathbf{x}),\hdots,K(\widetilde{\mathbf{x}}_{L},\mathbf{x})]^{T}$. Using the kernel function approximation, the unknown state transition and observation are represented, respectively, as
\par\noindent\small
\begin{align}
    &\mathbf{x}_{k+1}=\mathbf{A}\bm{\Phi}(\mathbf{x}_{k})+\mathbf{w}_{k},\label{eqn:RKHS-UKF state transition approx}\\
    &\mathbf{y}_{k}=\mathbf{B}\bm{\Phi}(\mathbf{x}_{k})+\mathbf{v}_{k},\label{eqn:RKHS-UKF observation approx}
\end{align}
\normalsize
where $\mathbf{A}\in\mathbb{R}^{n_{x}\times L}$ and $\mathbf{B}\in\mathbb{R}^{n_{y}\times L}$ are the unknown coefficient matrices to be learnt. The dictionary $\{\widetilde{\mathbf{x}}_{l}\}_{1\leq l\leq L}$ may be formed using a sliding window\cite{van2006sliding} or approximate linear dependency (ALD)\cite{engel2004kernel} criterion. At the $k$-th time instant, RKHS-UKF estimates the unknown parameters $\Theta=\{\mathbf{A},\mathbf{B},\mathbf{Q},\mathbf{R}\}$ and current state $\mathbf{x}_{k}$ given observations $\{\mathbf{y}_{i}\}_{1\leq i\leq k}$. Note that the system model \eqref{eqn:state transition x}-\eqref{eqn:observation y} and hence, system parameters $\Theta$ are not time-varying.

RKHS-UKF's approximate online EM to learn the parameters $\Theta$ closely follows the RKHS-EKF's parameter learning steps detailed in \cite[Sec. V-A]{singh2022inverse_part2} and, hence, we only summarize them here. The key difference is that RKHS-EKF linearizes the non-linear kernel function to approximate the statistics of a Gaussian random variable under non-linear transformation, whereas the same approximation is performed by RKHS-UKF using the unscented transform. The UKF-based recursions then provide the state estimates using these parameter estimates. 

\textbf{Parameter learning:} Consider $\mathbf{Y}^{k}=\{\mathbf{y}_{j}\}_{1\leq j\leq k}$ as the observations upto time $k$ and $\hat{\Theta}_{k-1}=\{\hat{\mathbf{A}}_{k-1},\hat{\mathbf{B}}_{k-1},$ $\hat{\mathbf{Q}}_{k-1},\hat{\mathbf{R}}_{k-1}\}$ as the current estimate of $\Theta$ considering the previous $k-1$ observations. For simplicity, denote the conditional expectation operator $\mathbb{E}[\cdot|\mathbf{Y}^{k},\hat{\Theta}_{k-1}]$ by $\mathbb{E}_{k}[\cdot]$. Define the partial sums $\mathbf{S}^{x\phi}_{k}=\sum_{j=1}^{k}\mathbb{E}_{k}[\mathbf{x}_{j}\bm{\Phi}(\mathbf{x}_{j-1})^{T}]$ and $\mathbf{S}^{\phi 1}_{k}=\sum_{j=1}^{k}\mathbb{E}_{k}[\bm{\Phi}(\mathbf{x}_{j-1})\bm{\Phi}(\mathbf{x}_{j-1})^{T}]$; and approximate them as $\mathbf{S}^{x\phi}_{k}\approx\mathbf{S}^{x\phi}_{k-1}+\mathbb{E}_{k}[\mathbf{x}_{k}\bm{\Phi}(\mathbf{x}_{k-1})^{T}]$ and $\mathbf{S}^{\phi 1}_{k}\approx\mathbf{S}^{\phi 1}_{k-1}+\mathbb{E}_{k}[\bm{\Phi}(\mathbf{x}_{k-1})\bm{\Phi}(\mathbf{x}_{k-1})^{T}]$. In the approximate online EM, the current observation $\mathbf{y}_{k}$ and parameter estimate $\hat{\Theta}_{k-1}$ are used only to compute the expectations $\mathbb{E}_{k}[\mathbf{x}_{k}\bm{\Phi}(\mathbf{x}_{k-1})^{T}]$ and $\mathbb{E}_{k}[\bm{\Phi}(\mathbf{x}_{k-1})\bm{\Phi}(\mathbf{x}_{k-1})^{T}]$ and not to update $\mathbf{S}^{x\phi}_{k-1}$ and $\mathbf{S}^{\phi 1}_{k-1}$. With this approximation, the current observation is considered only once to update the parameter estimates. Similarly, we define the sums $\mathbf{S}^{y\phi}_{k}=\sum_{j=1}^{k}\mathbb{E}_{k}[\mathbf{y}_{j}\bm{\Phi}(\mathbf{x}_{j})^{T}]$ and $\mathbf{S}^{\phi}_{k}=\sum_{j=1}^{k}\mathbb{E}_{k}[\bm{\Phi}(\mathbf{x}_{j})\bm{\Phi}(\mathbf{x}_{j})^{T}]$ which are evaluated, respectively, as $\mathbf{S}^{y\phi}_{k}=\mathbf{S}^{y\phi}_{k-1}+\mathbb{E}_{k}[\mathbf{y}_{k}\bm{\Phi}(\mathbf{x}_{k})^{T}]$ and $\mathbf{S}^{\phi}_{k}=\mathbf{S}^{\phi}_{k-1}+\mathbb{E}_{k}[\bm{\Phi}(\mathbf{x}_{k})\bm{\Phi}(\mathbf{x}_{k})^{T}]$. Further, using \eqref{eqn:RKHS-UKF observation approx}, we have $\mathbb{E}_{k}[\mathbf{y}_{k}\bm{\Phi}(\mathbf{x}_{k})^{T}]=\hat{\mathbf{B}}_{k}\mathbb{E}_{k}[\bm{\Phi}(\mathbf{x}_{k})\bm{\Phi}(\mathbf{x}_{k})^{T}]$ and $\mathbb{E}_{k}[\mathbf{y}_{k}\mathbf{y}_{k}^{T}]=\hat{\mathbf{B}}_{k}\mathbb{E}_{k}[\bm{\Phi}(\mathbf{x}_{k})\bm{\Phi}(\mathbf{x}_{k})^{T}]\hat{\mathbf{B}}_{k}^{T}+\hat{\mathbf{R}}_{k-1}$. Finally, the approximate parameter updates are
\par\noindent\small
\begin{align}
    &\hat{\mathbf{A}}_{k}=\Gamma_{a}(\mathbf{S}^{x\phi}_{k}(\mathbf{S}^{\phi 1}_{k})^{-1}),\label{eqn:RKHS-UKF A estimate}\\
    &\hat{\mathbf{Q}}_{k}=\left(1-\frac{1}{k}\right)\hat{\mathbf{Q}}_{k-1}+\frac{1}{k}(\mathbb{E}_{k}[\mathbf{x}_{k}\mathbf{x}_{k}^{T}]-\hat{\mathbf{A}}_{k}\mathbb{E}_{k}[\bm{\Phi}(\mathbf{x}_{k-1})\mathbf{x}_{k}^{T}]\nonumber\\
    &-\mathbb{E}_{k}[\mathbf{x}_{k}\bm{\Phi}(\mathbf{x}_{k-1})^{T}]\hat{\mathbf{A}}_{k}^{T}+\hat{\mathbf{A}}_{k}\mathbb{E}_{k}[\bm{\Phi}(\mathbf{x}_{k-1})\bm{\Phi}(\mathbf{x}_{k-1})^{T}]\hat{\mathbf{A}}_{k}^{T}),\label{eqn:RKHS-UKF Q estimate}\\
    &\hat{\mathbf{B}}_{k}=\Gamma_{b}(\mathbf{S}^{y\phi}_{k}(\mathbf{S}^{\phi}_{k})^{-1}),\label{eqn:RKHS-UKF B estimate}\\
    &\hat{\mathbf{R}}_{k}=\left(1-\frac{1}{k}\right)\hat{\mathbf{R}}_{k-1}+\frac{1}{k}(\mathbb{E}_{k}[\mathbf{y}_{k}\mathbf{y}_{k}^{T}]-\hat{\mathbf{B}}_{k}\mathbb{E}_{k}[\bm{\Phi}(\mathbf{x}_{k})\mathbf{y}_{k}^{T}]\nonumber\\
    &-\mathbb{E}_{k}[\mathbf{y}_{k}\bm{\Phi}(\mathbf{x}_{k})^{T}]\hat{\mathbf{B}}_{k}^{T}+\hat{\mathbf{B}}_{k}\mathbb{E}_{k}[\bm{\Phi}(\mathbf{x}_{k})\bm{\Phi}(\mathbf{x}_{k})^{T}]\hat{\mathbf{B}}_{k}^{T}),\label{eqn:RKHS-UKF R estimate}
\end{align}
\normalsize
where $\Gamma_{a}(\cdot)$ and $\Gamma_{b}(\cdot)$ denote respectively the projection of the coefficient matrices estimates to satisfy known bounds. Note that this projection is optional and used only to analyze the stability of the proposed RKHS-UKF algorithm in Section~\ref{sec:stability}.

To compute the required conditional expectations, we use the state estimates from the UKF recursions themselves, which approximate the posterior distribution of the required states given the observations by a Gaussian density. However, the expectations involve a non-linear transformation $\bm{\Phi}(\cdot)$. Based on the unscented transform, we generate the deterministic sigma points to approximate the statistics of a Gaussian distribution under non-linear transformation. Further, we need the statistics of $\bm{\Phi}(\mathbf{x}_{k-1})$ given $\mathbf{Y}^{k}$. Hence, an augmented state $\mathbf{z}_{k}=[\mathbf{x}_{k}^{T}\;\mathbf{x}_{k-1}^{T}]^{T}$ is considered in the RKHS-UKF formulation. Using these approximations, the required expectations are then computed as detailed below.

\textbf{Recursion:} In terms of the augmented state $\mathbf{z}_{k}$, the kernel approximated system model is
\par\noindent\small
\begin{align}
    &\mathbf{z}_{k}=\widetilde{f}(\mathbf{z}_{k-1})+\widetilde{\mathbf{w}}_{k-1},\label{eqn:RKHS-UKF state transition aug}\\
    &\mathbf{y}_{k}=\widetilde{h}(\mathbf{z}_{k})+\mathbf{v}_{k},\label{eqn:RKHS-UKF observation aug}
\end{align}
\normalsize
where $\widetilde{f}(\mathbf{z}_{k-1})=[(\mathbf{A}\bm{\Phi}(\mathbf{x}_{k-1}))^{T}\;\mathbf{x}_{k-1}^{T}]^{T}$ and $\widetilde{h}(\mathbf{z}_{k})=\mathbf{B}\bm{\Phi}(\mathbf{x}_{k})$ with process noise $\widetilde{\mathbf{w}}_{k-1}=[\mathbf{w}_{k-1}^{T}\;\mathbf{0}_{1\times n_{x}}]^{T}$ of actual noise covariance matrix $\widetilde{\mathbf{Q}}=\begin{bsmallmatrix}\mathbf{Q}&\mathbf{0}_{n_{x}\times n_{x}}\\\mathbf{0}_{n_{x}\times n_{x}}&\mathbf{0}_{n_{x}\times n_{x}}\end{bsmallmatrix}$. Denote $\kappa$ as the scaling parameter for RKHS-UKF's sigma points. Note that these sigma points are generated in the augmented state-space with dimension $n_{z}=2n_{x}$. At $k$-th recursion, we have an estimate $\hat{\mathbf{z}}_{k-1}=[\hat{\mathbf{x}}_{k-1|k-1}^{T}\;\hat{\mathbf{x}}_{k-2|k-1}^{T}]^{T}$, its associated error covariance matrix $\bm{\Sigma}^{z}_{k-1}$ and the parameter estimate $\hat{\Theta}_{k-1}$ from the previous time step. The current state and parameter estimates are then obtained using the new observation $\mathbf{y}_{k}$ as summarized below.\\
\textit{1) Prediction:} Generate the sigma points $\{\mathbf{s}_{i,k-1}\}_{0\leq i\leq 2n_{z}}=S_{gen}(\hat{\mathbf{z}}_{k-1},\bm{\Sigma}^{z}_{k-1})$ with their corresponding weights $\{\omega_{i}\}_{0\leq i\leq 2n_{z}}$ similar to \eqref{eqn:sigma points generation}. Using $\mathbf{A}=\hat{\mathbf{A}}_{k-1}$ and $\widetilde{\mathbf{Q}}_{k-1}=\begin{bsmallmatrix}\hat{\mathbf{Q}}_{k-1}&\mathbf{0}_{n_{x}\times n_{x}}\\\mathbf{0}_{n_{x}\times n_{x}}&\mathbf{0}_{n_{x}\times n_{x}}\end{bsmallmatrix}$ in \eqref{eqn:RKHS-UKF state transition aug}, propagate $\{\mathbf{s}_{i,k-1}\}$ through $\widetilde{f}(\cdot)$ to obtain the predicted state and associated prediction error covariance matrix as
\par\noindent\small
\begin{align}
    &\mathbf{s}^{*}_{i,k|k-1}=\widetilde{f}(\mathbf{s}_{i,k-1})\;\;\forall\; i=0,1,\hdots,2n_{z},\label{eqn:RKHS-UKF prediction propagate}\\  
    &\hat{\mathbf{z}}_{k|k-1}=\sum_{i=0}^{2n_{z}}\omega_{i}\mathbf{s}^{*}_{i,k|k-1},\label{eqn:RKHS-UKF predicted state}\\
    &\bm{\Sigma}^{z}_{k|k-1}=\sum_{i=0}^{2n_{z}}\omega_{i}\mathbf{s}^{*}_{i,k|k-1}(\mathbf{s}^{*}_{i,k|k-1})^{T}-\hat{\mathbf{z}}_{k|k-1}\hat{\mathbf{z}}_{k|k-1}^{T}+\widetilde{\mathbf{Q}}_{k-1}.\label{eqn:RKHS-UKF sig predict}
\end{align}
\normalsize
\textit{2) State update:} Generate the sigma points  $\{\mathbf{q}_{i,k|k-1}\}_{0\leq i\leq 2n_{z}}\\=S_{gen}(\hat{\mathbf{z}}_{k|k-1},\bm{\Sigma}^{z}_{k|k-1})$ and propagate through $\widetilde{h}(\cdot)$ using $\mathbf{B}=\hat{\mathbf{B}}_{k-1}$ and $\mathbf{R}=\hat{\mathbf{R}}_{k-1}$ in \eqref{eqn:RKHS-UKF observation aug} as
\par\noindent\small
\begin{align}
    &\mathbf{q}^{*}_{i,k|k-1}=\widetilde{h}(\mathbf{q}_{i,k|k-1})\;\;\forall\; i=0,1,\hdots,2n_{z},\label{eqn:RKHS-UKF state update propagate}\\
    &\hat{\mathbf{y}}_{k|k-1}=\sum_{i=0}^{2n_{z}}\omega_{i}\mathbf{q}^{*}_{i,k|k-1},\label{eqn:RKHS-UKF predicted y}\\
    &\bm{\Sigma}^{y}_{k}=\sum_{i=0}^{2n_{z}}\omega_{i}\mathbf{q}^{*}_{i,k|k-1}(\mathbf{q}^{*}_{i,k|k-1})^{T}-\hat{\mathbf{y}}_{k|k-1}\hat{\mathbf{y}}_{k|k-1}^{T}+\mathbf{R},\label{eqn:RKHS-UKF sig y predict}\\
    &\bm{\Sigma}^{zy}_{k}=\sum_{i=0}^{2n_{z}}\omega_{i}\mathbf{q}_{i,k|k-1}(\mathbf{q}^{*}_{i,k|k-1})^{T}-\hat{\mathbf{z}}_{k|k-1}\hat{\mathbf{y}}_{k|k-1}^{T},\label{eqn:RKHS-UKF sig zy predict}\\
    &\hat{\mathbf{z}}_{k}=\hat{\mathbf{z}}_{k|k-1}+\mathbf{K}_{k}(\mathbf{y}_{k}-\hat{\mathbf{y}}_{k|k-1}),\label{eqn:RKHS-UKF state update}\\
     &\bm{\Sigma}^{z}_{k}=\bm{\Sigma}^{z}_{k|k-1}-\mathbf{K}_{k}\bm{\Sigma}^{y}_{k}\mathbf{K}_{k}^{T},\label{eqn:RKHS-UKF sig update}
\end{align}
\normalsize
where gain matrix $\mathbf{K}_{k}=\bm{\Sigma}^{zy}_{k}(\bm{\Sigma}^{y}_{k})^{-1}$. Here, $\hat{\mathbf{z}}_{k}=[\hat{\mathbf{x}}_{k|k}^{T}\;\hat{\mathbf{x}}_{k-1|k}^{T}]$ where $\hat{\mathbf{x}}_{k|k}$ is the RKHS-UKF's estimate of state $\mathbf{x}_{k}$. These prediction and state updates follow from the standard UKF recursions to estimate $\mathbf{z}_{k}$ with system model \eqref{eqn:RKHS-UKF state transition aug} and \eqref{eqn:RKHS-UKF observation aug} given parameters $\hat{\Theta}_{k-1}$.\\
\textit{3) Parameters update:} Generate sigma points $\{\mathbf{s}_{i,k}\}_{0\leq i\leq 2n_{z}}\\=S_{gen}(\hat{\mathbf{z}}_{k},\bm{\Sigma}^{z}_{k})$. Denote the two $n_{x}$-dimensional sub-vectors of the $i$-th sigma point, respectively, as $\mathbf{s}_{i,k}^{(1)}=[\mathbf{s}_{i,k}]_{1:n_{x}}$ and $\mathbf{s}_{i,k}^{(2)}=[\mathbf{s}_{i,k}]_{n_{x}+1:2n_{x}}$ corresponding to $\mathbf{x}_{k}$ and $\mathbf{x}_{k-1}$ parts of the augmented state $\mathbf{z}_{k}$. In order to approximate the statistics of $\bm{\Phi}(\mathbf{x}_{k-1})$ given observations $\mathbf{Y}^{k}$, we propagate the $\mathbf{x}_{k-1}$ part of the sigma points through $\bm{\Phi}(\cdot)$, i.e., $\widetilde{\mathbf{s}}_{i,k}^{(2)}=\bm{\Phi}(\mathbf{s}_{i,k}^{(2)})$. Similarly, we propagate $\mathbf{s}_{i,k}^{(1)}$ through $\bm{\Phi}(\cdot)$ to approximate the statistics of $\bm{\Phi}(\mathbf{x}_{k})$ as $\widetilde{\mathbf{s}}_{i,k}^{(1)}=\bm{\Phi}(\mathbf{s}_{i,k}^{(1)})$. Also, by definition $\textrm{Cov}(\mathbf{x}_{k}-\hat{\mathbf{x}}_{k|k})\approx[\bm{\Sigma}^{z}_{k}]_{(1:n_{x},1:n_{x})}$. With these approximations, the various expectations are computed as
\par\noindent\small
\begin{align}
    &\mathbb{E}_{k}[\mathbf{x}_{k}\mathbf{x}_{k}^{T}]=[\bm{\Sigma}^{z}_{k}]_{(1:n_{x},1:n_{x})}+\hat{\mathbf{x}}_{k|k}\hat{\mathbf{x}}_{k|k}^{T},\label{eqn:RKHS-UKF expect xx}\\
     &\mathbb{E}_{k}[\bm{\Phi}(\mathbf{x}_{k-1})\bm{\Phi}(\mathbf{x}_{k-1})^{T}]=\sum_{i=0}^{2n_{z}}\omega_{i}\widetilde{\mathbf{s}}_{i,k}^{(2)}(\widetilde{\mathbf{s}}_{i,k}^{(2)})^{T},\label{eqn:RKHS-UKF expect phi1}\\
     &\mathbb{E}_{k}[\mathbf{x}_{k}\bm{\Phi}(\mathbf{x}_{k-1})^{T}]=\sum_{i=0}^{2n_{z}}\omega_{i}\mathbf{s}_{i,k}^{(1)}(\widetilde{\mathbf{s}}_{i,k}^{(2)})^{T},\label{eqn:RKHS-UKF expect xphi}\\
     &\mathbb{E}_{k}[\bm{\Phi}(\mathbf{x}_{k})\bm{\Phi}(\mathbf{x}_{k})^{T}]=\sum_{i=0}^{2n_{z}}\omega_{i}\widetilde{\mathbf{s}}_{i,k}^{(1)}(\widetilde{\mathbf{s}}_{i,k}^{(1)})^{T}.\label{eqn:RKHS-UKF expect phi}
\end{align}
\normalsize
Using these expectations, the updated parameters estimate $\hat{\Theta}_{k}$ is obtained using \eqref{eqn:RKHS-UKF A estimate}-\eqref{eqn:RKHS-UKF R estimate}.

Finally, the dictionary $\{\widetilde{\mathbf{x}}_{l}\}_{1\leq l\leq L}$ is updated using the current state estimate $\hat{\mathbf{x}}_{k|k}$ based on the sliding window\cite{van2006sliding} or ALD\cite{engel2004kernel} criterion. Note that in our RKHS-UKF, we need to generate only two sets of sigma points per recursion, similar to the standard UKF with a known system model. The sigma points generated in the parameters update step are the same as that obtained in the prediction step at the next time instant. Algorithm~\ref{alg:RKHS-UKF initialization} describes the initialization of the RKHS-UKF assuming initial state $\mathbf{x}_{0}\sim\mathcal{N}(\hat{\mathbf{x}}_{0},\bm{\Sigma}_{0})$, while Algorithm~\ref{alg:RKHS-UKF recursion} summarizes the RKHS-UKF recursions. In the ALD criterion, the dictionary size increases if the current estimate $\hat{\mathbf{x}}_{k|k}$ is added to the dictionary, while in the sliding window criterion, the dictionary size changes only during the initial phase when the current size $L$ is less than the window length considered. Algorithms~\ref{alg:RKHS-UKF initialization} and \ref{alg:RKHS-UKF recursion} summarize all the above-mentioned steps.

The usage of the unscented transform implies that the scaling parameter $\kappa$ also impacts the parameter learning in RKHS-UKF. Alternatively, efficient numerical integration techniques can be employed and RKHS-CKF (RKHS-QKF) can be trivially obtained from the developed RKHS-UKF. In \cite{haykin2004kalman,zhan2006neural}, UKF has been coupled with neural networks (NNs) for non-linear system identification with the unknown weights learned by augmenting them with the state, which is computationally expensive. In \cite{kallapur2008ukf}, NN is used to model the non-linear dynamics or uncertainties with the weights updated online using previous UKF-computed state estimates. However, in general, the stability of these algorithms can not be assured\cite{haykin2004kalman}. Contrarily, our RKHS-UKF extends the recursive Bayesian state estimation framework to unknown systems by including an additional parameter update step using the expectations from UKF recursion itself. Hence, we are able to analyze its stochastic stability using the unknown matrix approach\cite{xiong2006performance_ukf} in Section~\ref{subsec:RKHS-UKF stability}. Such stability guarantees have not been derived for RKHS-EKF\cite{singh2022inverse_part2} as well.

\begin{remark}[Difference from kernel KFs]\label{remark:kernel KF}
    Kernel-based KFs (KKFs) have been previously proposed in the literature for unknown non-linear time series prediction\cite{ralaivola2005time,zhu2013learning,dang2019kernel}. However, our RKHS-UKF is fundamentally different from these KKFs. In \cite{ralaivola2005time}, the unknown system model is transformed to an RKHS-based feature space $\mathcal{H}$, wherein it is assumed to be linear Gaussian such that KF-based filtering recursions are applicable. The unknown parameters are represented using an orthonormal basis in $\mathcal{H}$ and learned from an exact EM algorithm. The kernel principal component analysis (PCA) provides the orthonormal basis. KKFs in \cite{zhu2013learning,dang2019kernel} consider conditional distribution embedding of the noisy but linear observation $\mathbf{y}_{k}$ to construct a new state-space model in RKHS. KF recursions are then applied to estimate the embedded observation, from which the estimate $\hat{\mathbf{x}}_{k}$ is computed. On the contrary, our RKHS-UKF employs the kernel function approximation (with a universal Gaussian kernel) to directly represent the unknown functions in the original state space without any feature mapping. Furthermore, these KKFs involve a prior training phase to learn the orthonormal basis and the embedding operator in \cite{ralaivola2005time} and \cite{zhu2013learning,dang2019kernel}, respectively. Hence, KKFs suffer if the training data does not adequately represent the test data. On the contrary, our RKHS-UKF learns its dictionary online (from the computed estimates) based on the ALD or sliding window criterion.
\end{remark}

\begin{algorithm}[t]
	\caption{RKHS-UKF initialization}
	\label{alg:RKHS-UKF initialization}
    \begin{algorithmic}[1]
    \Statex \textbf{Input:} $\hat{\mathbf{x}}_{0}$, $\bm{\Sigma}_{0}$, $\kappa$
    \Statex \textbf{Output:} $\{\mathbf{s}_{i,0}\}_{0\leq i\leq 2n_{z}}$, $\{\omega_{i}\}_{0\leq i\leq 2n_{z}}$, $\bm{\Sigma}^{z}_{0}$, $L$, $\{\widetilde{x}_{l}\}_{1\leq l\leq L}$, $\hat{\mathbf{A}}_{0}$, $\hat{\mathbf{B}}_{0}$, $\hat{\mathbf{Q}}_{0}$, $\hat{\mathbf{R}}_{0}$, $\mathbf{S}^{x\phi}_{0}$, $\mathbf{S}^{\phi 1}_{0}$, $\mathbf{S}^{y\phi}_{0}$ and $\mathbf{S}^{\phi}_{0}$
\State $\hat{\mathbf{z}}_{0}\gets [\hat{\mathbf{x}}_{0}^{T}\;\;\hat{\mathbf{x}}_{0}^{T}]^{T}$, and $\bm{\Sigma}^{z}_{0}\gets\begin{bmatrix}\bm{\Sigma}_{0}&\mathbf{0}_{n_{x}\times n_{x}}\\\mathbf{0}_{n_{x}\times n_{x}}&\bm{\Sigma}_{0}\end{bmatrix}$.
\State $\{\mathbf{s}_{i,0},\omega_{i}\}_{0\leq i\leq 2n_{z}}\gets S_{gen}(\hat{\mathbf{z}}_{0},\bm{\Sigma}^{z}_{0})$ using \eqref{eqn:sigma points generation} and $\kappa$.
\State Set $L=1$ and $\widetilde{\mathbf{x}}_{1}\gets\hat{\mathbf{x}}_{0}$.
\State Initialize $\hat{\mathbf{A}}_{0}$ and $\hat{\mathbf{B}}_{0}$ with arbitrary $\mathbb{R}_{n_{x}\times L}$ and $\mathbb{R}_{n_{y}\times L}$ matrices, respectively.
\State Initialize $\hat{\mathbf{Q}}_{0}$ and $\hat{\mathbf{R}}_{0}$ with some suitable positive definite noise covariance matrices.
\State Set $\mathbf{S}^{x\phi}_{0}=\mathbf{0}_{n_{x}\times L}$, $\mathbf{S}^{\phi 1}_{0}=\mathbf{0}_{L\times L}$, $\mathbf{S}^{y\phi}_{0}=\mathbf{0}_{n_{y}\times L}$ and $\mathbf{S}^{\phi}_{0}=\mathbf{0}_{L\times L}$.

\Statex \Return $\{\mathbf{s}_{i,0}\}_{0\leq i\leq 2n_{z}}$, $\{\omega_{i}\}_{0\leq i\leq 2n_{z}}$, $\bm{\Sigma}^{z}_{0}$, $L$, $\{\widetilde{x}_{l}\}_{1\leq l\leq L}$, $\hat{\mathbf{A}}_{0}$, $\hat{\mathbf{B}}_{0}$, $\hat{\mathbf{Q}}_{0}$, $\hat{\mathbf{R}}_{0}$, $\mathbf{S}^{x\phi}_{0}$, $\mathbf{S}^{\phi 1}_{0}$, $\mathbf{S}^{y\phi}_{0}$ and $\mathbf{S}^{\phi}_{0}$.

    \end{algorithmic}
\end{algorithm}
\begin{algorithm}[ht]
	\caption{RKHS-UKF recursion}
	\label{alg:RKHS-UKF recursion}
    \begin{algorithmic}[1]
    \Statex \textbf{Input:} $\{\mathbf{s}_{i,k-1},\omega_{i}\}_{0\leq i\leq 2n_{z}}$, $\bm{\Sigma}^{z}_{k-1}$, $\hat{\mathbf{A}}_{k-1}$, $\hat{\mathbf{B}}_{k-1}$, $\hat{\mathbf{Q}}_{k-1}$, $\hat{\mathbf{R}}_{k-1}$, $\mathbf{S}^{x\phi}_{k-1}$, $\mathbf{S}^{\phi 1}_{k-1}$, $\mathbf{S}^{y\phi}_{k-1}$, $\mathbf{S}^{\phi}_{k-1}$, and $\mathbf{y}_{k}$
    \Statex \textbf{Output:} $\hat{\mathbf{x}}_{k|k}$, $\{\mathbf{s}_{i,k}\}_{0\leq i\leq 2n_{z}}$, $\bm{\Sigma}^{z}_{k}$, $\hat{\mathbf{A}}_{k}$, $\hat{\mathbf{B}}_{k}$, $\hat{\mathbf{Q}}_{k}$, $\hat{\mathbf{R}}_{k}$, $\mathbf{S}^{x\phi}_{k}$, $\mathbf{S}^{\phi 1}_{k}$, $\mathbf{S}^{y\phi}_{k}$, and $\mathbf{S}^{\phi}_{k}$
\State Propagate $\{\mathbf{s}_{i,k-1}\}$ through $\widetilde{f}(\cdot)$ and compute $\hat{\mathbf{z}}_{k|k-1}$ and $\bm{\Sigma}^{z}_{k|k-1}$ using \eqref{eqn:RKHS-UKF prediction propagate}-\eqref{eqn:RKHS-UKF sig predict}.
\State Generate sigma-points $\{\mathbf{q}_{i,k|k-1}\}_{0\leq i\leq 2n_{z}}=S_{gen}(\hat{\mathbf{z}}_{k|k-1},\bm{\Sigma}^{z}_{k|k-1})$ using \eqref{eqn:sigma points generation} and propagate through $\widetilde{h}(\cdot)$ using \eqref{eqn:RKHS-UKF state update propagate}.
\State Compute $\hat{\mathbf{z}}_{k}$ and $\bm{\Sigma}^{z}_{k}$ using \eqref{eqn:RKHS-UKF predicted y}-\eqref{eqn:RKHS-UKF sig update}.
\State $\hat{\mathbf{x}}_{k|k}\gets[\hat{\mathbf{z}}_{k}]_{1:n_{x}}$.
\State Generate sigma-points $\{\mathbf{s}_{i,k}\}_{0\leq i\leq 2n_{z}}=S_{gen}(\hat{\mathbf{z}}_{k},\bm{\Sigma}^{z}_{k})$ using \eqref{eqn:sigma points generation}.
\State Set $\mathbf{s}_{i,k}^{(1)}=[\mathbf{s}_{i,k}]_{1:n_{x}}$ and $\mathbf{s}_{i,k}^{(2)}=[\mathbf{s}_{i,k}]_{n_{x}+1:2n_{x}}$.
\State Compute $\widetilde{\mathbf{s}}_{i,k}^{(1)}=\bm{\Phi}(\mathbf{s}_{i,k}^{(1)})$ and $\widetilde{\mathbf{s}}_{i,k}^{(2)}=\bm{\Phi}(\mathbf{s}_{i,k}^{(2)})$.
\State Compute $\mathbb{E}_{k}[\mathbf{x}_{k}\mathbf{x}_{k}^{T}]$, $\mathbb{E}_{k}[\bm{\Phi}(\mathbf{x}_{k-1})\bm{\Phi}(\mathbf{x}_{k-1})^{T}]$, $\mathbb{E}_{k}[\mathbf{x}_{k}\bm{\Phi}(\mathbf{x}_{k-1})^{T}]$ and $\mathbb{E}_{k}[\bm{\Phi}(\mathbf{x}_{k})\bm{\Phi}(\mathbf{x}_{k})^{T}]$ using \eqref{eqn:RKHS-UKF expect xx}-\eqref{eqn:RKHS-UKF expect phi}.
\State Compute $\hat{\mathbf{A}}_{k}$, $\hat{\mathbf{B}}_{k}$, $\hat{\mathbf{Q}}_{k}$, and $\hat{\mathbf{R}}_{k}$ using \eqref{eqn:RKHS-UKF A estimate}-\eqref{eqn:RKHS-UKF R estimate}.
\State Update dictionary $\{\widetilde{\mathbf{x}}_{l}\}_{1\leq l\leq L}$ using the state estimate $\hat{\mathbf{x}}_{k|k}$ based on the sliding window\cite{van2006sliding} or ALD\cite{engel2004kernel} criterion.
\If{dictionary size increases}
    \Statex Augment $\hat{\mathbf{A}}_{k}$, $\hat{\mathbf{B}}_{k}$, $\mathbf{S}^{x\phi}_{k}$, $\mathbf{S}^{\phi 1}_{k}$, $\mathbf{S}^{y\phi}_{k}$, and $\mathbf{S}^{\phi}_{k}$ with suitable initial values to take into account the updated dictionary size.
\EndIf
\Statex \Return $\hat{\mathbf{x}}_{k|k}$, $\{\mathbf{s}_{i,k}\}_{0\leq i\leq 2n_{z}}$, $\bm{\Sigma}^{z}_{k}$, $\hat{\mathbf{A}}_{k}$, $\hat{\mathbf{B}}_{k}$, $\hat{\mathbf{Q}}_{k}$, $\hat{\mathbf{R}}_{k}$, $\mathbf{S}^{x\phi}_{k}$, $\mathbf{S}^{\phi 1}_{k}$, $\mathbf{S}^{y\phi}_{k}$, and $\mathbf{S}^{\phi}_{k}$.
    \end{algorithmic}    
\end{algorithm}
\section{Performance Analyses}
\label{sec:stability}
We adopt the unknown matrix approach to derive sufficient conditions for the stochastic stability of the proposed filters in the exponential-mean-squared-boundedness sense. Proposed in \cite{xiong2006performance_ukf} for UKF with linear observations, the unknown matrix approach has been used to derive the stability conditions for UKF with intermittent observations \cite{li2012stochastic_ukf} and consensus-based UKF \cite{li2015weighted}. In the following, we derive the stability conditions for I-UKF and RKHS-UKF considering the general case of time-varying process and measurement noise covariances $\mathbf{Q}_{k}$, $\mathbf{R}_{k}$ and $\overline{\mathbf{R}}_{k}$ instead of $\mathbf{Q}$, $\mathbf{R}$ and $\bm{\Sigma}_{\epsilon}$, respectively. Note that, analogous to the stability literature of KF, Theorems~\ref{theorem:forward ukf stability}, \ref{theorem:IUKF stability} and \ref{theorem:RKHS-UKF stability} are sufficient but not necessary conditions for stability.

The SPKFs, including UKF, are local approximation approaches\cite{li2017approximate} such that, unlike KF, the gain computation and covariance update steps are coupled with the state updates. Therefore, the asymptotic convergence of the filter is not defined for UKF and I-UKF. Instead, we show that the I-UKF provides conservative estimates in Theorem~\ref{thm:consistency}.

First, recall the definition of the exponential-mean-squared-boundedness of a stochastic process.
\begin{definition}[Exponential mean-squared boundedness \cite{reif1999stochastic}] A stochastic process $\{\bm{\zeta}_{k} \}_{k \geq 0}$ is defined to be exponentially bounded in mean-squared sense if there are real numbers $\eta,\nu>0$ and $0<\lambda<1$ such that $\mathbb{E}\left[\|\bm{\zeta}_{k}\|_{2}^{2}\right]\leq \eta\mathbb{E}\left[\|\bm{\zeta}_{0}\|_{2}^{2}\right]\lambda^{k}+\nu$ holds for every $k\geq 0$.
\end{definition}
\vspace{-10pt}
\subsection{Inverse UKF}
\label{subsec:IUKF stability}
The I-UKF's estimation error can be shown to be exponentially bounded if the forward UKF is stable and the system satisfies some additional assumptions. Particularly, the I-UKF's error dynamics can be shown to satisfy the stability conditions of a generalized UKF. Therefore, we first examine the stochastic stability of forward UKF.

\subsubsection{Forward UKF}
Consider the forward UKF of Section~\ref{subsec:forward UKF} with state-transition \eqref{eqn:state transition x} and observation \eqref{eqn:observation y}. Denote the state prediction, state estimation and measurement prediction errors by $\widetilde{\mathbf{x}}_{k+1|k}\doteq\mathbf{x}_{k+1}-\hat{\mathbf{x}}_{k+1|k}$, $\widetilde{\mathbf{x}}_{k}\doteq\mathbf{x}_{k}-\hat{\mathbf{x}}_{k}$ and $\widetilde{\mathbf{y}}_{k+1}\doteq\mathbf{y}_{k+1}-\hat{\mathbf{y}}_{k+1|k}$, respectively. From \eqref{eqn:state transition x} and \eqref{eqn:forward ukf x predict}, we have $\widetilde{\mathbf{x}}_{k+1|k}=f(\mathbf{x}_{k})+\mathbf{w}_{k}-\sum_{i=0}^{2n_{x}}\omega_{i}\mathbf{s}^{*}_{i,k+1|k}$, which on substituting $\mathbf{s}^{*}_{i,k+1|k}=f(\mathbf{s}_{i,k})$ yields $\widetilde{\mathbf{x}}_{k+1|k}=f(\mathbf{x}_{k})+\mathbf{w}_{k}-\sum_{i=0}^{2n_{x}}\omega_{i}f(\mathbf{s}_{i,k})$.
Using the first-order Taylor series expansion of $f(\cdot)$ at $\hat{\mathbf{x}}_{k}$, we have $\widetilde{\mathbf{x}}_{k+1|k}\approx f(\hat{\mathbf{x}}_{k})+\mathbf{F}_{k}(\mathbf{x}_{k}-\hat{\mathbf{x}}_{k})+\mathbf{w}_{k}-\sum_{i=0}^{2n_{x}}\omega_{i}(f(\hat{\mathbf{x}}_{k})+\mathbf{F}_{k}(\mathbf{s}_{i,k}-\hat{\mathbf{x}}_{k}))$, where $\mathbf{F}_{k}\doteq\frac{\partial f(\mathbf{x})}{\partial\mathbf{x}}\vert_{\mathbf{x}=\hat{\mathbf{x}}_{k}}$. The sigma points $\lbrace\mathbf{s}_{i,k}\rbrace_{0\leq i\leq 2n_{x}}$ are chosen symmetrically about $\hat{\mathbf{x}}_{k}$. Substituting for $\mathbf{s}_{i,k}$ in terms of $\hat{\mathbf{x}}_{k}$ and $\bm{\Sigma}_{k}$ using \eqref{eqn:forward ukf prediction sigma points} simplifies the state prediction error to $\widetilde{\mathbf{x}}_{k+1|k}\approx\mathbf{F}_{k}\widetilde{\mathbf{x}}_{k}+\mathbf{w}_{k}$. Similar to \cite{xiong2006performance_ukf}, we introduce an unknown instrumental diagonal matrix $\mathbf{U}^{x}_{k}\in\mathbb{R}^{n_{x}\times n_{x}}$ to account for the linearization errors as
\par\noindent\small
\begin{align}
    \widetilde{\mathbf{x}}_{k+1|k}=\mathbf{U}^{x}_{k}\mathbf{F}_{k}\widetilde{\mathbf{x}}_{k}+\mathbf{w}_{k}.\label{eqn:linearized x}
\end{align}
\normalsize
Similarly, linearizing $h(\cdot)$ in \eqref{eqn:observation y} and introducing unknown diagonal matrix $\mathbf{U}^{y}_{k}\in\mathbb{R}^{n_{y}\times n_{y}}$ in \eqref{eqn:forward ukf y predict} yields
\par\noindent\small
\begin{align}
    \widetilde{\mathbf{y}}_{k+1}=\mathbf{U}^{y}_{k+1}\mathbf{H}_{k+1}\widetilde{\mathbf{x}}_{k+1|k}+\mathbf{v}_{k+1},\label{eqn:linearized y}
\end{align}
\normalsize
where $\mathbf{H}_{k+1}\doteq\frac{\partial h(\mathbf{x})}{\partial\mathbf{x}}\vert_{\mathbf{x}=\hat{\mathbf{x}}_{k+1|k}}$. Using \eqref{eqn:forward ukf x update}, we have $\widetilde{\mathbf{x}}_{k}=\widetilde{\mathbf{x}}_{k|k-1}-\mathbf{K}_{k}\widetilde{\mathbf{y}}_{k}$, which when substituted in \eqref{eqn:linearized x} with \eqref{eqn:linearized y} yields the forward UKF's prediction error dynamics as
\par\noindent\small
\begin{align}
    \widetilde{\mathbf{x}}_{k+1|k}=\mathbf{U}^{x}_{k}\mathbf{F}_{k}(\mathbf{I}-\mathbf{K}_{k}\mathbf{U}^{y}_{k}\mathbf{H}_{k})\widetilde{\mathbf{x}}_{k|k-1}-\mathbf{U}^{x}_{k}\mathbf{F}_{k}\mathbf{K}_{k}\mathbf{v}_{k}+\mathbf{w}_{k}.\label{eqn:forward ukf error dynamics}
\end{align}
\normalsize
Denote the true prediction covariance by  $\mathbf{P}_{k+1|k}=\mathbb{E}\left[\widetilde{\mathbf{x}}_{k+1|k}\widetilde{\mathbf{x}}_{k+1|k}^{T}\right]$. Define $\delta\mathbf{P}_{k+1|k}$ as the difference of estimated prediction covariance $\bm{\Sigma}_{k+1|k}$ and the true prediction covariance $\mathbf{P}_{k+1|k}$, while $\Delta\mathbf{P}_{k+1|k}$ is the error in the approximation of the expectation $\mathbb{E}\left[\mathbf{U}^{x}_{k}\mathbf{F}_{k}(\mathbf{I}-\mathbf{K}_{k}\mathbf{U}^{y}_{k}\mathbf{H}_{k})\widetilde{\mathbf{x}}_{k|k-1}\widetilde{\mathbf{x}}_{k|k-1}^{T}(\mathbf{I}-\mathbf{K}_{k}\mathbf{U}^{y}_{k}\mathbf{H}_{k})^{T}\mathbf{F}_{k}^{T}\mathbf{U}^{x}_{k}\right]$ by $\mathbf{U}^{x}_{k}\mathbf{F}_{k}(\mathbf{I}-\mathbf{K}_{k}\mathbf{U}^{y}_{k}\mathbf{H}_{k})\bm{\Sigma}_{k|k-1}(\mathbf{I}-\mathbf{K}_{k}\mathbf{U}^{y}_{k}\mathbf{H}_{k})^{T}\mathbf{F}_{k}^{T}\mathbf{U}^{x}_{k}$. Denoting $\hat{\mathbf{Q}}_{k}=\mathbf{Q}_{k}+\mathbf{U}^{x}_{k}\mathbf{F}_{k}\mathbf{K}_{k}\mathbf{R}_{k}\mathbf{K}_{k}^{T}\mathbf{F}_{k}^{T}\mathbf{U}^{x}_{k}+\delta\mathbf{P}_{k+1|k}+\Delta\mathbf{P}_{k+1|k}$ and using \eqref{eqn:forward ukf error dynamics} similar to \cite{xiong2006performance_ukf, li2012stochastic_ukf}, we have
\par\noindent\small
\begin{align*}
    \bm{\Sigma}_{k+1|k}&=\mathbf{U}^{x}_{k}\mathbf{F}_{k}(\mathbf{I}-\mathbf{K}_{k}\mathbf{U}^{y}_{k}\mathbf{H}_{k})\bm{\Sigma}_{k|k-1}(\mathbf{I}-\mathbf{K}_{k}\mathbf{U}^{y}_{k}\mathbf{H}_{k})^{T}\mathbf{F}_{k}^{T}\mathbf{U}^{x}_{k}\\
    &+\hat{\mathbf{Q}}_{k}.
\end{align*}
\normalsize
Similarly, we have
\par\noindent\small
\begin{align*}
    \bm{\Sigma}^{y}_{k+1}&=\mathbf{U}^{y}_{k+1}\mathbf{H}_{k+1}\bm{\Sigma}_{k+1|k}\mathbf{H}_{k+1}^{T}\mathbf{U}^{y}_{k+1}+\hat{\mathbf{R}}_{k+1},\\
    \bm{\Sigma}^{xy}_{k+1}&=\begin{cases}\bm{\Sigma}_{k+1|k}\mathbf{U}^{xy}_{k+1}\mathbf{H}_{k+1}^{T}\mathbf{U}^{y}_{k+1}, & n_{x}\geq n_{y}\\
    \bm{\Sigma}_{k+1|k}\mathbf{H}_{k+1}^{T}\mathbf{U}^{y}_{k+1}\mathbf{U}^{xy}_{k+1}, & n_{x}<n_{y}\end{cases},
\end{align*}
\normalsize
where $\hat{\mathbf{R}}_{k+1}=\mathbf{R}_{k+1}+\Delta\mathbf{P}^{y}_{k+1}+\delta\mathbf{P}^{y}_{k+1}$ with $\delta\mathbf{P}^{y}_{k+1}$ and $\Delta\mathbf{P}^{y}_{k+1}$, respectively, accounting for the difference in true and estimated measurement prediction covariances, and error in the approximation of the expectation. Also, $\mathbf{U}^{xy}_{k+1}$ is the unknown matrix introduced to account for errors in the estimated cross-covariance $\bm{\Sigma}^{xy}_{k+1}$.

The error dynamics \eqref{eqn:forward ukf error dynamics} and the various covariances have the same form as that for forward EKF given by \cite[Sec. V-B-1]{singh2022inverse_part1}. Hence, \cite[Theorem 2]{singh2022inverse_part1} is applicable for forward UKF stability as well. This is formalized in Theorem \ref{theorem:forward ukf stability} below.
\begin{theorem}[Stochastic stability of forward UKF]
\label{theorem:forward ukf stability}
Consider the forward UKF with the non-linear stochastic system given by \eqref{eqn:state transition x} and \eqref{eqn:observation y}. The forward UKF's estimation error $\widetilde{\mathbf{x}}_{k}$ is exponentially bounded in mean-squared sense and bounded with probability one if the following conditions hold true.\\
\textbf{C1.} There exist positive real numbers $\bar{f}$, $\bar{h}$, $\bar{\alpha}$, $\bar{\beta}$, $\bar{\gamma}$, $\underline{\sigma}$, $\bar{\sigma}$, $\bar{q}$, $\bar{r}$, $\hat{q}$ and $\hat{r}$ such that the following bounds are fulfilled for all $k\geq 0$.
\par\noindent\small
\begin{align*}    &\|\mathbf{F}_{k}\|\leq\bar{f},\;\;\;\|\mathbf{H}_{k}\|\leq\bar{h},\;\;\;\|\mathbf{U}^{x}_{k}\|\leq\bar{\alpha},\;\;\;\|\mathbf{U}^{y}_{k}\|\leq\bar{\beta},\;\;\;\|\mathbf{U}^{xy}_{k}\|\leq\bar{\gamma},\\        &\mathbf{Q}_{k}\preceq\bar{q}\mathbf{I},\;\;\;\mathbf{R}_{k}\preceq\bar{r}\mathbf{I},\;\;\;\hat{q}\mathbf{I}\preceq\hat{\mathbf{Q}}_{k},\;\;\;\hat{r}\mathbf{I}\preceq\hat{\mathbf{R}}_{k},\;\;\;\underline{\sigma}\mathbf{I}\preceq\bm{\Sigma}_{k|k-1}\preceq\bar{\sigma}\mathbf{I}.
\end{align*}
\normalsize
\textbf{C2.} $\mathbf{U}^{x}_{k}$ and $\mathbf{F}_{k}$ are non-singular for every $k\geq 0$.\\
\textbf{C3.} The constants satisfy the inequality $\bar{\sigma}\bar{\gamma}\bar{h}^{2}\bar{\beta}^{2}<\hat{r}$.
\end{theorem}

The UKF stability mentioned in \cite[Theorem~1]{xiong2006performance_ukf} for only linear measurements requires a lower bound on measurement noise covariance $\mathbf{R}_{k}$ and also, an upper bound on $\hat{\mathbf{Q}}_{k}$. But $\hat{\mathbf{Q}}_{k}$ is not upper-bounded in Theorem~\ref{theorem:forward ukf stability}. Our nonlinear stability guarantee requires upper (lower) bounds on noise covariances $\mathbf{Q}_{k}$ ($\hat{\mathbf{Q}}_{k}$) and $\mathbf{R}_{k}$ ($\hat{\mathbf{R}}_{k}$). Both $\hat{\mathbf{Q}}_{k}$ and $\hat{\mathbf{R}}_{k}$ can be made positive definite to satisfy the lower bounds and enhance the filter's stability by enlarging the noise covariance matrices $\mathbf{Q}_{k}$ and $\mathbf{R}_{k}$, respectively \cite{xiong2006performance_ukf,xiong2007authorreply}. For practical systems where we have a reasonable estimate of the state $\mathbf{x}_{k}$ (due to the process's constraints), the bounds on unknown matrices $\mathbf{U}^{x}_{k}$ and $\mathbf{U}^{y}_{k}$ can be estimated using functions $f(\cdot)$ and $h(\cdot)$\cite{boutayeb1999strong}.

\subsubsection{I-UKF}
We now consider the I-UKF of Section~\ref{subsec:IUKF} with state-transition \eqref{eqn:inverse ukf state transition} and observation \eqref{eqn:observation a}. Similar to the forward UKF, we introduce unknown matrices $\overline{\mathbf{U}}^{x}_{k}$ and $\overline{\mathbf{U}}^{a}_{k}$ to account for the errors in the linearization of functions $\widetilde{f}(\cdot)$ and $g(\cdot)$, respectively, and $\overline{\mathbf{U}}^{xa}_{k}$ for the errors in cross-covariance matrix estimation. Also, $\hat{\overline{\mathbf{Q}}}_{k}$ and $\hat{\overline{\mathbf{R}}}_{k}$ denote the counterparts of $\hat{\mathbf{Q}}_{k}$ and $\hat{\mathbf{R}}_{k}$, respectively, in the I-UKF's error dynamics. Define $\widetilde{\mathbf{F}}_{k}\doteq\left.\frac{\partial\widetilde{f}(\mathbf{x},\bm{\Sigma}_{k},\mathbf{x}_{k+1},\mathbf{0})}{\partial\mathbf{x}}\right\vert_{\mathbf{x}=\doublehat{\mathbf{x}}_{k}}$ and $\mathbf{G}_{k}\doteq\left.\frac{\partial g(\mathbf{x})}{\partial\mathbf{x}}\right\vert_{\mathbf{x}=\doublehat{\mathbf{x}}_{k|k-1}}$. While approximating $\bm{\Sigma}_{k}$ by $\bm{\Sigma}^{*}_{k}$ in I-UKF, we ignore any possible errors. Assume that the forward gain $\mathbf{K}_{k+1}$ computed from $\doublehat{\mathbf{x}}_{k}$ is approximately same as that computed from $\hat{\mathbf{x}}_{k}$ in forward UKF. Additionally, these approximation errors can be bounded by positive constants because $\bm{\Sigma}_{k}$ and $\mathbf{K}_{k+1}$ can be proved to be bounded matrices under the I-UKF's stability assumptions. The bounds required on various matrices for forward UKF's stability are also satisfied when these matrices are evaluated by I-UKF at its own estimates, i.e., $\left\|\frac{\partial f(\mathbf{x})}{\partial\mathbf{x}}\right\|\leq \bar{f}$ and $\left\|\frac{\partial h(\mathbf{x})}{\partial\mathbf{x}}\right\|\leq \bar{h}$ where $\mathbf{x}$ can be any sigma point.

\begin{theorem}[Stochastic stability of I-UKF]
\label{theorem:IUKF stability}
Consider the attacker's forward UKF that is stable as per Theorem~\ref{theorem:forward ukf stability}. The I-UKF's state estimation error is exponentially bounded in mean-squared sense and bounded with probability one if the following assumptions hold true.\\
 \textbf{C4.} There exist positive real numbers $\bar{g},\bar{c},\bar{d},\bar{\epsilon},\hat{c},\hat{d},\underline{p}$ and $\bar{p}$ such that the following bounds are fulfilled for all $k\geq 0$.
 \par\noindent\small
\begin{align*}
&\|\mathbf{G}_{k}\|\leq\bar{g},\;\;\|\overline{\mathbf{U}}^{a}_{k}\|\leq\bar{c},\;\;\|\overline{\mathbf{U}}^{xa}_{k}\|\leq\bar{d},\;\;\overline{\mathbf{R}}_{k}\preceq\bar{\epsilon}\mathbf{I},\;\;\hat{c}\mathbf{I}\preceq\hat{\overline{\mathbf{Q}}}_{k},\\
&\hat{d}\mathbf{I}\preceq\hat{\overline{R}}_{k},\;\;\underline{p}\mathbf{I}\preceq\overline{\bm{\Sigma}}_{k|k-1}\preceq\bar{p}\mathbf{I}.
\end{align*}
\normalsize
\textbf{C5.} There exist a real constant $\underline{y}$ (not necessarily positive) such that $\bm{\Sigma}^{y}_{k}\succeq\underline{y}\mathbf{I}$ for all $k\geq 0$.\\
\textbf{C6.} The functions $f(\cdot)$ and $h(\cdot)$ have bounded outputs i.e. $\|f(\cdot)\|_{2}\leq\delta_{f}$ and $\|h(\cdot)\|_{2}\leq\delta_{h}$ for some real positive numbers $\delta_{f}$ and $\delta_{h}$.\\
\textbf{C7.} For all $k\geq 0$, $\widetilde{\mathbf{F}}_{k}$ is non-singular and its inverse satisfies $\|\widetilde{\mathbf{F}}^{-1}_{k}\|\leq\bar{a}$ for some positive real constant $\bar{a}$.\\
\textbf{C8.} The constants satisfy the inequality $\bar{p}\bar{d}\bar{g}^{2}\bar{c}^{2}<\hat{d}$.
\end{theorem}
\begin{proof}
See Appendix~\ref{App-thm-IUKF}.
\end{proof}

While there are no constraints on the constant $\underline{y}$ in \textbf{C5}, Theorem~\ref{theorem:IUKF stability} requires an additional lower bound on $\bm{\Sigma}^{y}_{k}$ which was not needed for forward UKF's stability. Also, $\underline{y}\neq 0$ because $(\bm{\Sigma}^{y}_{k})^{-1}$ exists for forward UKF to compute its gain. Conditions \textbf{C5} and \textbf{C6} are necessary to upper-bound the Jacobian $\widetilde{\mathbf{F}}_{k}$. The bounds $\|f(\cdot)\|_{2}\leq\delta_{f}$ and $\|h(\cdot)\|_{2}\leq\delta_{h}$ help to bound the magnitude of the propagated sigma points $\lbrace\mathbf{s}^{*}_{i,k+1|k}\rbrace$ and $\lbrace\mathbf{q}^{*}_{i,k+1|k}\rbrace$, respectively, generated from a state estimate, which in turn, upper bounds the various covariance estimates. Also, the computation of gain matrix $\mathbf{K}_{k}$ involves $(\bm{\Sigma}^{y}_{k})^{-1}$, which can be upper-bounded by assuming a lower bound on $\bm{\Sigma}^{y}_{k}$ as in \textbf{C5}. Note that Theorem~\ref{theorem:IUKF stability} assumes bounded functions $f(\cdot)$ and $h(\cdot)$, but $g(\cdot)$ does not have bounded outputs in general.

\subsubsection{Conservative estimator}\label{subsubsec:consistency}
Recall the definition of a conservative estimate pair $(\hat{\mathbf{x}},\bm{\Sigma})$. Note that \cite{battistelli2014kullback} defines the same as a consistent estimator.
\begin{definition}[Conservative estimate\cite{battistelli2014kullback}]\label{def:consistency}
An unbiased estimate $\hat{\mathbf{x}}$ of random variable $\mathbf{x}$ and the corresponding error covariance estimate $\bm{\Sigma}$ are defined to be conservative if $\mathbb{E}[(\mathbf{x}-\hat{\mathbf{x}})(\mathbf{x}-\hat{\mathbf{x}})^{T}]\preceq\bm{\Sigma}$, i.e., the estimated covariance $\bm{\Sigma}$ upper bounds the true error covariance. 
\end{definition}

Following the statistical linearization technique (SLT)\cite{arasaratnam2007qkf} to linearize \eqref{eqn:inverse ukf state transition} and \eqref{eqn:observation a} with respect to I-UKF's augmented state $\mathbf{z}_{k}$ and (forward) estimate $\hat{\mathbf{x}}_{k}$, respectively, we obtain
\par\noindent\small
\begin{align}       \hat{\mathbf{x}}_{k+1}&=\mathbf{U}^{z}_{k}\overline{\mathbf{F}}^{x}_{k}\hat{\mathbf{x}}_{k}+\mathbf{U}^{z}_{k}\overline{\mathbf{F}}^{v}_{k}\mathbf{v}_{k+1},\label{eqn:SLT state transition}\\
    \mathbf{a}_{k}&=\mathbf{U}^{a}_{k}\overline{\mathbf{G}}_{k}\hat{\mathbf{x}}_{k}+\bm{\epsilon}_{k},\label{eqn:SLT observation}
\end{align}
\normalsize
where $\overline{\mathbf{F}}_{k}=[\overline{\mathbf{F}}^{x}_{k},\overline{\mathbf{F}}^{v}_{k}]$ and $\overline{\mathbf{G}}_{k}$ are the respective linear pseudo transition matrices. Also, $\mathbf{U}^{z}_{k}$ and $\mathbf{U}^{a}_{k}$ are unknown diagonal matrices introduced to account for the approximation errors in SLT. Note that these unknown matrices are different from the ones introduced in Theorem~\ref{theorem:IUKF stability} for the higher-order terms in the Taylor approximation.
\begin{theorem}
    \label{thm:consistency}
    Consider a conservative initial estimate pair $(\doublehat{\mathbf{x}}_{0},\overline{\bm{\Sigma}}_{0})$ for the I-UKF. Then for any $k\geq 1$, the I-UKF's recursive estimate pair $(\doublehat{\mathbf{x}}_{k},\overline{\bm{\Sigma}}_{k})$ is also conservative such that $\mathbb{E}[(\hat{\mathbf{x}}_{k}-\doublehat{\mathbf{x}}_{k})(\hat{\mathbf{x}}_{k}-\doublehat{\mathbf{x}}_{k})^{T}]\preceq\overline{\bm{\Sigma}}_{k}$, where $\hat{\mathbf{x}}_{k}$ is the forward UKF's state estimate.
\end{theorem}
\begin{proof}
    See Appendix~\ref{App-thm-consistency}.
\end{proof}
It follows from Theorem~\ref{thm:consistency} and Definition~\ref{def:consistency} that I-UKF is a conservative estimator if the initial estimate $\doublehat{\mathbf{x}}_{0}$ is unbiased and initial error covariance estimate $\overline{\bm{\Sigma}}_{0}$ is chosen sufficiently large. However, Theorem~\ref{thm:consistency} holds only when the SLT-based linearized models \eqref{eqn:SLT state transition} and \eqref{eqn:SLT observation} well approximate the non-linear equations \eqref{eqn:inverse ukf state transition} and \eqref{eqn:observation a}, respectively. In particular, if the unknown matrices $\mathbf{U}^{z}_{k}$ and $\mathbf{U}^{a}_{k}$ are not sufficient to account for the approximation errors, the estimates $\doublehat{\mathbf{x}}_{k}$ will be in general biased and hence, not conservative.

\vspace{-8pt}
\subsection{RKHS-UKF}
\label{subsec:RKHS-UKF stability}
Consider the following assumptions on the system dynamics and the RKHS-UKF.\\
\textbf{A1.} The kernel is a Gaussian kernel with width $\sigma>0$ such that $K(\mathbf{x}_{i},\mathbf{x}_{j})=\exp{\left(-\frac{\|\mathbf{x}_{i}-\mathbf{x}_{j}\|^{2}_{2}}{\sigma^2}\right)}$.\\
\textbf{A2.}The actual states $\mathbf{x}_{k}$ lie in a compact set $\mathcal{X}$ for all $k\geq 0$. For any state $\mathbf{x}$, $f(\mathbf{x})\in\mathcal{X}$, i.e., without any process noise, the state remains within the compact set.\\
\textbf{A3.} The dictionary $\{\widetilde{\mathbf{x}}_{l}\}_{1\leq l\leq L}$ is finite with fixed size.\\
\textbf{A4.} The true coefficient matrices $\mathbf{A}$ and $\mathbf{B}$ in \eqref{eqn:RKHS-UKF state transition approx} and \eqref{eqn:RKHS-UKF observation approx} satisfy $\|\mathbf{A}\|\leq\overline{a}$ and $\|\mathbf{B}\|\leq\overline{b}$, respectively, for some constants $\overline{a}$ and $\overline{b}$.

The conditions for a finite dictionary using the ALD criterion are discussed in \cite[Theorem~3.1]{engel2004kernel}. \textbf{A2} is essential for the Representer theorem to be valid\cite{scholkopf2001generalized}, the Gaussian kernel to approximate non-linear functions with arbitrarily small error \cite{steinwart2001influence} and also, for a finite dictionary under the ALD criterion\cite{engel2004kernel}. Also, the bound on coefficient matrices $\mathbf{A}$ and $\mathbf{B}$ in \textbf{A4} agrees with functions $f(\cdot)$ and $h(\cdot)$ being bounded which is needed for the Gaussian kernel to approximate them with small approximation errors. Under \textbf{A2}, the state transition \eqref{eqn:state transition x} is modified to $\mathbf{x}_{k+1}=\Gamma(f(\mathbf{x}_{k})+\mathbf{w}_{k})$ where $\Gamma(\cdot)$ denotes the projection operator onto the set $\mathcal{X}$. Denote the projection error by $\bm{\eta}_{k}$ such that
\par\noindent\small
\begin{align}
    \mathbf{x}_{k+1}=f(\mathbf{x}_{k})+\mathbf{w}_{k}+\bm{\eta}_{k}.\label{eqn:state transition with projection}
\end{align}
\normalsize
Note that this projection is considered only for the actual state evolution. Intuitively, this projection represents the physical constraints on the state of the process being observed. For instance, in a radar's target localization problem, the actual target location is reasonably upper-bounded by the maximum unambiguous range and beam pattern (main lobe) of the radar. The noise $\mathbf{w}_{k}$ then represents the modeling uncertainties that are assumed to be Gaussian to obtain simplified closed-form solutions\cite{ristic2003beyond}. The RKHS-UKF's state estimates are not projected onto the set $\mathcal{X}$. However, the RKHS-UKF's coefficient matrix estimates $\hat{\mathbf{A}}_{k}$ and $\hat{\mathbf{B}}_{k}$ are projected to satisfy the bounds of \textbf{A4}. We denote the approximation errors in the kernel function approximation of functions $f(\cdot)$ and $h(\cdot)$, respectively, by $\bm{\delta}_{f}(\cdot)$ and $\bm{\delta}_{h}(\cdot)$ such that
\par\noindent\small
\begin{align}
    &f(\mathbf{x})=\mathbf{A}\bm{\Phi}(\mathbf{x})+\bm{\delta}_{f}(\mathbf{x}),\label{eqn:approx error in f}\\
    &h(\mathbf{x})=\mathbf{B}\bm{\Phi}(\mathbf{x})+\bm{\delta}_{h}(\mathbf{x}).\label{eqn:approx error in h}
\end{align}
\normalsize

The state prediction error $\widetilde{\mathbf{x}}_{k+1|k}$, state estimation error $\widetilde{\mathbf{x}}_{k}$ and measurement prediction error $\widetilde{\mathbf{y}}_{k+1}$ are as earlier defined for forward UKF. 
With $\hat{\mathbf{z}}_{k+1|k}=[\hat{\mathbf{x}}_{k+1|k}^{T}\;\hat{\mathbf{x}}_{k|k}^{T}]$, using \eqref{eqn:RKHS-UKF state transition aug}, \eqref{eqn:RKHS-UKF prediction propagate} and \eqref{eqn:RKHS-UKF predicted state}, we have $\hat{\mathbf{x}}_{k+1|k}=\sum_{i=0}^{2n_{z}}\omega_{i}\hat{\mathbf{A}}_{k}\bm{\Phi}([\mathbf{s}_{i,k}]_{1:n_{x}})$. Then, substituting \eqref{eqn:state transition with projection} and \eqref{eqn:approx error in f}, we obtain $\widetilde{\mathbf{x}}_{k+1|k}=\mathbf{A}\bm{\Phi}(\mathbf{x}_{k})+\bm{\delta}_{f}(\mathbf{x}_{k})+\mathbf{w}_{k}+\bm{\eta}_{k}-\sum_{i=0}^{2n_{z}}\omega_{i}\hat{\mathbf{A}}_{k}\bm{\Phi}([\mathbf{s}_{i,k}]_{1:n_{x}})$. Now, linearizing $\bm{\Phi}(\cdot)$ at $\hat{\mathbf{x}}_{k|k}$ and introducing unknown diagonal matrix $\mathbf{U}^{\phi 1}_{k}\in\mathbb{R}^{n_{x}\times n_{x}}$ similar to \eqref{eqn:linearized x}, we have
\par\noindent\small
\begin{align}
    \widetilde{\mathbf{x}}_{k+1|k}&=(\mathbf{A}-\hat{\mathbf{A}}_{k})\bm{\Phi}(\hat{\mathbf{x}}_{k|k})+\mathbf{U}^{\phi 1}_{k}\mathbf{A}\nabla\bm{\Phi}(\hat{\mathbf{x}}_{k|k})\widetilde{\mathbf{x}}_{k|k}+\mathbf{w}_{k}\nonumber\\
    &\;\;\;+\bm{\eta}_{k}+\bm{\delta}_{f}(\mathbf{x}_{k}),\label{eqn:RKHS-UKF state prediction error}
\end{align}
\normalsize
where $\nabla\bm{\Phi}(\hat{\mathbf{x}}_{k|k})\doteq\frac{\partial\bm{\Phi}(\mathbf{x})}{\partial\mathbf{x}}\vert_{\mathbf{x}=\hat{\mathbf{x}}_{k|k}}$. Similarly, linearizing $\bm{\Phi}(\cdot)$ at $\hat{\mathbf{x}}_{k+1|k}$, introducing unknown diagonal matrix $\mathbf{U}^{\phi 2}_{k+1}\in\mathbb{R}^{n_{y}\times n_{y}}$ and using \eqref{eqn:RKHS-UKF observation aug}, \eqref{eqn:RKHS-UKF state update propagate}, \eqref{eqn:RKHS-UKF predicted y} and \eqref{eqn:approx error in h}, we obtain
\par\noindent\small
\begin{align}
    \widetilde{\mathbf{y}}_{k+1}&=(\mathbf{B}-\hat{\mathbf{B}}_{k})\bm{\Phi}(\hat{\mathbf{x}}_{k+1|k})+\mathbf{U}^{\phi 2}_{k+1}\mathbf{B}\nabla\bm{\Phi}(\hat{\mathbf{x}}_{k+1|k})\widetilde{\mathbf{x}}_{k+1|k}\nonumber\\
    &\;\;\;+\mathbf{v}_{k+1}+\bm{\delta}_{h}(\mathbf{x}_{k+1}).\label{eqn:RKHS-UKF measurement prediction error}
\end{align}
\normalsize
Denote $\mathbf{K}^{1}_{k}$ as the sub-matrix $[\mathbf{K}_{k}]_{(1:n_{x},:)}$ such that \eqref{eqn:RKHS-UKF state update} yields $\hat{\mathbf{x}}_{k|k}=\hat{\mathbf{x}}_{k|k-1}+\mathbf{K}^{1}_{k}\widetilde{\mathbf{y}}_{k}$ and $\widetilde{\mathbf{x}}_{k|k}=\widetilde{\mathbf{x}}_{k|k-1}-\mathbf{K}^{1}_{k}\widetilde{\mathbf{y}}_{k}$. Substituting for $\widetilde{\mathbf{x}}_{k|k}$ and $\widetilde{\mathbf{y}}_{k}$ (using \eqref{eqn:RKHS-UKF measurement prediction error}) in \eqref{eqn:RKHS-UKF state prediction error}, the RKHS-UKF's prediction error dynamics becomes
\par\noindent\small
\begin{align}
    \widetilde{\mathbf{x}}_{k+1|k}&=\mathbf{U}^{\phi 1}_{k}\mathbf{A}\nabla\bm{\Phi}(\hat{\mathbf{x}}_{k|k})(\mathbf{I}-\mathbf{K}^{1}_{k}\mathbf{U}^{\phi 2}_{k}\mathbf{B}\nabla\bm{\Phi}(\hat{\mathbf{x}}_{k|k-1}))\widetilde{\mathbf{x}}_{k|k-1}\nonumber\\
    &+\mathbf{w}_{k}-\mathbf{U}^{\phi 1}_{k}\mathbf{A}\nabla\bm{\Phi}(\hat{\mathbf{x}}_{k|k})\mathbf{K}^{1}_{k}\mathbf{v}_{k}+(\mathbf{A}-\hat{\mathbf{A}}_{k})\bm{\Phi}(\hat{\mathbf{x}}_{k|k})\nonumber\\
    &-\mathbf{U}^{\phi 1}_{k}\mathbf{A}\nabla\bm{\Phi}(\hat{\mathbf{x}}_{k|k})\mathbf{K}^{1}_{k}(\mathbf{B}-\hat{\mathbf{B}}_{k-1})\bm{\Phi}(\hat{\mathbf{x}}_{k|k-1})+\bm{\eta}_{k}\nonumber\\
    &+\bm{\delta}_{f}(\mathbf{x}_{k})-\mathbf{U}^{\phi 1}_{k}\mathbf{A}\nabla\bm{\Phi}(\hat{\mathbf{x}}_{k|k})\mathbf{K}^{1}_{k}\bm{\delta}_{h}(\mathbf{x}_{k}).\label{eqn:RKHS-UKF error dynamics}
\end{align}
\normalsize

Denote $\bm{\Sigma}_{k+1|k}$ as the filter's estimate of $\mathbb{E}[\widetilde{\mathbf{x}}_{k+1|k}\widetilde{\mathbf{x}}_{k+1|k}^{T}]$ and $\bm{\Sigma}^{xy}_{k+1}$ as the estimate of $\mathbb{E}[\widetilde{\mathbf{x}}_{k+1|k}\widetilde{\mathbf{y}}_{k+1}^{T}]$. These are appropriate sub-matrices of $\bm{\Sigma}^{z}_{k+1|k}$ from \eqref{eqn:RKHS-UKF sig predict} and $\bm{\Sigma}^{zy}_{k+1}$ from \eqref{eqn:RKHS-UKF sig zy predict}, respectively. Then, following similar steps as for forward UKF's stability, we obtain
\par\noindent\small
\begin{align}
    &\bm{\Sigma}_{k+1|k}=\mathbf{U}^{\phi 1}_{k}\mathbf{A}\nabla\bm{\Phi}(\hat{x}_{k|k})(\mathbf{I}-\mathbf{K}^{1}_{k}\mathbf{U}^{\phi 2}_{k}\mathbf{B}\nabla\bm{\Phi}(\hat{\mathbf{x}}_{k|k-1}))\bm{\Sigma}_{k|k-1}\nonumber\\
    &\times(\mathbf{I}-\mathbf{K}^{1}_{k}\mathbf{U}^{\phi 2}_{k}\mathbf{B}\nabla\bm{\Phi}(\hat{\mathbf{x}}_{k|k-1}))^{T}\nabla\bm{\Phi}(\hat{x}_{k|k})^{T}\mathbf{A}^{T}\mathbf{U}^{\phi 1}_{k}+\widetilde{\mathbf{Q}}_{k},\label{eqn:RKHS-UKF stability sig predict}\\
    &\bm{\Sigma}^{y}_{k+1}=\mathbf{U}^{\phi 2}_{k+1}\mathbf{B}\nabla\bm{\Phi}(\hat{\mathbf{x}}_{k+1|k})\bm{\Sigma}_{k+1|k}\nabla\bm{\Phi}(\hat{\mathbf{x}}_{k+1|k})^{T}\mathbf{B}^{T}\mathbf{U}^{\phi 2}_{k+1}\nonumber\\
    &\;\;\;+\widetilde{\mathbf{R}}_{k+1},\label{eqn:RKHS-UKF stability sig y predict}\\
    &\bm{\Sigma}^{xy}_{k+1}=\bm{\Sigma}_{k+1|k}\mathbf{U}^{xy}_{k+1}\nabla\bm{\Phi}(\hat{\mathbf{x}}_{k+1|k})^{T}\mathbf{B}^{T}\mathbf{U}^{\phi 2}_{k+1},\label{eqn:RKHS-UKF stability sig xy predict}
\end{align}
\normalsize
where\\ $\widetilde{\mathbf{Q}}_{k}=\mathbf{Q}_{k}+\mathbf{U}^{\phi 1}_{k}\mathbf{A}\nabla\bm{\Phi}(\hat{\mathbf{x}}_{k|k})\mathbf{K}^{1}_{k}\mathbf{R}_{k}(\mathbf{K}^{1}_{k})^{T}\nabla\bm{\Phi}(\hat{\mathbf{x}}_{k|k})^{T}\mathbf{A}^{T}\mathbf{U}^{\phi 1}_{k}\\+\delta\mathbf{P}_{k+1|k}+\Delta\mathbf{P}_{k+1|k}$ and $\widetilde{\mathbf{R}}_{k+1}=\mathbf{R}_{k+1}+\delta\mathbf{P}^{y}_{k+1}+\Delta\mathbf{P}^{y}_{k+1}$ with $\delta\mathbf{P}_{k+1|k}$, $\Delta\mathbf{P}_{k+1|k}$, $\delta\mathbf{P}^{y}_{k+1}$ and $\Delta\mathbf{P}^{y}_{k+1}$ defined similarly as for the UKF stability. Here, for simplicity, we have considered only the $n_{x}\geq n_{y}$ case, but the results can be trivially proved to hold for $n_{y}\geq n_{x}$ as well. Also, $\mathbf{U}^{xy}_{k+1}\in\mathbb{R}^{n_{x}\times n_{x}}$ is the unknown matrix introduced to account for errors in cross-covariance estimation. Finally, Theorem~\ref{theorem:RKHS-UKF stability} provides the stability conditions for RKHS-UKF.

\begin{theorem}[Stochastic stability of RKHS-UKF]
\label{theorem:RKHS-UKF stability}
Consider the RKHS-UKF for an unknown system model \eqref{eqn:state transition x} and \eqref{eqn:observation y} with the system satisfying \textbf{A1}-\textbf{A4}. The coefficient matrices estimates are projected as in \eqref{eqn:RKHS-UKF A estimate} and \eqref{eqn:RKHS-UKF B estimate} such that the estimates also satisfy the bounds in \textbf{A4}. The RKHS-UKF's estimation error $\widetilde{\mathbf{x}}_{k}$ is exponentially bounded in the mean-squared sense if the following hold true.\\
\textbf{C9.} There exist positive real numbers $\underline{\sigma}, \overline{\sigma}, \overline{\gamma}, \overline{\beta}, \widetilde{r}, \overline{\alpha}, \widetilde{q}, \overline{q}, \overline{r}, \overline{\phi}, \overline{f}$ and $\overline{h}$ such that the following bounds are fulfilled for all $k\geq 0$:
\par\noindent\small
\begin{align*}
    &\underline{\sigma}\mathbf{I}\preceq\bm{\Sigma}_{k|k-1}\preceq\overline{\sigma}\mathbf{I},\;\;\|\mathbf{U}^{\phi 1}_{k}\|\leq\overline{\alpha},\;\;\|\mathbf{U}^{\phi 2}_{k}\|\leq\overline{\beta},\;\;\|\mathbf{U}^{xy}_{k}\|\leq\overline{\gamma},\\ &\mathbf{Q}_{k}\preceq\overline{q}\mathbf{I},\;\;\mathbf{R}_{k}\leq\overline{r}\mathbf{I},\;\;\widetilde{\mathbf{Q}}_{k}\succeq\widetilde{q}\mathbf{I},\;\;\widetilde{\mathbf{R}}_{k}\succeq\widetilde{r}\mathbf{I},\;\;\|\nabla\bm{\Phi}(\cdot)\|\leq\overline{\phi},\\
    &\|\bm{\delta}_{f}(\cdot)\|_{2}\leq\overline{f},\;\;\|\bm{\delta}_{h}(\cdot)\|_{2}\leq\overline{h}.
\end{align*}
\normalsize
\textbf{C10.} The Jacobian $\nabla\bm{\Phi}(\mathbf{x})$ at any state $\mathbf{x}$, the actual coefficient matrix $\mathbf{A}$ and $\mathbf{U}^{\phi 1}_{k}$ are non-singular for every $k\geq 0$.\\
\textbf{C11.} The constants satisfy the inequality $\overline{\sigma}\overline{\gamma}\overline{\phi}^{2}\overline{\beta}^{2}\overline{b}^{2}<\widetilde{r}$.
\end{theorem}
\begin{proof}
    See Appendix~\ref{App-thm-RKHS-UKF}.
\end{proof}
Note that for the Gaussian kernel $K(\cdot,\cdot)$, the Jacobian at state estimate $\hat{\mathbf{x}}$ is given as $\nabla\bm{\Phi}(\hat{\mathbf{x}})=\frac{2}{\sigma^{2}}\begin{bsmallmatrix}
    K(\widetilde{\mathbf{x}}_{1},\mathbf{x})(\widetilde{\mathbf{x}}_{1}-\hat{\mathbf{x}})^{T}\\
    \vdots\\
    K(\widetilde{\mathbf{x}}_{L},\mathbf{x})(\widetilde{\mathbf{x}}_{L}-\hat{\mathbf{x}})^{T}
    \end{bsmallmatrix}$. Since $K(\cdot,\cdot)$ is a bounded function,  $\|\nabla\bm{\Phi}(\cdot)\|\leq\overline{\phi}$ implies $\widetilde{\mathbf{x}}_{l}-\hat{\mathbf{x}}$ is bounded where $\hat{\mathbf{x}}$ is any state estimate and $\widetilde{\mathbf{x}}_{l}$ is a dictionary element obtained from the state estimates themselves. Hence, this assumption implies our computed state estimates (more specifically, the difference between these estimates) lie in a bounded set. However, this set need not be the same as the compact set $\mathcal{X}$ of the true states in \textbf{A2}.

\section{Numerical Experiments}
\label{sec:numerical}
We demonstrate the performance of the proposed filters by comparing the estimation error with RCRLB for different example systems. Additionally, we compare the relative performance of the inverse filters with the corresponding forward filters. Note that, in general, a non-linear filter's performance also depends on the system itself, such that choosing an appropriate filter involves a trade-off between accuracy and computational efforts\cite{li2017approximate}. The same argument also holds for non-linear inverse filters. RCRLB provides a lower bound on mean-squared error (MSE) for discrete-time non-linear filtering as $\mathbb{E}\left[(\mathbf{x}_{k}-\hat{\mathbf{x}}_{k})(\mathbf{x}_{k}-\hat{\mathbf{x}}_{k})^{T}\right]\succeq\mathbf{J}_{k}^{-1}$ where $\mathbf{J}_{k}=\mathbb{E}\left[-\frac{\partial^{2}\ln{p(Y^{k},X^{k})}}{\partial\mathbf{x}_{k}^{2}}\right]$ is the Fisher information matrix\cite{tichavsky1998posterior}. Here, $X^{k}=\lbrace\mathbf{x}_{0},\mathbf{x}_{1},\hdots,\mathbf{x}_{k}\rbrace$ is the state vector series while $Y^{k}=\lbrace\mathbf{y}_{0},\mathbf{y}_{1},\hdots,\mathbf{y}_{k}\rbrace$ are the noisy observations. Also, $p(Y^{k},X^{k})$ is the joint probability density of pair $(Y^{k},X^{k})$ and $\hat{\mathbf{x}}_{k}$ is an estimate of $\mathbf{x}_{k}$ with $\frac{\partial^{2}(\cdot)}{\partial\mathbf{x}^{2}}$ denoting the Hessian.

The information matrix $\mathbf{J}_{k}$ is computed recursively as $\mathbf{J}_{k}=\mathbf{D}_{k}^{22}-\mathbf{D}_{k}^{21}(\mathbf{J}_{k-1}+\mathbf{D}_{k}^{11})^{-1}\mathbf{D}_{k}^{12}$, where $\mathbf{D}_{k}^{11}=\mathbb{E}\left[-\frac{\partial^{2}\ln{p(\mathbf{x}_{k}\vert\mathbf{x}_{k-1})}}{\partial\mathbf{x}_{k-1}^{2}}\right]$, $\mathbf{D}_{k}^{12}=\mathbb{E}\left[-\frac{\partial^{2}\ln{p(\mathbf{x}_{k}\vert\mathbf{x}_{k-1})}}{\partial\mathbf{x}_{k}\partial\mathbf{x}_{k-1}}\right]=(\mathbf{D}_{k}^{21})^{T}$ and $\mathbf{D}_{k}^{22}=\mathbb{E}\left[-\frac{\partial^{2}\ln{p(\mathbf{x}_{k}\vert\mathbf{x}_{k-1})}}{\partial\mathbf{x}_{k}^{2}}\right]+\mathbb{E}\left[-\frac{\partial^{2}\ln{p(\mathbf{y}_{k}\vert\mathbf{x}_{k})}}{\partial\mathbf{x}_{k}^{2}}\right]$ \cite{tichavsky1998posterior}. For the non-linear system given by \eqref{eqn:state transition x} and \eqref{eqn:observation y}, the forward information matrices $\lbrace\mathbf{J}_{k}\rbrace$ recursions are $\mathbf{J}_{k+1}=\mathbf{H}_{k+1}^{T}\mathbf{R}_{k+1}^{-1}\mathbf{H}_{k+1}-\mathbf{Q}_{k}^{-1}\mathbf{F}_{k}(\mathbf{J}_{k}+\mathbf{F}_{k}^{T}\mathbf{Q}_{k}^{-1}\mathbf{F}_{k})^{-1}\mathbf{F}_{k}^{T}\mathbf{Q}_{k}^{-1}+\mathbf{Q}_{k}^{-1}$, where $\mathbf{F}_{k}=\frac{\partial f(\mathbf{x})}{\partial\mathbf{x}}\vert_{\mathbf{x}=\mathbf{x}_{k}}$ and $\mathbf{H}_{k}=\frac{\partial h(\mathbf{x})}{\partial\mathbf{x}}\vert_{\mathbf{x}=\mathbf{x}_{k}}$\cite{xiong2006performance_ukf}. These recursions can be trivially extended to compute the information matrix $\overline{\mathbf{J}}_{k}$ for inverse filter's estimate $\doublehat{\mathbf{x}}_{k}$. Throughout all experiments, the initial information matrices $\mathbf{J}_{0}$ and $\overline{\mathbf{J}}_{0}$ were set to $\bm{\Sigma}_{0}^{-1}$ and $\overline{\bm{\Sigma}}_{0}^{-1}$, respectively.

Besides achieving RCRLB, an estimator also needs to be credible, i.e., its estimated error covariance $\bm{\Sigma}$ is statistically close to the actual MSE matrix $\mathbf{P}$. In \cite{li2001practical}, averaged normalized estimation error squared (ANEES) and non-credibility indices (NCI) have been proposed as credibility measures. However, NCI is preferable for comparing different estimators because it penalizes optimism and pessimism to the same degree. An optimistic (pessimistic) estimator's $\bm{\Sigma}$ is statistically smaller (larger) than $\mathbf{P}$. Define $NCI=(10/M)\sum_{m=1}^{M}\log_{10}(\epsilon_{m}/\epsilon^{*}_{m})$ where $\epsilon_{m}=\widetilde{\mathbf{x}}_{m}^{T}\bm{\Sigma}_{m}^{-1}\widetilde{\mathbf{x}}_{m}$ and $\epsilon^{*}_{m}=\widetilde{\mathbf{x}}_{m}^{T}\mathbf{P}_{m}^{-1}\widetilde{\mathbf{x}}_{m}$ with $M$ as the total number of Monte-Carlo independent runs. Here, $\widetilde{\mathbf{x}}_{m}=\mathbf{x}_{m}-\hat{\mathbf{x}}_{m}$ is the estimation error at $m$-th run for actual state $\mathbf{x}_{m}$ and its estimate $\hat{\mathbf{x}}_{m}$. A perfect NCI is $0$ while positive (negative) NCI represents optimism (pessimism).

\vspace{-8pt}
\subsection{FM demodulation with I-UKF}
\label{subsec: FM demod}
Consider the discrete-time non-linear system model of FM demodulator \cite[Sec. 8.2]{anderson2012optimal}
\par\noindent\small
\begin{align*}
&\mathbf{x}_{k+1}\doteq\begin{bsmallmatrix}\lambda_{k+1}\\\theta_{k+1}\end{bsmallmatrix}=\begin{bsmallmatrix}\exp{(-T/\beta)}&0\\-\beta \exp{(-T/\beta)}-1&1\end{bsmallmatrix}\begin{bsmallmatrix}\lambda_{k}\\\theta_{k}\end{bsmallmatrix}+\begin{bsmallmatrix}1\\-\beta\end{bsmallmatrix}w_{k},
\end{align*}
\normalsize
with observation $\mathbf{y}_{k}=\sqrt{2}[\sin{\theta_{k}},\cos{\theta_{k}}]^{T}+\mathbf{v}_{k}$ and $a_{k}=\hat{\lambda}_{k}^{2}+\epsilon_{k}$. Here, $w_{k}\sim\mathcal{N}(0,0.01)$, $\mathbf{v}_{k}\sim\mathcal{N}(\mathbf{0},\mathbf{I}_{2})$, $\epsilon_{k}\sim\mathcal{N}(0,5)$, $T=2\pi/16$ and $\beta=100$. Also, $\hat{\lambda}_{k}$ is the forward filter's estimate of $\lambda_{k}$. For this system, EKF is observed to be more accurate than UKF in the forward filtering case. Here, we compare I-UKF's and I-EKF's performance. For forward and inverse UKF, $\kappa$ and $\overline{\kappa}$ both were set to $1$, but I-UKF assumed the forward UKF's $\kappa$ to be $2$. Other parameters and initial estimates were as in \cite{singh2022inverse_part1}.

Fig.~\ref{fig:fmdemod iukf} shows the time-averaged root MSE (RMSE), RCRLB and NCI for state estimation for forward and inverse UKF and EKF averaged over 500 runs. The RCRLB value for state estimation is $\sqrt{\textrm{Tr}(\mathbf{J}^{-1})}$ with $\mathbf{J}$ denoting the associated information matrix. We also consider inverse filters with a forward filter that is not the same as the true forward filter. For instance, IUKF-E in Fig.~\ref{fig:fmdemod iukf}a denotes the I-UKF's estimation error which assumes a forward UKF when the true forward filter is EKF. The other notations in Fig.~\ref{fig:fmdemod iukf} and also, in further experiments are similarly defined. From Fig.~\ref{fig:fmdemod iukf}, we observe that the forward EKF has lower estimation error but higher NCI than forward UKF. Hence, with correct forward filter assumption, IUKF-U has a higher error than IEKF-E. However, IUKF-E outperforms I-EKF even with incorrect forward filter assumption. On the other hand, incorrect forward filter assumption degrades I-EKF's performance, i.e., IEKF-U has lower estimation accuracy and higher NCI than IEKF-E. While all filters considered here are optimistic, I-UKF is the most credible filter. I-UKF also outperforms forward UKF in terms of credibility because it uses additional true state $\mathbf{x}_{k}$ information. Note that even though I-UKF assumes forward UKF's $\kappa$ to be different from its true value, IUKF-U performs better than the forward UKF.
\begin{figure}
  \centering
  \includegraphics[width = 1.0\columnwidth]{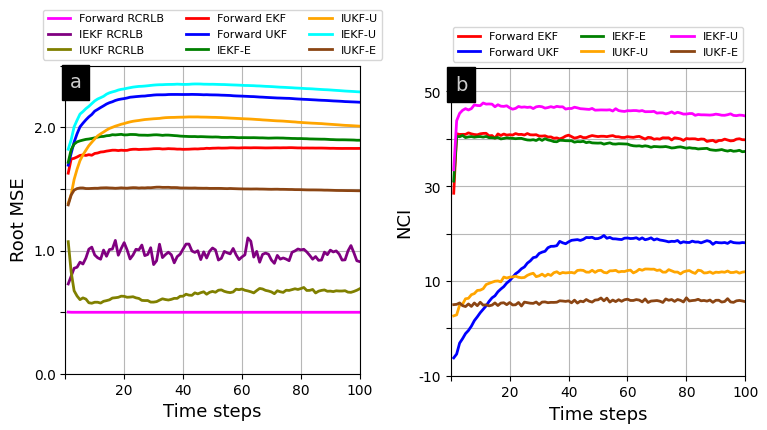}
  \caption{(a) Time-averaged RMSE and RCRLB, and (b) NCI for forward and inverse UKF for FM demodulator system.\vspace{-15pt}}
 \label{fig:fmdemod iukf}
\end{figure}

\subsection{Vehicle reentry with I-UKF}
\label{subsec: vehicle reentry}
Consider a radar tracking a vehicle's reentry using range and bearing measurements, widely used to illustrate UKF's performance \cite{julier2004unscented,sarkka2007unscented}. Here, we consider I-UKF's performance under both correct and incorrect forward filter assumptions. We denote the position of the vehicle at $k$-th time instant as $[\mathbf{x}_{k}]_{1}$ and $[\mathbf{x}_{k}]_{2}$, its velocity as $[\mathbf{x}_{k}]_{3}$ and $[\mathbf{x}_{k}]_{4}$, and its constant aerodynamic parameter as $[\mathbf{x}_{k}]_{5}$. The vehicle's state continuous-time evolution follows $[\dot{\mathbf{x}}_{k}]_{1}=[\mathbf{x}_{k}]_{3}$, $[\dot{\mathbf{x}}_{k}]_{2}=[\mathbf{x}_{k}]_{4}$, $[\dot{\mathbf{x}}_{k}]_{3}=d_{k}[\mathbf{x}_{k}]_{3}+g_{k}[\mathbf{x}_{k}]_{1}+w_{1}$, $[\dot{\mathbf{x}}_{k}]_{4}=d_{k}[\mathbf{x}_{k}]_{4}+g_{k}[\mathbf{x}_{k}]_{2}+w_{2}$, $[\dot{\mathbf{x}}_{k}]_{5}=w_{3}$, where $[\dot{\mathbf{x}}_{k}]_{i}$ is the first-order partial derivative of $[\mathbf{x}_{k}]_{i}$ with respect to time, and $w_{1}, w_{2}$ and $w_{3}$ represent process noise. We consider the discretized version of this system with a time step of $0.1$ sec in our experiment. The quantities $d_{k}=\beta_{k}\exp{(\left(\rho_{0}-\rho_{k})/h_{0}\right)}V_{k}$ and $g_{k}=-Gm_{0}\rho_{k}^{-3}$ where $\beta_{k}=\beta_{0}\exp{([\mathbf{x}_{k}]_{5})}$, $V_{k}=\sqrt{[\mathbf{x}_{k}]_{3}^{2}+[\mathbf{x}_{k}]_{4}^{2}}$ and $\rho_{k}=\sqrt{[\mathbf{x}_{k}]_{1}^{2}+[\mathbf{x}_{k}]_{2}^{2}}$ with $\rho_{0}$, $h_{0}$, $G$, $m_{0}$ and $\beta_{0}$ as constants. The radar's range and bearing measurements are $[\mathbf{y}_{k}]_{1}=\sqrt{([\mathbf{x}_{k}]_{1}-\rho_{0})^{2}+[\mathbf{x}_{k}]_{2}^{2}}+v_{1}$, and $[\mathbf{y}_{k}]_{2}=\tan^{-1}{\left(\frac{[\mathbf{x}_{k}]_{2}}{[\mathbf{x}_{k}]_{1}-\rho_{0}}\right)}+v_{2}$, where $v_{1}$ and $v_{2}$ represent measurement noises \cite{julier2004unscented}. 

For the inverse filter, we consider a linear observation 
    $\mathbf{a}_{k}=\left[[\hat{\mathbf{x}}_{k}]_{1},[\hat{\mathbf{x}}_{k}]_{2}\right]^{T}+\bm{\epsilon}_{k}$, 
where $\bm{\epsilon}_{k}\sim\mathcal{N}(\mathbf{0},3\mathbf{I}_{2})$. The initial state was $\mathbf{x}_{0}=[6500.4,349.14,-1.8093,-6.7967,0.6932]^{T}$. The initial state estimate $\Hat{\Hat{\mathbf{x}}}_{k}$ for I-UKF was set to actual $\mathbf{x}_{0}$ with initial covariance estimate $\overline{\bm{\Sigma}}_{0}=diag(10^{-5},10^{-5},10^{-5},10^{-5},1)$. For forward UKF, $\kappa$ was chosen as $2.5$ such that the weight for $0$-th sigma point at $\hat{\mathbf{x}}_{k}$ is $1/3$, and all other sigma points have equal weights. Similarly, $\overline{\kappa}$ of I-UKF was set to $3.5$. Here, we assumed that the forward UKF's $\kappa$ was perfectly known to I-UKF. 
All other system parameters and initial estimates were identical to \cite{julier2004unscented}.

Fig. \ref{fig:reentry iukf} shows the (root) time-averaged error in position estimation, its RCRLB (also, time-averaged) and NCI for forward and inverse UKF (IUKF-U), including forward EKF and IUKF-E which incorrectly assumes the forward filter to be UKF when the adversary's actual forward filter is EKF. Here, the RCRLB is computed as $\sqrt{[\mathbf{J}^{-1}]_{1,1}+[\mathbf{J}^{-1}]_{2,2}}$. The I-UKF's error are observed to be lower than that of forward UKF, as is the case with their corresponding RCRLBs. I-UKF's NCI is approximately $0$ (perfect NCI) while forward UKF and EKF are pessimistic. Further, incorrect forward filter assumption (IUKF-E case) does not affect the I-UKF's estimation because forward UKF and EKF have similar performances. For the vehicle re-entry example, the I-EKF's error and NCI were similar to I-UKF and hence, omitted in Fig. \ref{fig:reentry iukf}.
\begin{figure}
  \centering
  \includegraphics[width = 1.0\columnwidth]{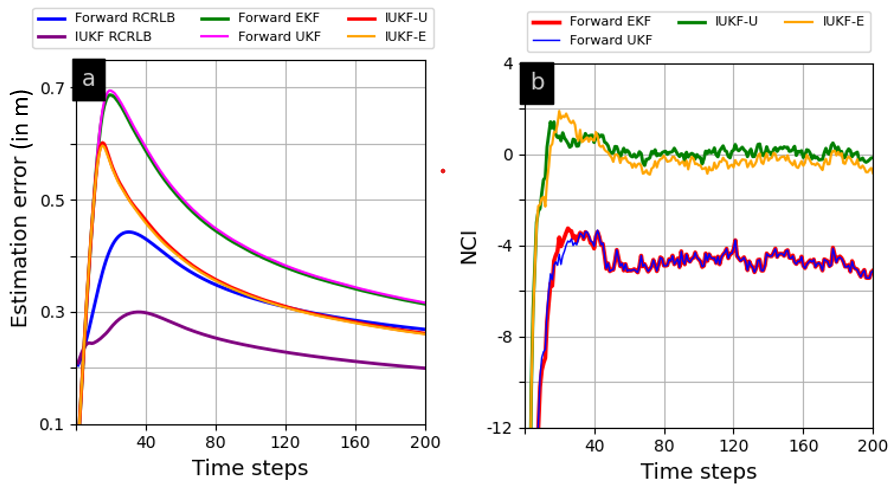}
  \caption{(a) Time-averaged estimation error and RCRLB, and (b) NCI for forward and inverse UKF for vehicle reentry system averaged over $100$ runs.\vspace{-15pt}}
 \label{fig:reentry iukf}
\end{figure}

\subsection{FM demodulation with RKHS-UKF}
\label{subsec: rkhs-ukf fm demod}
Recall the FM demodulator system and compare RKHS-UKF's estimation accuracy with forward and inverse EKF (with perfect system model information) and RKHS-EKF\cite{singh2022inverse_part2} (without any prior information). Similar to \cite[VI-C]{singh2022inverse_part2}, the attacker knows its observation function and employs a simplified RKHS-UKF/EKF. The defender learns both its state evolution and observation model using the general RKHS-UKF/EKF as its inverse filter (denoted by I-RKHS-UKF/EKF). Note that forward RKHS-UKF/EKF and I-RKHS-UKF/EKF are essentially the same algorithms but employed by different agents to compute their desired estimates. Considering a Gaussian kernel, all the parameters of the forward and inverse RKHS-based filters were the same as in \cite[VI-C]{singh2022inverse_part2}. For RKHS-UKF, $\kappa$ was set to $5$ and $3$, respectively, for the forward and inverse filters.

Fig.~\ref{fig:rkhs fmdemod} shows the time-averaged RMSE and NCI for forward and inverse EKF, RKHS-EKF, and RKHS-UKF. We omit the forward and inverse UKF performances, which were observed to be less accurate than forward and inverse EKF in Section~\ref{subsec: FM demod}. For I-EKF and I-RKHS-EKF, the true forward filters are EKF and RKHS-EKF, respectively. On the other hand, I-RKHS-UKF-1, I-RKHS-UKF-2, and I-RKHS-UKF-3 have RKHS-UKF, EKF, and RKHS-EKF, respectively, as true forward filters. From Fig.~\ref{fig:rkhs fmdemod}a, we observe that both forward RKHS-EKF and RKHS-UKF have higher errors than forward EKF because they do not perfectly know the defender's state evolution function. On the other hand, both forward RKHS-EKF and RKHS-UKF are more credible than forward EKF, but optimistic and pessimistic, respectively. However, forward RKHS-UKF is more accurate than RKHS-EKF. While all RKHS-based inverse filters perform better than I-EKF without any prior system model information, I-RKHS-UKF outperforms both I-EKF and I-RKHS-EKF when estimating the same state, i.e., I-RKHS-UKF-2 and I-RKHS-UKF-3 have lower error than I-EKF and I-RKHS-EKF, respectively. The improved performance for the RKHS-based inverse filters owes to the fact that we use a kernel function approximation in RKHS-based filters instead of the first-order Taylor approximation as in I-EKF for the non-linear functions. However, RKHS-EKF linearizes the non-linear kernel function for expectation computations and thus, introduces linearization errors. RKHS-UKF, on the other hand, considers the unscented transform and hence, also outperforms RKHS-EKF. Interestingly, even though RKHS-UKF is pessimistic, it has a similar NCI for all cases, including when employed as a forward filter (forward RKHS-UKF) and hence, is more robust than all other filters.
\begin{figure}
  \centering
  \includegraphics[width = 1.0\columnwidth]{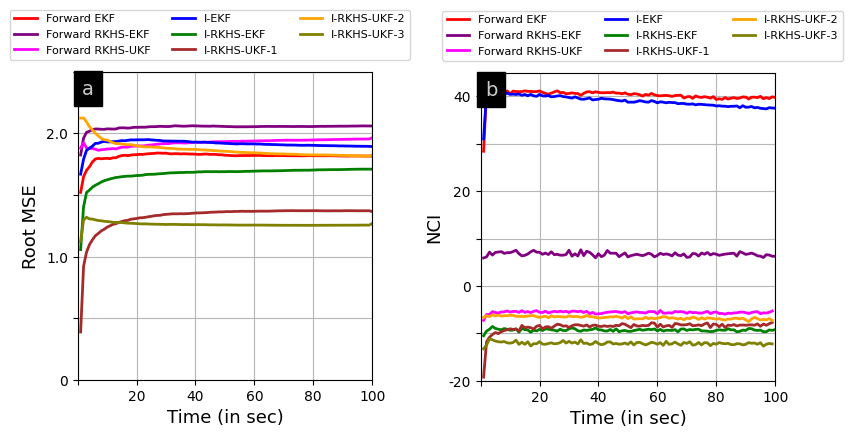}
  \caption{(a) Time-averaged RMSE, and (b) NCI for forward and inverse RKHS-UKF for FM demodulator system, compared with forward and inverse EKF and RKHS-EKF.\vspace{-15pt}}
 \label{fig:rkhs fmdemod}
\end{figure}

\subsection{Lorenz system with RKHS-UKF}
\label{subsec: rkhs-ukf lorenz system}
Consider the $3$-dimensional Lorenz system\cite{ito2000gaussian}
\par\noindent\small
\begin{align*}
&\mathbf{x}_{k+1}=\begin{bsmallmatrix}
    [\mathbf{x}_{k}]_{1}+\Delta tr_{1}(-[\mathbf{x}_{k}]_{1}+[\mathbf{x}_{k}]_{2})\\
    [\mathbf{x}_{k}]_{2}+\Delta t(r_{2}[\mathbf{x}_{k}]_{1}-[\mathbf{x}_{k}]_{2}-[\mathbf{x}_{k}]_{1}[\mathbf{x}_{k}]_{3})\\
    [\mathbf{x}_{k}]_{3}+\Delta t(-r_{3}[\mathbf{x}_{k}]_{3}+[\mathbf{x}_{k}]_{1}[\mathbf{x}_{k}]_{2})
\end{bsmallmatrix}+\begin{bsmallmatrix}
    0\\0\\0.5
\end{bsmallmatrix}w_{k},
\end{align*}
\normalsize
with $y_{k}=\Delta t\sqrt{([\mathbf{x}_{k}]_{1}-0.5)^{2}+[\mathbf{x}_{k}]_{2}^{2}+[\mathbf{x}_{k}]_{3}^{2}}+0.065v_{k}$ and $a_{k}=\Delta t\sqrt{[\hat{\mathbf{x}}_{k}]_{1}^{2}+([\hat{\mathbf{x}}_{k}]_{2}-0.5)^{2}+[\hat{\mathbf{x}}_{k}]_{3}^{2}}+0.1\epsilon_{k}$, where $w_{k}, v_{k}, \epsilon_{k}\sim\mathcal{N}(0,\Delta t)$ with parameters $\Delta t=0.01$, $r_{1}=10$, $r_{2}=28$ and $r_{3}=8/3$. Here, we compare forward and inverse UKF (with perfect system model information) with RKHS-UKF (without any prior information). The considered system is mathematically interesting because of its three unstable equilibria\cite{ito2000gaussian}. For the attacker's state estimate, we employ a forward RKHS-UKF but unlike the FM demodulator application, the observation function is not known to the attacker. For both forward and inverse RKHS-UKF, we chose $\kappa=3$ and Gaussian kernel's width $\sigma^{2}=20$. The dictionaries were constructed using the sliding window criterion with a window length $15$. The initial coefficient matrix estimates ($\hat{\mathbf{A}}_{0}, \hat{\mathbf{B}}_{0}$) and noise covariance matrix estimates ($\hat{\mathbf{Q}}_{0},\hat{\mathbf{R}}_{0}$) were set to appropriate (size) all ones and identity matrices, respectively. We compare the RKHS-UKF-based forward and inverse filters with a forward UKF with $\kappa=1.5$ and I-UKF with $\overline{\kappa}=2$, respectively. We initialize $\mathbf{x}_{0}=[-0.2,-0.3,-0.5]^{T}=\doublehat{\mathbf{x}}_{0}$ and $\hat{\mathbf{x}}_{0}=[1.35,-3,6]^{T}$, while all initial covariances were set to $0.35\mathbf{I}$.

Fig.~\ref{fig:rkhs lorenz} shows the time-averaged RMSE and NCI for forward and inverse RKHS-UKF and UKF, including the incorrect forward filter assumption case I-UKF-R which assumes a forward UKF instead of the actual forward RKHS-UKF. The true forward filters for I-RKHS-UKF and I-RKHS-UKF-U are RKHS-UKF and UKF, respectively. Similar to Fig.~\ref{fig:rkhs fmdemod}a, we observe that without perfect system information, forward RKHS-UKF has a higher estimation error than the forward UKF. While I-UKF has the same accuracy as its forward UKF, I-RKHS-UKF (I-RKHS-UKF-U case) could not achieve the same accuracy by learning the system model on its own. However, even though forward RKHS-UKF is not accurate, I-RKHS-UKF is able to learn its state estimate's evolution and has the lowest estimation error. Interestingly, I-UKF-R with incorrect forward filter assumption also shows a lower error than its forward RKHS-UKF. From Fig.~\ref{fig:rkhs lorenz}b, we observe that only I-RKHS-UKF has stable performance in terms of credibility. I-RKHS-UKF-U is slightly optimistic while I-RKHS-UKF is pessimistic.
\begin{figure}
  \centering
  \includegraphics[width = 1.0\columnwidth]{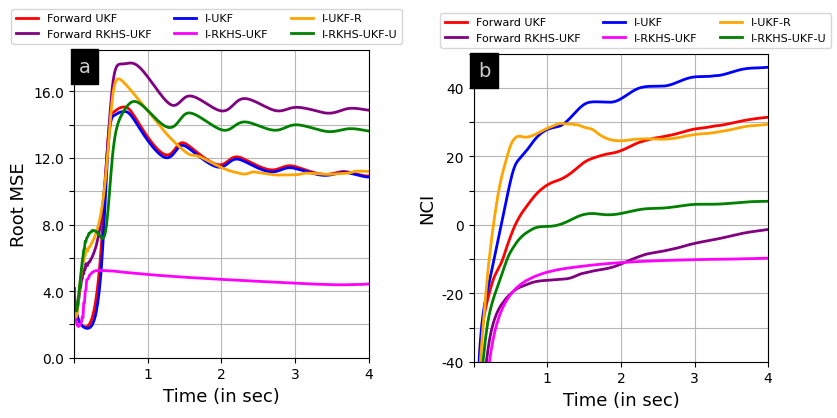}
  \caption{(a) Time-averaged RMSE, and (b) Time-averaged NCI for forward and inverse RKHS-UKF for Lorenz system, compared with forward and inverse UKF, averaged over $50$ runs.\vspace{-15pt}}
 \label{fig:rkhs lorenz}
\end{figure}

\section{Summary}
\label{sec:summary}
For the inverse cognition problem in counter-adversarial systems, we developed unscented transform-based inverse non-linear filters. Our basic I-UKF assumes perfect system information. We also suitably modify it for other general scenarios of non-Gaussian noises, continuous-time state evolution, and complex-valued systems. Our theoretical guarantees state that, if the attacker's forward UKF is stable, then the I-UKF is also stable under mild conditions. Numerical results suggested that the inverse filters provide reasonable estimates, even while incorrectly assuming the form of the forward filter. I-UKF also outperforms forward UKF because of the perfect actual state information. When the prior system information is not available, we proposed RKHS-UKF using the kernel function approximation and coupling an online approximate EM with UKF recursions. The RKHS-UKF was shown to provide better estimates than RKHS-EKF.

\appendices
\section{Proof of Theorem~\ref{theorem:IUKF stability}}
\label{App-thm-IUKF}
We first obtain stability results for the augmented state UKF with \emph{non-additive process noise} (from which I-UKF was formulated) in Appendix~\ref{subsec:augmented state ukf stability}. In Appendix~\ref{subsec:preliminaries}, we provide some preliminary results, including a bound on Jacobian $\widetilde{\mathbf{F}}_{k}$ of state transition \eqref{eqn: IUKF state transition detail} with respect to $\hat{\mathbf{x}}_{k}$. The I-UKF's stability then follows in Appendix~\ref{subsec:proof of IUKF theorem} while \ref{subsec:app-thm-intermediate} provides detailed proofs of Claims~\ref{claim:IUKF stable bound on cholesky derivative} and \ref{claim:IUKF stable K j row bound}. We restate a useful lemma for bounding a stochastic process.

\begin{lemma}[Boundedness of stochastic process {\cite[Lemma 2.1]{reif1999stochastic}}]
\label{lemma:exponential boundedness}
Consider a function $V_{k}(\bm{\zeta}_{k})$ of the stochastic process $\bm{\zeta}_{k}$ and real numbers $v_{\textrm{min}}$, $v_{\textrm{max}}$, $\mu>0$, and $0<\lambda\leq 1$ such that \textbf{(a)} $v_{\textrm{min}}\|\bm{\zeta}_{k}\|_{2}^{2}\leq V_{k}(\bm{\zeta}_{k})\leq v_{\textrm{max}}\|\bm{\zeta}_{k}\|_{2}^{2}$, and \textbf{(b)} $\mathbb{E}\left[ V_{k+1}(\bm{\zeta}_{k+1})|\bm{\zeta}_{k}\right]-V_{k}(\bm{\zeta}_{k})\leq\mu-\lambda V_{k}(\bm{\zeta}_{k})$ for all $k\geq 0$. Then, the stochastic process $\{\bm{\zeta}_{k}\}_{k \geq 0}$ is exponentially bounded in mean-squared sense, i.e., $\mathbb{E}\left[\|\bm{\zeta}_{k}\|_{2}^{2}\right]\leq\frac{v_{\textrm{max}}}{v_{\textrm{min}}}\mathbb{E}\left[\|\bm{\zeta}_{0}\|_{2}^{2}\right](1-\lambda)^{k}+\frac{\mu}{v_{\textrm{min}}}\sum_{i=1}^{k-1}(1-\lambda)^{i}$ for every $k\geq 0$. Further, $\{\bm{\zeta}_{k}\}_{k \geq 0}$ is also bounded with probability one.
\end{lemma}
\vspace{-8pt}
\subsection{Stochastic stability of augmented state UKF}
\label{subsec:augmented state ukf stability}
Consider state-evolution \eqref{eqn:state transition x} with non-additive process noise as
\par\noindent\small
\begin{align}
    \mathbf{x}_{k+1}=f(\mathbf{x}_{k},\mathbf{w}_{k}).\label{eqn:augmented ukf state transition}
\end{align}
\normalsize
The non-additive noise term leads us to formulate UKF using an augmented state $\mathbf{z}_{k}=[\mathbf{x}_{k}^{T},\mathbf{w}_{k}^{T}]^{T}$ to estimate $\hat{\mathbf{x}}_{k}$ as described in \cite{wan2000unscented}. Linearizing $f(\cdot)$ at $\hat{\mathbf{z}}_{k}=[\hat{\mathbf{x}}_{k}^{T},\mathbf{0}]^{T}$, the state prediction error $\widetilde{\mathbf{x}}_{k+1|k}$ is approximated as $\widetilde{\mathbf{x}}_{k+1|k}\approx\mathbf{F}_{k}\widetilde{\mathbf{x}}_{k}+\mathbf{F}^{w}_{k}\mathbf{w}_{k}$, where $\mathbf{F}_{k}\doteq\frac{\partial f(\mathbf{x},\mathbf{0})}{\partial\mathbf{x}}\vert_{\mathbf{x}=\hat{\mathbf{x}}_{k}}$ and $\mathbf{F}^{w}_{k}\doteq\frac{\partial f(\hat{\mathbf{x}}_{k},\mathbf{w})}{\partial\mathbf{w}}\vert_{\mathbf{w}=\mathbf{0}}$. Similar to forward UKF, we introduce unknown diagonal matrices $\mathbf{U}^{x}_{k}$ and $\mathbf{U}^{w}_{k}\in\mathbb{R}^{n_{x}\times n_{x}}$ to account for the linearization errors as
\par\noindent\small
\begin{align}
    \widetilde{\mathbf{x}}_{k+1|k}=\mathbf{U}^{x}_{k}\mathbf{F}_{k}\widetilde{\mathbf{x}}_{k}+\mathbf{U}^{w}_{k}\mathbf{F}^{w}_{k}\mathbf{w}_{k}.\label{eqn:augmented state ukf linearization}
\end{align}
\normalsize
Hence, the prediction error dynamics becomes
\par\noindent\small
\begin{align}
    \widetilde{\mathbf{x}}_{k+1|k}&=\mathbf{U}^{x}_{k}\mathbf{F}_{k}(\mathbf{I}-\mathbf{K}_{k}\mathbf{U}^{y}_{k}\mathbf{H}_{k})\widetilde{\mathbf{x}}_{k|k-1}-\mathbf{U}^{x}_{k}\mathbf{F}_{k}\mathbf{K}_{k}\mathbf{v}_{k}+\mathbf{U}^{w}_{k}\mathbf{F}^{w}_{k}\mathbf{w}_{k}.\label{eqn:augmented ukf error dynamics}
\end{align}
\normalsize
The covariances $\bm{\Sigma}_{k+1|k}$, $\bm{\Sigma}^{y}_{k+1}$ and $\bm{\Sigma}^{xy}_{k+1}$ can be trivially expressed in the same forms as in forward UKF case, but with $\hat{\mathbf{Q}}_{k}=\mathbf{U}^{w}_{k}\mathbf{F}^{w}_{k}\mathbf{Q}_{k}(\mathbf{U}^{w}_{k}\mathbf{F}^{w}_{k})^{T}+\mathbf{U}^{x}_{k}\mathbf{F}_{k}\mathbf{K}_{k}\mathbf{R}_{k}\mathbf{K}_{k}^{T}\mathbf{F}_{k}^{T}\mathbf{U}^{x}_{k}+\delta\mathbf{P}_{k+1|k}+\Delta\mathbf{P}_{k+1|k}$. The following lemma extends the forward UKF's stability results to augmented state UKF.

\begin{lemma}[Stochastic stability of augmented state UKF]
\label{lemma:augmented state ukf stability}
Consider the non-linear stochastic system given by \eqref{eqn:augmented ukf state transition} and \eqref{eqn:observation y}. The augmented-state UKF's estimation error $\widetilde{\mathbf{x}}_{k}$ is exponentially bounded in mean-squared sense and bounded with probability one if all the assumptions of Theorem~\ref{theorem:forward ukf stability} hold true and additionally, there exists a constant $\overline{w}$ such that $\|\mathbf{U}^{w}_{k}\mathbf{F}^{w}_{k}\|\leq\overline{w}$ is fulfilled for all $k\geq 0$.
\end{lemma}
\begin{proof}
Similar to forward UKF, the different covariance matrices are expressed in terms of the unknown matrices as
\par\noindent\small
\begin{align*}
    \bm{\Sigma}_{k+1|k}&=\mathbf{U}^{x}_{k}\mathbf{F}_{k}(\mathbf{I}-\mathbf{K}_{k}\mathbf{U}^{y}_{k}\mathbf{H}_{k})\bm{\Sigma}_{k|k-1}(\mathbf{I}-\mathbf{K}_{k}\mathbf{U}^{y}_{k}\mathbf{H}_{k})^{T}\mathbf{F}_{k}^{T}\mathbf{U}^{x}_{k}\nonumber\\
    &\;\;\;+\hat{\mathbf{Q}}_{k},\\
    \bm{\Sigma}^{y}_{k+1}&=\mathbf{U}^{y}_{k+1}\mathbf{H}_{k+1}\bm{\Sigma}_{k+1|k}\mathbf{H}_{k+1}^{T}\mathbf{U}^{y}_{k+1}+\hat{\mathbf{R}}_{k+1},\nonumber\\
    \bm{\Sigma}^{xy}_{k+1}&=\begin{cases}\bm{\Sigma}_{k+1|k}\mathbf{U}^{xy}_{k+1}\mathbf{H}_{k+1}^{T}\mathbf{U}^{y}_{k+1}, & n_{x}\geq n_{y}\\
    \bm{\Sigma}_{k+1|k}\mathbf{H}_{k+1}^{T}\mathbf{U}^{y}_{k+1}\mathbf{U}^{xy}_{k+1}, & n_{x}<n_{y}\end{cases},\nonumber
\end{align*}
\normalsize
where $\hat{\mathbf{Q}}_{k}=\mathbf{U}^{w}_{k}\mathbf{F}^{w}_{k}\mathbf{Q}_{k}(\mathbf{U}^{w}_{k}\mathbf{F}^{w}_{k})^{T}+\mathbf{U}^{x}_{k}\mathbf{F}_{k}\mathbf{K}_{k}\mathbf{R}_{k}\mathbf{K}_{k}^{T}\mathbf{F}_{k}^{T}\mathbf{U}^{x}_{k}+\delta\mathbf{P}_{k+1|k}+\Delta\mathbf{P}_{k+1|k}$ and all other matrices including unknown matrices $\mathbf{U}^{y}_{k+1}$ and $\mathbf{U}^{xy}_{k+1}$ are same as defined for forward UKF. 

Define $V_{k}(\widetilde{\mathbf{x}}_{k|k-1})=\widetilde{\mathbf{x}}_{k|k-1}^{T}\bm{\Sigma}_{k|k-1}^{-1}\widetilde{\mathbf{x}}_{k|k-1}$ to apply Lemma~\ref{lemma:exponential boundedness}. Following similar steps as in proof of \cite[Theorem 2]{singh2022inverse_part1}, we have
\par\noindent\small
\begin{align}
&\mathbb{E}\left[V_{k+1}(\widetilde{\mathbf{x}}_{k+1|k})\vert\widetilde{\mathbf{x}}_{k|k-1}\right]=\widetilde{\mathbf{x}}_{k|k-1}^{T}(\mathbf{U}^{x}_{k}\mathbf{F}_{k}(\mathbf{I}-\mathbf{K}_{k}\mathbf{U}^{y}_{k}\mathbf{H}_{k}))^{T}\nonumber\\
&\;\;\;\times\bm{\Sigma}_{k+1|k}^{-1}(\mathbf{U}^{x}_{k}\mathbf{F}_{k}(\mathbf{I}-\mathbf{K}_{k}\mathbf{U}^{y}_{k}\mathbf{H}_{k}))\widetilde{\mathbf{x}}_{k|k-1}\nonumber\\
&+\mathbb{E}[\mathbf{w}_{k}^{T}(\mathbf{U}^{w}_{k}\mathbf{F}^{w}_{k})^{T}\bm{\Sigma}_{k+1|k}^{-1}(\mathbf{U}^{w}_{k}\mathbf{F}^{w}_{k})\mathbf{w}_{k}\vert\widetilde{\mathbf{x}}_{k|k-1}]\nonumber\\
&+\mathbb{E}[\mathbf{v}_{k}^{T}(\mathbf{U}^{x}_{k}\mathbf{F}_{k}\mathbf{K}_{k})^{T}\bm{\Sigma}_{k+1|k}^{-1}(\mathbf{U}^{x}_{k}\mathbf{F}_{k}\mathbf{K}_{k})\mathbf{v}_{k}\vert\widetilde{\mathbf{x}}_{k|k-1}].\label{eqn:augmented ukf stable Vk term}
\end{align}
\normalsize

Consider the last expectation term of \eqref{eqn:augmented ukf stable Vk term}. Since, $\bm{\Sigma}_{k+1|k}\succeq\underline{\sigma}\mathbf{I}$, we have $(\mathbf{U}^{w}_{k}\mathbf{F}^{w}_{k})^{T}\bm{\Sigma}_{k+1|k}^{-1}(\mathbf{U}^{w}_{k}\mathbf{F}^{w}_{k})\preceq\frac{1}{\underline{\sigma}}(\mathbf{U}^{w}_{k}\mathbf{F}^{w}_{k})^{T}(\mathbf{U}^{w}_{k}\mathbf{F}^{w}_{k})$. With the bound $\|\mathbf{U}^{w}_{k}\mathbf{F}^{w}_{k}\|\leq\overline{w}$, it can be upper bounded as $(\mathbf{U}^{w}_{k}\mathbf{F}^{w}_{k})^{T}\bm{\Sigma}_{k+1|k}^{-1}(\mathbf{U}^{w}_{k}\mathbf{F}^{w}_{k})\preceq\frac{\overline{w}^{2}}{\underline{\sigma}}\mathbf{I}$ such that similar to proof of \cite[Theorem 2]{singh2022inverse_part1}, we have $\mathbb{E}\left[\mathbf{w}_{k}^{T}(\mathbf{U}^{w}_{k}\mathbf{F}^{w}_{k})^{T}\bm{\Sigma}_{k+1|k}^{-1}(\mathbf{U}^{w}_{k}\mathbf{F}^{w}_{k})\mathbf{w}_{k}\vert\widetilde{\mathbf{x}}_{k|k-1}\right]\leq\frac{\overline{w}^{2}}{\underline{\sigma}}\mathbb{E}[\mathbf{w}_{k}^{T}\mathbf{w}_{k}]\leq\frac{\overline{w}^{2}}{\underline{\sigma}}\bar{q}n_{x}$. With this upper bound, Lemma~\ref{lemma:augmented state ukf stability} can be proved trivially following similar steps as in the proof of \cite[Theorem 2]{singh2022inverse_part1}.
\end{proof}
\vspace{-10pt}
\subsection{Preliminaries to the Proof}\label{subsec:preliminaries}
We state Lemma~\ref{lemma:vector matrix bounds} that we employ in the sequel.
\begin{lemma}
\label{lemma:vector matrix bounds}
Bounds on a vector $\mathbf{a}$ and a `$n\times m$' matrix $\mathbf{A}$ lead to the following bounds.\\
\textbf{(a)} If $\|\mathbf{a}\|_{2}\leq\delta$, then each component satisfies $|[\mathbf{a}]_{i}|\leq\delta.$\\
\textbf{(b)} If $\|\mathbf{A}\|\leq\delta$, then the $i$-th row sum $\sum_{j=1}^{m}|[\mathbf{A}]_{i,j}|\leq\sqrt{m}\delta$.\\
\textbf{(c)} If $\|\mathbf{A}\|\leq\delta$, then the $i$-th row satisfies $\|[\mathbf{A}]_{(i,:)}\|_{2}\leq\delta$.\\
\textbf{(d)} If each component satisfies $|[\mathbf{A}]_{i,j}|\leq\delta$, then $\|\mathbf{A}\|\leq\sqrt{nm}\delta$.
\end{lemma}
\begin{proof}
In the following, the notation $\|\cdot\|_{\infty}$ denotes the $l_{\infty}$ norm and the induced maximum row-sum norm for vectors and matrices, respectively.

For \textbf{(a)}, by the equivalence of vector norms, we have $\|\mathbf{a}\|_{\infty}\leq\|\mathbf{a}\|_{2}$. But by definition of $l_{\infty}$ norm, $\|\mathbf{a}\|_{\infty}= \textrm{max} |[\mathbf{a}]_{i}|$ and hence, $|[\mathbf{a}]_{i}|\leq\delta$.

For \textbf{(b)}, by the equivalence of matrix norms, $\|\mathbf{A}\|_{\infty}\leq\sqrt{m}\|\mathbf{A}\|\leq\sqrt{m}\delta$. But by the definition of $\|\mathbf{A}\|_{\infty}$, it is the maximum row sum such that for any i-th row of matrix $\mathbf{A}$, we have the bound $\sum_{j=1}^{m}|[\mathbf{A}]_{i,j}|\leq\sqrt{m}\delta$.

For \textbf{(c)}, by the definition of spectral norm, $\|\mathbf{A}^{T}\|=\textrm{max}_{\|\mathbf{x}\|_{2}=1}\|\mathbf{A}^{T}\mathbf{x}\|_{2}$. However, $\|\mathbf{A}^{T}\|=\|\mathbf{A}\|$ such that with $\mathbf{y}=\mathbf{x}^{T}$, we have $\|\mathbf{A}\|=\textrm{max}_{\|\mathbf{y}\|_{2}=1}\|\mathbf{yA}\|_{2}$, which implies $\|\mathbf{yA}\|_{2}\leq\|\mathbf{A}\|$ for any vector $\mathbf{y}$ with $\|\mathbf{y}\|_{2}=1$. Choosing $\mathbf{y}$ as a `$1\times n$' vector with only $i$-th component as $1$ and all other components as $0$, $\mathbf{yA}$ becomes the $i$-th row of $\mathbf{A}$ from which the result follows trivially.

For \textbf{(d)}, by the equivalence of matrix norms, we have $\|\mathbf{A}\|\leq\sqrt{nm}\|\mathbf{A}\|_{\textrm{max}}$ where $\|\mathbf{A}\|_{\textrm{max}}=\textrm{max}_{i,j} |[\mathbf{A}]_{i,j}|$. Hence, $|[\mathbf{A}]_{i,j}|\leq\delta$ leads to $\|\mathbf{A}\|\leq\sqrt{nm}\delta$.
\end{proof}

\begin{lemma}
\label{lemma: qi jacobian bound}
Under the assumptions of Theorem~\ref{theorem:IUKF stability}, the Jacobian $\frac{\partial\mathbf{q}_{i,k+1|k}}{\partial\hat{\mathbf{x}}_{k}}$ of the $i$-th sigma point generated for measurement update in the forward UKF with respect to the state estimate $\hat{\mathbf{x}}_{k}$ is bounded as $\left\|\frac{\partial\mathbf{q}_{i,k+1|k}}{\partial\hat{\mathbf{x}}_{k}}\right\|\leq c'$ for some positive real constant $c'$.
\end{lemma}

\begin{proof}
From \eqref{eqn:forward UKF update sigma points}, we observe that $\mathbf{q}_{i,k+1|k}$ is a linear function of the predicted state $\hat{\mathbf{x}}_{k+1|k}$ and the $i$-th column of $\sqrt{\bm{\Sigma}_{k+1|k}}$. In order to bound its Jacobian with respect to $\hat{\mathbf{x}}_{k}$, we need to bound the Jacobians of $\hat{\mathbf{x}}_{k+1|k}$ and $\sqrt{\bm{\Sigma}_{k+1|k}}$ with respect to $\hat{\mathbf{x}}_{k}$.

We start with $\hat{\mathbf{x}}_{k+1|k}$. Since $\mathbf{s}^{*}_{i,k+1|k}=f(\mathbf{s}_{i,k})$, the Jacobian $\frac{\partial\mathbf{s}^{*}_{i,k+1|k}}{\partial\hat{\mathbf{x}}_{k}}=\left.\frac{\partial f(\mathbf{x})}{\partial\mathbf{x}}\right\vert_{\mathbf{x}=\mathbf{s}_{i,k}}$ for all $0\leq i\leq 2n_{x}$, because $\frac{\partial\mathbf{s}_{i,k}}{\partial\hat{\mathbf{x}}_{k}}=\mathbf{I}$ from \eqref{eqn:forward ukf prediction sigma points}. Hence, differentiating \eqref{eqn:forward ukf x predict}, we have $\frac{\partial\hat{\mathbf{x}}_{k+1|k}}{\partial\hat{\mathbf{x}}_{k}}=\sum_{i=0}^{2n_{x}}\omega_{i}\left.\frac{\partial f(\mathbf{x})}{\partial\mathbf{x}}\right\vert_{\mathbf{x}=\mathbf{s}_{i,k}}$, which on using the upper-bound on Jacobian $\mathbf{F}_{k}$ from Theorem~\ref{theorem:forward ukf stability} yields
\par\noindent\small
\begin{align}
  \left\|\frac{\partial\hat{\mathbf{x}}_{k+1|k}}{\partial\hat{\mathbf{x}}_{k}}\right\|\leq\bar{f}.\label{eqn:IUKF stable bound on x predict derivative}  
\end{align}
\normalsize
The following Claim~\ref{claim:IUKF stable bound on cholesky derivative} bounds the derivative of $\sqrt{\bm{\Sigma}_{k+1|k}}$ with the detailed proof provided in Appendix~\ref{claim1}.
\begin{claim}
\label{claim:IUKF stable bound on cholesky derivative}
For any $i$-th column of $\sqrt{\bm{\Sigma}_{k+1|k}}$, we have the upper bound $\left\|\frac{\partial[\sqrt{\bm{\Sigma}_{k+1|k}}]_{(:,i)}}{\partial\hat{\mathbf{x}}_{k}}\right\|\leq n_{x}\delta_{\sigma}$ for some $\delta_{\sigma}>0$.
\end{claim}

Using Claim~\ref{claim:IUKF stable bound on cholesky derivative} and bound \eqref{eqn:IUKF stable bound on x predict derivative} in \eqref{eqn:forward UKF update sigma points}, we have $\left\|\frac{\partial\mathbf{q}_{i,k+1|k}}{\partial\hat{\mathbf{x}}_{k}}\right\|\leq c'$ where $c'=\bar{f}+n_{x}\delta_{\sigma}\sqrt{n_{x}+\kappa}$.
\end{proof}
\begin{lemma}
\label{lemma: jacobian bound}
Under the assumptions of Theorem~\ref{theorem:IUKF stability}, the Jacobian $\widetilde{\mathbf{F}}_{k}\doteq\left.\frac{\partial\widetilde{f}(\mathbf{x},\bm{\Sigma}_{k},\mathbf{x}_{k+1},\mathbf{0})}{\partial\mathbf{x}}\right\vert_{\mathbf{x}=\doublehat{\mathbf{x}}_{k}}$ of state transition equation \eqref{eqn: IUKF state transition detail} satisfies the bound $\|\widetilde{\mathbf{F}}_{k}\|\leq c_{f}$ for some positive real constant $c_{f}$.
\end{lemma}
\begin{proof}
Define $\mathbf{t}_{k}=h(\mathbf{x}_{k+1})+\mathbf{v}_{k+1}-\sum_{i=0}^{2n_{x}}\omega_{i}\mathbf{q}^{*}_{i,k+1|k}$ i.e. the difference in the actual observation $\mathbf{y}_{k+1}$ and its prediction. Rearranging \eqref{eqn: IUKF state transition detail} and using $\hat{\mathbf{x}}_{k+1|k}=\sum_{i=0}^{2n_{x}}\omega_{i}\mathbf{s}^{*}_{i,k+1|k}$, I-UKF's state transition becomes $\hat{\mathbf{x}}_{k+1}=\hat{\mathbf{x}}_{k+1|k}+\mathbf{K}_{k+1}\mathbf{t}_{k}$ such that its Jacobian with respect to $\hat{\mathbf{x}}_{k}$ (state to be estimated from I-UKF's state transition) is
\par\noindent\small
\begin{align}
    \widetilde{\mathbf{F}}_{k}=\left.\frac{\partial\hat{\mathbf{x}}_{k+1|k}}{\partial\hat{\mathbf{x}}_{k}}\right\vert_{\mathbf{v}_{k+1}=\mathbf{0}}+\left.\frac{\partial(\mathbf{K}_{k+1}\mathbf{t}_{k})}{\partial\hat{\mathbf{x}}_{k}}\right\vert_{\mathbf{v}_{k+1}=\mathbf{0}}.\label{eqn:IUKF stable jacobian exp}
\end{align}
\normalsize

First, consider the second derivative term. The $j$-th row of `$n_{x}\times n_{x}$' Jacobian $\frac{\partial(\mathbf{K}_{k+1}\mathbf{t}_{k})}{\partial\hat{\mathbf{x}}_{k}}$ consists of the first order partial derivatives of $j$-th element of `$n_{x}\times 1$' vector $\mathbf{K}_{k+1}\mathbf{t}_{k}$ with respect to the elements of $\hat{\mathbf{x}}_{k}$. This $j$-th element $[\mathbf{K}_{k+1}\mathbf{t}_{k}]_{j}=\sum_{m=1}^{n_{y}}[\mathbf{K}_{k+1}]_{j,m}[\mathbf{t}_{k}]_{m}$, such that the $j$-th row of the Jacobian is obtained as
\par\noindent\small
\begin{align}
    \hspace{-0.2cm}\left[\frac{\partial(\mathbf{K}_{k+1}\mathbf{t}_{k})}{\partial\hat{\mathbf{x}}_{k}}\right]_{(j,:)}=\sum_{m=1}^{n_{y}}[\mathbf{K}_{k+1}]_{j,m}\frac{\partial[\mathbf{t}_{k}]_{m}}{\partial\hat{\mathbf{x}}_{k}}+\sum_{m=1}^{n_{y}}[\mathbf{t}_{k}]_{m}\frac{\partial[\mathbf{K}_{k+1}]_{j,m}}{\partial\hat{\mathbf{x}}_{k}}.\label{eqn:IUKF stable jac K term j-th row}
\end{align}
\normalsize
The following Claim~\ref{claim:IUKF stable K j row bound} upper bounds this $j$-th row with the detailed proof provided in Appendix~\ref{claim2}.
\begin{claim}
    \label{claim:IUKF stable K j row bound}
    The $j$-th row of Jacobian of $\mathbf{K}_{k+1}\mathbf{t}_{k}$ satisfies $\left\|\left[\frac{\partial(\mathbf{K}_{k+1}\mathbf{t}_{k})}{\partial\hat{\mathbf{x}}_{k}}\right]_{(j,:)}\right\|_{2}\leq c_{t}$ for some $c_{t}>0$.
\end{claim}

With Claim~\ref{claim:IUKF stable K j row bound}, we can show that $\left\|\frac{\partial(\mathbf{K}_{k+1}\mathbf{t}_{k})}{\partial\hat{\mathbf{x}}_{k}}\right\|\leq n_{x}c_{t}$ using Lemma~\ref{lemma:vector matrix bounds}\textbf{(a)} followed by Lemma~\ref{lemma:vector matrix bounds}\textbf{(d)}. Hence, from \eqref{eqn:IUKF stable jacobian exp} along with the bound \eqref{eqn:IUKF stable bound on x predict derivative} since $\hat{\mathbf{x}}_{k+1|k}$ is independent of the noise term $\mathbf{v}_{k+1}$, we have $\|\widetilde{\mathbf{F}}_{k}\|\leq c_{f}=\bar{f}+n_{x}c_{t}$.
\end{proof}
\vspace{-5pt}
\subsection{Proof of the Theorem}\label{subsec:proof of IUKF theorem}
Define $\widetilde{\mathbf{F}}^{v}_{k}\doteq\left.\frac{\partial\widetilde{f}(\doublehat{\mathbf{x}}_{k},\bm{\Sigma}_{k},\mathbf{x}_{k+1},\mathbf{v})}{\partial\mathbf{v}}\right\vert_{\mathbf{v}=\mathbf{0}}$ and $\overline{\mathbf{U}}^{v}_{k}$ as the counterpart of $\mathbf{U}^{w}_{k}$ for I-UKF dynamics. We will show that under the assumptions of Theorem~\ref{theorem:IUKF stability}, the I-UKF's dynamics satisfies the required conditions of Lemma~\ref{lemma:augmented state ukf stability}. In this regard, we show that I-UKF's dynamics satisfies the following conditions for all $k\geq 0$ for some constants $c_{f},c_{\alpha},c_{q},c_{v}$.\\
\textbf{A.C1.} $\|\widetilde{\mathbf{F}}_{k}\|\leq c_{f}$.\\
\textbf{A.C2.} $\overline{\mathbf{Q}}_{k}\preceq c_{q}\mathbf{I}$.\\
\textbf{A.C3.} $\|\overline{\mathbf{U}}^{x}_{k}\|\leq c_{\alpha}$.\\
\textbf{A.C4.} $\overline{\mathbf{U}}^{x}_{k}$ is non-singular for all $k\geq 0$.\\
\textbf{A.C5.} $\|\overline{\mathbf{U}}^{v}_{k}\widetilde{\mathbf{F}}^{v}_{k}\|\leq c_{v}$.\\
All other conditions of Lemma~\ref{lemma:augmented state ukf stability} are assumed to hold true in Theorem~\ref{theorem:IUKF stability}, such that the I-UKF's estimation error is exponentially bounded in mean-squared sense and bounded with probability one. Also, \textbf{A.C1} is satisfied by Lemma~\ref{lemma: jacobian bound}.

For \textbf{A.C2}, the noise $\mathbf{v}_{k+1}$ of state transition \eqref{eqn: IUKF state transition detail} has covariance $\overline{\mathbf{Q}}_{k}=\mathbf{R}_{k+1}$ which is upper-bounded as $\mathbf{R}_{k+1}\preceq\bar{r}\mathbf{I}$ from one of the assumptions of Theorem~\ref{theorem:forward ukf stability}. Hence, \textbf{A.C2} holds true with $c_{q}=\bar{r}$.

For \textbf{A.C3} and \textbf{A.C4}, introducing the unknown matrix $\overline{\mathbf{U}}^{x}_{k}$ and $\overline{\mathbf{U}}^{v}_{k}\widetilde{\mathbf{F}}^{v}_{k}=\mathbf{K}_{k+1}$ since the second and higher order derivatives of \eqref{eqn: IUKF state transition detail} with respect to the noise term $\mathbf{v}_{k+1}$ are zero, we have I-UKF's state prediction error $\hat{\widetilde{\mathbf{x}}}_{k+1|k}\doteq\hat{\mathbf{x}}_{k+1}-\doublehat{\mathbf{x}}_{k+1|k}$ as
\par\noindent\small
\begin{align}    \hat{\widetilde{\mathbf{x}}}_{k+1|k}=\overline{\mathbf{U}}^{x}_{k}\widetilde{\mathbf{F}}_{k}\hat{\widetilde{\mathbf{x}}}_{k}+\mathbf{K}_{k+1}\mathbf{v}_{k+1},\label{eqn:IUKF stable unknown matrix}
\end{align}
\normalsize
where $\hat{\widetilde{\mathbf{x}}}_{k}\doteq\hat{\mathbf{x}}_{k}-\doublehat{\mathbf{x}}_{k}$ is the I-UKF's state estimation error. However, substituting for $\hat{\mathbf{x}}_{k+1}$ using \eqref{eqn: IUKF state transition detail} and $\doublehat{\mathbf{x}}_{k+1|k}$ using \eqref{eqn:IUKF x predict}, we have
\par\noindent\small
\begin{align}
    \hat{\widetilde{\mathbf{x}}}_{k+1|k}&=\sum_{i=0}^{2n_{x}}\omega_{i}\mathbf{s}^{*}_{i,k+1|k}-\mathbf{K}_{k+1}\sum_{i=0}^{2n_{x}}\omega_{i}\mathbf{q}^{*}_{i,k+1|k}+\mathbf{K}_{k+1}h(\mathbf{x}_{k+1})\nonumber\\
    &\;\;\;+\mathbf{K}_{k+1}\mathbf{v}_{k+1}-\sum_{j=0}^{2n_{z}}\overline{\omega}_{j}\overline{\mathbf{s}}^{*}_{j,k+1|k}.\label{eqn:IUKF stable state prediction error}
\end{align}
\normalsize
From $\sum_{i=0}^{2n_{x}}\omega_{i}\mathbf{s}^{*}_{i,k+1|k}=\sum_{i=0}^{2n_{x}}\omega_{i}f(\mathbf{s}_{i,k})$ using the first-order Taylor series expansion, we have $\sum_{i=0}^{2n_{x}}\omega_{i}\mathbf{s}^{*}_{i,k+1|k}=f(\hat{\mathbf{x}}_{k})$. Similarly, $\sum_{i=0}^{2n_{x}}\omega_{i}\mathbf{q}^{*}_{i,k+1|k}=h(\hat{\mathbf{x}}_{k+1|k})$ and $\sum_{j=0}^{2n_{z}}\overline{\omega}_{j}\overline{\mathbf{s}}^{*}_{j,k+1|k}=\widetilde{f}(\doublehat{\mathbf{x}}_{k},\bm{\Sigma}_{k},\mathbf{x}_{k+1},\mathbf{0})$. Substituting \eqref{eqn: IUKF state transition detail} for $\widetilde{f}(\cdot)$ and again using the Taylor series expansion for the summation terms involving sigma points, we have $\sum_{j=0}^{2n_{z}}\overline{\omega}_{j}\overline{\mathbf{s}}^{*}_{j,k+1|k}=f(\doublehat{\mathbf{x}}_{k})-\mathbf{K}_{k+1}h(\overline{\mathbf{x}}_{k+1|k})+\mathbf{K}_{k+1}h(\mathbf{x}_{k+1})$, where $\overline{\mathbf{x}}_{k+1|k}=f(\doublehat{\mathbf{x}}_{k})$ under the assumption that the gain $\mathbf{K}_{k+1}$ computed from $\doublehat{\mathbf{x}}_{k}$ is approximately same as that computed from forward UKF's $\hat{\mathbf{x}}_{k}$. Hence, \eqref{eqn:IUKF stable state prediction error} simplifies to $\hat{\widetilde{\mathbf{x}}}_{k+1|k}=f(\hat{\mathbf{x}}_{k})-f(\doublehat{\mathbf{x}}_{k})-\mathbf{K}_{k+1}(h(\hat{\mathbf{x}}_{k+1|k})-h(\overline{\mathbf{x}}_{k+1|k}))+\mathbf{K}_{k+1}\mathbf{v}_{k+1}$. Using the unknown matrices $\mathbf{U}^{x}_{k}$ and $\mathbf{U}^{y}_{k}$ introduced in forward UKF for linearizing $f(\cdot)$ and $h(\cdot)$, respectively, at $\hat{\mathbf{x}}_{k}$ and $\hat{\mathbf{x}}_{k+1|k}$, we obtain $\hat{\widetilde{\mathbf{x}}}_{k+1|k}=\mathbf{U}^{x}_{k}\mathbf{F}_{k}\hat{\widetilde{\mathbf{x}}}_{k}-\mathbf{K}_{k+1}\mathbf{U}^{y}_{k+1}\mathbf{H}_{k+1}\mathbf{U}^{x}_{k}\mathbf{F}_{k}\hat{\widetilde{\mathbf{x}}}_{k}+\mathbf{K}_{k+1}\mathbf{v}_{k+1}$. Comparing with \eqref{eqn:IUKF stable unknown matrix}, we have $\overline{\mathbf{U}}^{x}_{k}=(\mathbf{I}-\mathbf{K}_{k+1}\mathbf{U}^{y}_{k+1}\mathbf{H}_{k+1})\mathbf{U}^{x}_{k}\mathbf{F}_{k}\widetilde{\mathbf{F}}_{k}^{-1}$ since $\widetilde{\mathbf{F}}_{k}^{-1}$ is assumed to be non-singular in Theorem~\ref{theorem:IUKF stability}. With the bounds assumed on various matrices in Theorem~\ref{theorem:forward ukf stability} and the assumption $\|\widetilde{\mathbf{F}}_{k}^{-1}\|\leq\bar{a}$ from Theorem~\ref{theorem:IUKF stability}, it is straightforward to obtain $ \|\overline{\mathbf{U}}^{x}_{k}\|\leq\bar{\alpha}\bar{f}\bar{a}(1+\bar{k}\bar{\beta}\bar{h})$ such that \textbf{A.C3} holds true with $c_{\alpha}=\bar{\alpha}\bar{f}\bar{a}(1+\bar{k}\bar{\beta}\bar{h})$. Also, \textbf{A.C4} holds true, i.e, $\overline{\mathbf{U}}^{x}_{k}$ is non-singular because $\mathbf{U}^{x}_{k}$ and $\mathbf{F}_{k}$ are non-singular from the assumptions of Theorem~\ref{theorem:forward ukf stability}, and $(\mathbf{I}-\mathbf{K}_{k+1}\mathbf{U}^{y}_{k+1}\mathbf{H}_{k+1})$ can be proved to be invertible under the forward UKF's stability conditions as proved intermediately, in the proof of \cite[Theorem~2]{singh2022inverse_part1}.

For \textbf{A.C5}, we have $\overline{\mathbf{U}}^{v}_{k}\widetilde{\mathbf{F}}^{v}_{k}=\mathbf{K}_{k+1}$ since the second and higher order derivatives of \eqref{eqn: IUKF state transition detail} with respect to $\mathbf{v}_{k+1}$ are zero. Under Theorem~\ref{theorem:forward ukf stability}'s assumptions, $\|\mathbf{K}_{k+1}\|\leq\bar{k}$ and \textbf{A.C5} holds with $c_{v}=\bar{k}$.

\subsection{Proofs of Claims~\ref{claim:IUKF stable bound on cholesky derivative} and \ref{claim:IUKF stable K j row bound}}\label{subsec:app-thm-intermediate}
\subsubsection{Proof of Claim~\ref{claim:IUKF stable bound on cholesky derivative}}\label{claim1}
In order to bound the derivative of $\sqrt{\bm{\Sigma}_{k+1|k}}$, we first upper-bound the derivative of $\bm{\Sigma}_{k+1|k}$ in Claim~\ref{claim:IUKF stable predict sig derivative bound}.
\begin{claim}
\label{claim:IUKF stable predict sig derivative bound}
For any $(l,m)$-th element of $\bm{\Sigma}_{k+1|k}$, we have the upper bound $\left\|\frac{\partial[\bm{\Sigma}_{k+1|k}]_{l,m}}{\partial\hat{\mathbf{x}}_{k}}\right\|_{2}\leq 4\delta_{f}\bar{f}$.
\end{claim}
\begin{claimproof}
We have $\bm{\Sigma}_{k+1|k}=\sum_{i=0}^{2n_{x}}\omega_{i}\mathbf{s}^{*}_{i,k+1|k}(\mathbf{s}^{*}_{i,k+1|k})^{T}-\hat{\mathbf{x}}_{k+1|k}(\hat{\mathbf{x}}_{k+1|k})^{T}+\mathbf{Q}_{k}$. Its $(l,m)$-th element is $[\bm{\Sigma}_{k+1|k}]_{l,m}\\=\sum_{i=0}^{2n_{x}}\omega_{i}[\mathbf{s}^{*}_{i,k+1|k}]_{l}[\mathbf{s}^{*}_{i,k+1|k}]_{m}-[\hat{\mathbf{x}}_{k+1|k}]_{l}[\hat{\mathbf{x}}_{k+1|k}]_{m}+[\mathbf{Q}_{k}]_{l,m}$, which implies
\par\noindent\small
\begin{align*}
&\frac{\partial[\bm{\Sigma}_{k+1|k}]_{l,m}}{\partial\hat{\mathbf{x}}_{k}}=-[\hat{\mathbf{x}}_{k+1|k}]_{l}\frac{\partial[\hat{\mathbf{x}}_{k+1|k}]_{m}}{\partial\hat{\mathbf{x}}_{k}}-[\hat{\mathbf{x}}_{k+1|k}]_{m}\frac{\partial[\hat{\mathbf{x}}_{k+1|k}]_{l}}{\partial\hat{\mathbf{x}}_{k}}\\
&+\sum_{i=0}^{2n_{x}}\omega_{i}[\mathbf{s}^{*}_{i,k+1|k}]_{m}\frac{\partial[\mathbf{s}^{*}_{i,k+1|k}]_{l}}{\partial\hat{\mathbf{x}}_{k}}+\sum_{i=0}^{2n_{x}}\omega_{i}[\mathbf{s}^{*}_{i,k+1|k}]_{l}\frac{\partial[\mathbf{s}^{*}_{i,k+1|k}]_{m}}{\partial\hat{\mathbf{x}}_{k}}.
\end{align*}
\normalsize
Note that all derivatives here are gradients since $[\bm{\Sigma}_{k+1|k}]_{l,m}$ is a scalar. Theorem~\ref{theorem:IUKF stability} assumes $f(\cdot)$ has bounded outputs and hence, the magnitude of each element of $\mathbf{s}^{*}_{i,k+1|k}$ and $\hat{\mathbf{x}}_{k+1|k}$ is also bounded by $\delta_{f}$ according to Lemma~\ref{lemma:vector matrix bounds}\textbf{(a)}, which leads to
\par\noindent\small
\begin{align*}
&\left\|\frac{\partial[\bm{\Sigma}_{k+1|k}]_{l,m}}{\partial\hat{\mathbf{x}}_{k}}\right\|_{2}=\delta_{f}\left\|\frac{\partial[\hat{\mathbf{x}}_{k+1|k}]_{m}}{\partial\hat{\mathbf{x}}_{k}}\right\|_{2}+\delta_{f}\left\|\frac{\partial[\hat{\mathbf{x}}_{k+1|k}]_{l}}{\partial\hat{\mathbf{x}}_{k}}\right\|_{2}\\
&+\sum_{i=0}^{2n_{x}}\omega_{i}\delta_{f}\left\|\frac{\partial[\mathbf{s}^{*}_{i,k+1|k}]_{l}}{\partial\hat{\mathbf{x}}_{k}}\right\|_{2}+\sum_{i=0}^{2n_{x}}\omega_{i}\delta_{f}\left\|\frac{\partial[\mathbf{s}^{*}_{i,k+1|k}]_{m}}{\partial\hat{\mathbf{x}}_{k}}\right\|_{2}.
\end{align*}
\normalsize
However, $\frac{\partial[\mathbf{s}^{*}_{i,k+1|k}]_{l}}{\partial\hat{\mathbf{x}}_{k}}$ is the $l$-th row of Jacobian $\frac{\partial\mathbf{s}^{*}_{i,k+1|k}}{\partial\hat{\mathbf{x}}_{k}}$ whose spectral norm is bounded by $\bar{f}$ since $\|\mathbf{F}_{k}\|\leq\bar{f}$ (as used to obtain \eqref{eqn:IUKF stable bound on x predict derivative}). Hence, using Lemma~\ref{lemma:vector matrix bounds}\textbf{(c)} with the bounds on the spectral norms $\left\|\frac{\partial\mathbf{s}^{*}_{i,k+1|k}}{\partial\hat{\mathbf{x}}_{k}}\right\|$ and $\left\|\frac{\partial\hat{\mathbf{x}}_{k+1|k}}{\partial\hat{\mathbf{x}}_{k}}\right\|$, we have $\left\|\frac{\partial[\bm{\Sigma}_{k+1|k}]_{l,m}}{\partial\hat{\mathbf{x}}_{k}}\right\|_{2}\leq 4\delta_{f}\bar{f}$.
\end{claimproof}

From the Cholesky decomposition of $\bm{\Sigma}_{k+1|k}$, we have $[\bm{\Sigma}_{k+1|k}]_{l,m}=\sum_{j=1}^{n_{x}}[\sqrt{\bm{\Sigma}_{k+1|k}}]_{l,j}[\sqrt{\bm{\Sigma}_{k+1|k}}]_{m,j}$. However, $\bm{\Sigma}_{k+1|k}$ is a symmetric, positive definite matrix (one of the assumptions of Theorem~\ref{theorem:forward ukf stability}) and hence, $\sqrt{\bm{\Sigma}_{k+1|k}}$ is a lower triangular matrix with $[\sqrt{\bm{\Sigma}_{k+1|k}}]_{i,j}=0$ for $j\geq i$ such that $[\bm{\Sigma}_{k+1|k}]_{l,m}=\sum_{j=1}^{\textrm{min}(l,m)}[\sqrt{\bm{\Sigma}_{k+1|k}}]_{l,j}[\sqrt{\bm{\Sigma}_{k+1|k}}]_{m,j}$ where indices $l$ and $m$ ranges from $1$ to $n_{x}$. Differentiating with respect to $i$-th element of $\hat{\mathbf{x}}_{k}$, we have
\par\noindent\small
\begin{align}
&\frac{\partial[\bm{\Sigma}_{k+1|k}]_{l,m}}{\partial[\hat{\mathbf{x}}_{k}]_{i}}=\sum_{j=1}^{\textrm{min}(l,m)}[\sqrt{\bm{\Sigma}_{k+1|k}}]_{l,j}\frac{\partial[\sqrt{\bm{\Sigma}_{k+1|k}}]_{m,j}}{\partial[\hat{\mathbf{x}}_{k}]_{i}}\nonumber\\
&\;\;\;+\sum_{j=1}^{\textrm{min}(l,m)}[\sqrt{\bm{\Sigma}_{k+1|k}}]_{m,j}\frac{\partial[\sqrt{\bm{\Sigma}_{k+1|k}}]_{l,j}}{\partial[\hat{\mathbf{x}}_{k}]_{i}}.\label{eqn:IUKF stable cholesky derivative}
\end{align}
\normalsize
Note that here, all the derivatives are scalar. We will bound each individual term of this equation. We denote $\frac{\partial[\bm{\Sigma}_{k+1|k}]_{l,m}}{\partial[\hat{\mathbf{x}}_{k}]_{i}}$ by $a_{l,m}$ which is the $i$-th element of $\frac{\partial[\bm{\Sigma}_{k+1|k}]_{l,m}}{\partial\hat{\mathbf{x}}_{k}}$. Also, denote the $(l,j)$-th element $[\sqrt{\bm{\Sigma}_{k+1|k}}]_{l,j}$ by $b_{l,j}$ and its derivative $\frac{\partial[\sqrt{\bm{\Sigma}_{k+1|k}}]_{l,j}}{\partial[\hat{\mathbf{x}}_{k}]_{i}}$ by $c_{l,j}$. We require an upper-bound on $c_{l,j}$.

\textbf{Upper-bounds on $|a_{l,m}|$ and $|b_{l,j}|$, and lower bound on $|b_{i,i}|$:} Using Claim~\ref{claim:IUKF stable predict sig derivative bound} and Lemma~\ref{lemma:vector matrix bounds}\textbf{(a)}, we have $|a_{l,m}|\leq 4\delta_{f}\bar{f}$. By the definition of spectral norm and bound $\bm{\Sigma}_{k+1|k}\preceq\bar{\sigma}\mathbf{I}$ from one of the assumptions of Theorem~\ref{theorem:forward ukf stability}, we have $\|\sqrt{\bm{\Sigma}_{k+1|k}}\|\leq\sqrt{\bar{\sigma}}$, which again using Lemma~\ref{lemma:vector matrix bounds}\textbf{(c)} followed by Lemma~\ref{lemma:vector matrix bounds}\textbf{(a)} gives $|b_{l,j}|\leq\sqrt{\bar{\sigma}}$.

Also, $\textrm{det}(\bm{\Sigma}_{k+1|k})\neq 0$ (positive definite matrix) and hence, $\textrm{det}(\sqrt{\bm{\Sigma}_{k+1|k}})\neq 0$. However, being a lower triangular matrix, $\textrm{det}(\sqrt{\bm{\Sigma}_{k+1|k}})$ is the product of its diagonal entries. Hence, no diagonal entry of $\sqrt{\bm{\Sigma}_{k+1|k}}$ is $0$ i.e. $b_{i,i}\neq 0$. Next, consider the bound $\underline{\sigma}\mathbf{I}\preceq\bm{\Sigma}_{k+1|k}$ from one of the assumptions of Theorem~\ref{theorem:IUKF stability}. Since, $\underline{\sigma}$ is a lower bound on eigenvalues of $\bm{\Sigma}_{k+1|k}$, $\textrm{det}(\bm{\Sigma}_{k+1|k})\geq\underline{\sigma}^{n_{x}}$ such that $|\textrm{det}(\sqrt{\bm{\Sigma}_{k+1|k}})|\geq\underline{\sigma}^{n_{x}/2}$. Expressing $\textrm{det}(\sqrt{\bm{\Sigma}_{k+1|k}})$ as the product of its diagonal entries $b_{i,i}$ and using the upper-bound on $|b_{i,i}|$ for all but one diagonal entry, we obtain the lower bound $|b_{i,i}|\geq c_{b}$ where $c_{b}=\underline{\sigma}^{n_{x}/2}\bar{\sigma}^{(1-n_{x})/2}$.

\textbf{Upper-bounds on $|c_{l,j}|$:} Putting different values of $l$ and $m$ in \eqref{eqn:IUKF stable cholesky derivative}, we obtain a system of $n_{x}(n_{x}+1)/2$ linear equations with same number of unknowns $c_{l,j}$ since $\bm{\Sigma}_{k+1|k}$ is symmetric. Consider $l=1$ and $m=1$ which gives $2b_{1,1}c_{1,1}=a_{1,1}$. Since, $b_{1,1}\neq 0$, $c_{1,1}=a_{1,1}/2b_{1,1}$. With the upper-bound on $|a_{1,1}|$ and lower-bound on $|b_{1,1}|$, we have $|c_{1,1}|\leq\frac{4\delta_{f}\bar{f}}{2c_{b}}$. Again, putting $l=2$ and $m=1$, we have $c_{2,1}=\frac{a_{1,2}-b_{2,1}c_{1,1}}{b_{1,1}}$ which implies $|c_{2,1}|\leq\frac{8c_{b}\delta_{f}\bar{f}+4\delta_{f}\bar{f}\sqrt{\bar{\sigma}}}{2c_{b}^{2}}$. Continuing further in the same manner for all the equations of the linear system, we can show that $|c_{l,j}|\leq\delta_{\sigma}$ for all $(l,j)$-th element of $\sqrt{\bm{\Sigma}_{k+1|k}}$ where $\delta_{\sigma}$ is the maximum of all these upper-bounds i.e. the magnitude of the partial derivative of any element of $\sqrt{\bm{\Sigma}_{k+1|k}}$ with respect to any element of $\hat{\mathbf{x}}_{k}$ is bounded. Hence, the magnitude of each element of the Jacobian $\frac{\partial[\sqrt{\bm{\Sigma}_{k+1|k}}]_{(:,i)}}{\partial\hat{\mathbf{x}}_{k}}$ of $i$-th column of $\sqrt{\bm{\Sigma}_{k+1|k}}$ is bounded such that Lemma~\ref{lemma:vector matrix bounds} \textbf{(d)} yields $\left\|\frac{\partial[\sqrt{\bm{\Sigma}_{k+1|k}}]_{(:,i)}}{\partial\hat{\mathbf{x}}_{k}}\right\|\leq n_{x}\delta_{\sigma}$.

\subsubsection{Proof of Claim~\ref{claim:IUKF stable K j row bound}}\label{claim2}
Here, we will upper bound the $j$-th row in \eqref{eqn:IUKF stable jac K term j-th row} by bounding the magnitude of the terms in the R.H.S. in the following Claims~\ref{claim:IUKF stable tk term bound}-\ref{claim:IUKF stable derivative of K bound}.
\begin{claim}
\label{claim:IUKF stable tk term bound}
The $m$-the element of vector $\mathbf{t}_{k}$ satisfies $|[\mathbf{t}_{k}]_{m}|\leq 2\delta_{h}$.
\end{claim}
\begin{claimproof}
    Using $\mathbf{v}_{k+1}=\mathbf{0}$ ($\widetilde{\mathbf{F}}_{k}$ is evaluated at $\mathbf{v}_{k+1}=0$), we have $[\mathbf{t}_{k}]_{m}=[h(\mathbf{x}_{k+1})]_{m}-\sum_{i=0}^{2n_{x}}\omega_{i}[\mathbf{q}^{*}_{i,k+1|k}]_{m}$ because $h(\mathbf{x}_{k+1})$ and $\lbrace\mathbf{q}^{*}_{i,k+1|k}\rbrace_{0\leq i\leq 2n_{x}}$ do not depend on the noise term. Also, $[\mathbf{q}^{*}_{i,k+1|k}]_{m}=[h(\mathbf{q}_{i,k+1|k})]_{m}$. Since Theorem~\ref{theorem:IUKF stability} assumes $h(\cdot)$ has bounded outputs and $\sum_{i=0}^{2n_{x}}\omega_{i}=1$, Lemma~\ref{lemma:vector matrix bounds}\textbf{(a)} leads to $|[\mathbf{t}_{k}]_{m}|\leq 2\delta_{h}$.
\end{claimproof}
\begin{claim}
\label{claim:IUKF stable tk derivative bound}
The derivative of $m$-th element of $\mathbf{t}_{k}$ is bounded as $\left\|\frac{\partial[\mathbf{t}_{k}]_{m}}{\partial\hat{\mathbf{x}}_{k}}\right\|_{2}\leq\bar{h}c'$ where the constant $c'$ is same as defined in Lemma~\ref{lemma: qi jacobian bound}.
\end{claim}
\begin{claimproof}
From $\mathbf{t}_{k}=h(\mathbf{x}_{k+1})+\mathbf{v}_{k+1}-\sum_{i=0}^{2n_{x}}\omega_{i}\mathbf{q}^{*}_{i,k+1|k}$, we obtain the derivative of $m$-th element as
\par\noindent\small
\begin{align}
    \frac{\partial[\mathbf{t}_{k}]_{m}}{\partial\hat{\mathbf{x}}_{k}}=-\sum_{i=0}^{2n_{x}}\omega_{i}\frac{\partial[\mathbf{q}^{*}_{i,k+1|k}]_{m}}{\partial\hat{\mathbf{x}}_{k}}.\label{eqn:IUKF stable tk term derivative}
\end{align}
\normalsize
But $\mathbf{q}^{*}_{i,k+1|k}=h(\mathbf{q}_{i,k+1|k})$ such that $\left\|\frac{\partial\mathbf{q}^{*}_{i,k+1|k}}{\partial\hat{\mathbf{x}}_{k}}\right\|=\left\|\left.\frac{\partial h(\mathbf{x})}{\partial\mathbf{x}}\right\vert_{\mathbf{x}=\mathbf{q}_{i,k+1|k}}\frac{\partial\mathbf{q}_{i,k+1|k}}{\partial\hat{\mathbf{x}}_{k}}\right\|$. Using the bound on Jacobian $\mathbf{H}_{k}$ from one of the assumptions of Theorem~\ref{theorem:forward ukf stability}, we have $\left\|\frac{\partial\mathbf{q}^{*}_{i,k+1|k}}{\partial\hat{\mathbf{x}}_{k}}\right\|\leq\bar{h}\left\|\frac{\partial\mathbf{q}_{i,k+1|k}}{\partial\hat{\mathbf{x}}_{k}}\right\|$. Further, with Lemma~\ref{lemma: qi jacobian bound}, we have $\left\|\frac{\partial\mathbf{q}^{*}_{i,k+1|k}}{\partial\hat{\mathbf{x}}_{k}}\right\|\leq\bar{h}c'$. Finally, using Lemma~\ref{lemma:vector matrix bounds}\textbf{(c)} ($\frac{\partial[\mathbf{q}^{*}_{i,k+1|k}]_{m}}{\partial\hat{\mathbf{x}}_{k}}$ is the $m$-th row of $\frac{\partial\mathbf{q}^{*}_{i,k+1|k}}{\partial\hat{\mathbf{x}}_{k}}$) and \eqref{eqn:IUKF stable tk term derivative}, we have $\left\|\frac{\partial[\mathbf{t}_{k}]_{m}}{\partial\hat{\mathbf{x}}_{k}}\right\|_{2}\leq\bar{h}c'$.
\end{claimproof}
\begin{claim}
\label{claim:IUKF stable derivative of K bound}
The derivative of $(j,m)$-th element of $\mathbf{K}_{k+1}$ satisfies $\left\|\frac{\partial[\mathbf{K}_{k+1}]_{j,m}}{\partial\hat{\mathbf{x}}_{k}}\right\|_{2}\leq c_{k}$ for some $c_{k}>0$.
\end{claim}
\begin{claimproof}
The $(j,m)$-th element of $\mathbf{K}_{k+1}$ is $[\mathbf{K}_{k+1}]_{j,m}=\sum_{a=1}^{n_{y}}\left[\bm{\Sigma}^{xy}_{k+1}\right]_{j,a}\left[(\bm{\Sigma}^{y}_{k+1})^{-1}\right]_{a,m}$ which implies
\par\noindent\small
\begin{align}
    &\frac{\partial[\mathbf{K}_{k+1}]_{j,m}}{\partial\hat{\mathbf{x}}_{k}}=\sum_{a=1}^{n_{y}}\left[\bm{\Sigma}^{xy}_{k+1}\right]_{j,a}\frac{\partial\left[(\bm{\Sigma}^{y}_{k+1})^{-1}\right]_{a,m}}{\partial\hat{\mathbf{x}}_{k}}\nonumber\\
    &\;\;\;+\sum_{a=1}^{n_{y}}\left[(\bm{\Sigma}^{y}_{k+1})^{-1}\right]_{a,m}\frac{\partial\left[\bm{\Sigma}^{xy}_{k+1}\right]_{j,a}}{\partial\hat{\mathbf{x}}_{k}}.\label{eqn:IUKF stable K derivative terms}
\end{align}
\normalsize
Again, we will bound the magnitude of each individual term on R.H.S.

\textbf{Bound on $\left[\bm{\Sigma}^{xy}_{k+1}\right]_{j,a}$ and its derivative:} Under the bounds on functions $f(\cdot)$ and $h(\cdot)$, we have $|[\hat{\mathbf{x}}_{k+1|k}]_{j}|\leq\delta_{f}$ and $|[\sqrt{\bm{\Sigma}_{k+1|k}}]_{j,i}|\leq\sqrt{\bar{\sigma}}$ (the bound on $|b_{j,i}|$ used to prove Claim~\ref{claim:IUKF stable bound on cholesky derivative} of Lemma~\ref{lemma: qi jacobian bound}). It is then straightforward to obtain the bound on the $a$-th element as $|[\hat{\mathbf{y}}_{k+1|k}]_{a}|\leq\delta_{h}$ and the $j$-th element as $|[\mathbf{q}_{i,k+1|k}]_{j}|\leq\delta_{f}+\sqrt{\bar{\sigma}(n_{x}+\kappa)}$ from \eqref{eqn:forward UKF update sigma points}. Since $[\bm{\Sigma}^{xy}_{k+1}]_{j,a}=\sum_{i=0}^{2n_{x}}\omega_{i}[\mathbf{q}_{i,k+1|k}]_{j}[\mathbf{q}^{*}_{i,k+1|k}]_{a}-[\hat{\mathbf{x}}_{k+1|k}]_{j}[\hat{\mathbf{y}}_{k+1|k}]_{a}$, we have $|[\bm{\Sigma}^{xy}_{k+1}]_{j,a}|\leq\delta_{f}\delta_{h}+\delta_{h}(\delta_{f}+\sqrt{\bar{\sigma}(n_{x}+\kappa)})$. The upper-bound $|[\mathbf{q}^{*}_{i,k+1|k}]_{a}|\leq\delta_{h}$ was used in proof of Claim~\ref{claim:IUKF stable tk term bound} as well.

Furthermore, following similar steps as used to bound $\left\|\frac{\partial[\bm{\Sigma}_{k+1|k}]_{l,m}}{\partial\hat{\mathbf{x}}_{k}}\right\|_{2}$ in Claim~\ref{claim:IUKF stable predict sig derivative bound} of Lemma~\ref{lemma: qi jacobian bound}, we have $\left\|\frac{\partial[\bm{\Sigma}^{xy}_{k+1}]_{l,m}}{\partial\hat{\mathbf{x}}_{k}}\right\|_{2}\leq c_{xy}$ with $c_{xy}=\bar{h}c'(\delta_{f}+\sqrt{\bar{\sigma}(n_{x}+\kappa)})+c'(\delta_{f}+\sqrt{\bar{\sigma}(n_{x}+\kappa)})+\delta_{f}\bar{h}c'+\delta_{h}\bar{f}$.

\textbf{Bound on $\left[(\bm{\Sigma}^{y}_{k+1})^{-1}\right]_{a,m}$ and its derivative:} With $\underline{y}\mathbf{I}\preceq\bm{\Sigma}^{y}_{k+1}$ as one of the assumptions of Theorem~\ref{theorem:IUKF stability}, we have $\|(\bm{\Sigma}^{y}_{k+1})^{-1}\|\leq\frac{1}{|\underline{y}|}$. Using Lemma~\ref{lemma:vector matrix bounds}\textbf{(c)} followed by Lemma~\ref{lemma:vector matrix bounds}\textbf{(a)} yields $|[(\bm{\Sigma}^{y}_{k+1})^{-1}]_{a,m}|\leq 1/|\underline{y}|$.

The derivative of $(a,m)$-th element of $(\bm{\Sigma}^{y}_{k+1})^{-1}$ can be expressed in terms of derivative of elements of $\bm{\Sigma}^{y}_{k+1}$ as $\frac{\partial[(\bm{\Sigma}^{y}_{k+1})^{-1}]_{a,m}}{\partial\hat{\mathbf{x}}_{k}}=\sum_{c,d}-[(\bm{\Sigma}^{y}_{k+1})^{-1}]_{a,c}[(\bm{\Sigma}^{y}_{k+1})^{-1}]_{d,m}\frac{\partial[\bm{\Sigma}^{y}_{k+1}]_{c,d}}{\partial\hat{\mathbf{x}}_{k}}$. Using similar steps as used to obtain an upper-bound on $\left\|\frac{\partial[\bm{\Sigma}_{k+1|k}]_{l,m}}{\partial\hat{\mathbf{x}}_{k}}\right\|_{2}$ in Claim~\ref{claim:IUKF stable predict sig derivative bound} of Lemma~\ref{lemma: qi jacobian bound}, we can show that $\left\|\frac{\partial[\bm{\Sigma}^{y}_{k+1}]_{c,d}}{\partial\hat{\mathbf{x}}_{k}}\right\|_{2}\leq 4\delta_{h}\bar{h}c'$ such that the bound on elements of $(\bm{\Sigma}^{y}_{k+1})^{-1}$ yields $\left\|\frac{\partial[(\bm{\Sigma}^{y}_{k+1})^{-1}]_{a,m}}{\partial\hat{\mathbf{x}}_{k}}\right\|_{2}\leq\frac{n_{x}^{2}4\delta_{h}\bar{h}c'}{\underline{y}^{2}}$. With these bounds on magnitudes of all $[\bm{\Sigma}^{xy}_{k+1}]_{j,a}$ and $[(\bm{\Sigma}^{y}_{k+1})^{-1}]_{a,m}$ elements along with the bounds on their derivatives, \eqref{eqn:IUKF stable K derivative terms} gives $\left\|\frac{\partial[\mathbf{K}_{k+1}]_{j,m}}{\partial\hat{\mathbf{x}}_{k}}\right\|_{2}\leq c_{k}$ where $c_{k}=\frac{c_{xy}n_{y}}{|\underline{y}|}+\frac{n_{y}n_{x}^{2}c_{xy}}{\underline{y}^{2}}(\delta_{f}\delta_{h}+\delta_{h}(\delta_{f}+\sqrt{\bar{\sigma}(n_{x}+\kappa)}))$.
\end{claimproof}

Now, under the assumptions of Theorem~\ref{theorem:forward ukf stability}, $\|\mathbf{K}_{k+1}\|\leq\bar{k}=\bar{\sigma}\bar{\gamma}\bar{h}\bar{\beta}/\hat{r}$ (as obtained intermediately in proof of \cite[Theorem~2]{singh2022inverse_part1}) such that Lemma~\ref{lemma:vector matrix bounds}\textbf{(b)} yields $\sum_{m=1}^{n_{y}}|[\mathbf{K}_{k+1}]_{j,m}|\leq\sqrt{n_{y}}\bar{k}$. Finally, using this bound along with Claims~\ref{claim:IUKF stable tk term bound}-\ref{claim:IUKF stable derivative of K bound} in \eqref{eqn:IUKF stable jac K term j-th row}, we have $\left\|\left[\frac{\partial(\mathbf{K}_{k+1}\mathbf{t}_{k})}{\partial\hat{\mathbf{x}}_{k}}\right]_{(j,:)}\right\|_{2}\leq\sqrt{n_{y}}\bar{k}\bar{h}c'+2n_{y}c_{k}\delta_{h}$ such that Claim~\ref{claim:IUKF stable K j row bound} is satisfied with $c_{t}=\sqrt{n_{y}}\bar{k}\bar{h}c'+2n_{y}c_{k}\delta_{h}$.

\section{Proof of Theorem~\ref{thm:consistency}}
\label{App-thm-consistency}
We prove the theorem by the principle of mathematical induction. Define the prediction and estimation errors as $\hat{\widetilde{\mathbf{x}}}_{k|k-1}\doteq\hat{\mathbf{x}}_{k}-\doublehat{\mathbf{x}}_{k|k-1}$ and $\hat{\widetilde{\mathbf{x}}}_{k}\doteq\hat{\mathbf{x}}_{k}-\doublehat{\mathbf{x}}_{k}$, respectively. Assume $\mathbb{E}[\hat{\widetilde{\mathbf{x}}}_{k}\hat{\widetilde{\mathbf{x}}}_{k}^{T}]\preceq\overline{\bm{\Sigma}}_{k}$. We show that the inequality also holds for $(k+1)$-th time step. Substituting \eqref{eqn:SLT state transition} in the I-UKF's recursions and using the symmetry of the generated sigma points, we have $\doublehat{\mathbf{x}}_{k+1|k}=\mathbf{U}^{z}_{k}\overline{\mathbf{F}}^{x}_{k}\doublehat{\mathbf{x}}_{k}$ and $\overline{\bm{\Sigma}}_{k+1|k}=\mathbf{U}^{z}_{k}\overline{\mathbf{F}}^{x}_{k}\overline{\bm{\Sigma}}_{k}(\overline{\mathbf{F}}^{x}_{k})^{T}\mathbf{U}^{z}_{k}+\mathbf{U}^{z}_{k}\overline{\mathbf{F}}^{v}_{k}\mathbf{R}_{k+1}(\overline{\mathbf{F}}^{v}_{k})^{T}\mathbf{U}^{z}_{k}$. Hence, $\hat{\widetilde{\mathbf{x}}}_{k+1|k}=\mathbf{U}^{z}_{k}\overline{\mathbf{F}}^{x}_{k}\hat{\widetilde{\mathbf{x}}}_{k}+\mathbf{U}^{z}_{k}\overline{\mathbf{F}}^{v}_{k}\mathbf{v}_{k+1}$ such that $\mathbb{E}[\hat{\widetilde{\mathbf{x}}}_{k+1|k}\hat{\widetilde{\mathbf{x}}}_{k+1|k}^{T}]=\mathbf{U}^{z}_{k}\overline{\mathbf{F}}^{x}_{k}\mathbb{E}[\hat{\widetilde{\mathbf{x}}}_{k}\hat{\widetilde{\mathbf{x}}}_{k}^{T}](\overline{\mathbf{F}}^{x}_{k})^{T}\mathbf{U}^{z}_{k}+\mathbf{U}^{z}_{k}\overline{\mathbf{F}}^{v}_{k}\mathbf{R}_{k+1}(\overline{\mathbf{F}}^{v}_{k})^{T}\mathbf{U}^{z}_{k}$. Since $\mathbb{E}[\hat{\widetilde{\mathbf{x}}}_{k}\hat{\widetilde{\mathbf{x}}}_{k}^{T}]\preceq\overline{\bm{\Sigma}}_{k}$, we have $\mathbb{E}[\hat{\widetilde{\mathbf{x}}}_{k+1|k}\hat{\widetilde{\mathbf{x}}}_{k+1|k}^{T}]\preceq\overline{\bm{\Sigma}}_{k+1|k}$. Similarly, using \eqref{eqn:SLT observation}, we have $\hat{\mathbf{a}}_{k+1|k}=\mathbf{U}^{a}_{k+1}\overline{\mathbf{G}}_{k+1}\doublehat{\mathbf{x}}_{k+1|k}$ and $\overline{\bm{\Sigma}}^{a}_{k+1}=\mathbf{U}^{a}_{k+1}\overline{\mathbf{G}}_{k+1}\overline{\bm{\Sigma}}_{k+1|k}\overline{\mathbf{G}}_{k+1}^{T}\mathbf{U}^{a}_{k+1}+\overline{\mathbf{R}}_{k+1}$ with $\overline{\bm{\Sigma}}^{xa}_{k+1}=\overline{\bm{\Sigma}}_{k+1|k}\overline{\mathbf{G}}_{k+1}^{T}\mathbf{U}^{a}_{k+1}$. Again, $\hat{\widetilde{\mathbf{x}}}_{k+1}=(\mathbf{I}-\overline{\mathbf{K}}_{k+1}\mathbf{U}^{a}_{k+1}\overline{\mathbf{G}}_{k+1})\hat{\widetilde{\mathbf{x}}}_{k+1|k}-\overline{\mathbf{K}}_{k+1}\bm{\epsilon}_{k+1}$, which implies $\mathbb{E}[\hat{\widetilde{\mathbf{x}}}_{k+1}\hat{\widetilde{\mathbf{x}}}_{k+1}^{T}]=(\mathbf{I}-\overline{\mathbf{K}}_{k+1}\mathbf{U}^{a}_{k+1}\overline{\mathbf{G}}_{k+1})\mathbb{E}[\hat{\widetilde{\mathbf{x}}}_{k+1|k}\hat{\widetilde{\mathbf{x}}}_{k+1|k}^{T}](\mathbf{I}-\overline{\mathbf{K}}_{k+1}\mathbf{U}^{a}_{k+1}\overline{\mathbf{G}}_{k+1})^{T}+\overline{\mathbf{K}}_{k+1}\overline{\mathbf{R}}_{k+1}\overline{\mathbf{K}}_{k+1}^{T}$. Finally, using $\mathbb{E}[\hat{\widetilde{\mathbf{x}}}_{k+1|k}\hat{\widetilde{\mathbf{x}}}_{k+1|k}^{T}]\preceq\overline{\bm{\Sigma}}_{k+1|k}$, we have $\mathbb{E}[\hat{\widetilde{\mathbf{x}}}_{k+1}\hat{\widetilde{\mathbf{x}}}_{k+1}^{T}]\preceq\overline{\bm{\Sigma}}_{k+1}$.

\section{Proof of Theorem~\ref{theorem:RKHS-UKF stability}}
\label{App-thm-RKHS-UKF}
From $\mathbf{K}_{k+1}=\bm{\Sigma}^{zy}_{k+1}(\bm{\Sigma}^{y}_{k+1})^{-1}$, the sub-matrix $\mathbf{K}^{1}_{k+1}=\bm{\Sigma}^{xy}_{k+1}(\bm{\Sigma}^{y}_{k+1})^{-1}$ such that 
\par\noindent\small
\begin{align*}
    &\mathbf{K}^{1}_{k+1}=\bm{\Sigma}_{k+1|k}\mathbf{U}^{xy}_{k+1}\nabla\bm{\Phi}(\hat{\mathbf{x}}_{k+1|k})^{T}\mathbf{B}^{T}\mathbf{U}^{\phi 2}_{k+1}\\
    &\times\left(\mathbf{U}^{\phi 2}_{k+1}\mathbf{B}\nabla\bm{\Phi}(\hat{\mathbf{x}}_{k+1|k})\bm{\Sigma}_{k+1|k}\nabla\bm{\Phi}(\hat{\mathbf{x}}_{k+1|k})^{T}\mathbf{B}^{T}\mathbf{U}^{\phi 2}_{k+1}+\widetilde{\mathbf{R}}_{k+1}\right)^{-1},
\end{align*}
\normalsize
and $\bm{\Sigma}_{k+1}=\bm{\Sigma}_{k+1|k}-\mathbf{K}^{1}_{k+1}\bm{\Sigma}^{y}_{k+1}(\mathbf{K}^{1}_{k+1})^{T}$ yields
\par\noindent\small
\begin{align*}
    &\bm{\Sigma}_{k+1}=\bm{\Sigma}_{k+1|k}-\bm{\Sigma}_{k+1|k}\mathbf{U}^{xy}_{k+1}\nabla\bm{\Phi}(\hat{\mathbf{x}}_{k+1|k})^{T}\mathbf{B}^{T}\mathbf{U}^{\phi 2}_{k+1}\\
    &\times\left(\mathbf{U}^{\phi 2}_{k+1}\mathbf{B}\nabla\bm{\Phi}(\hat{\mathbf{x}}_{k+1|k})\bm{\Sigma}_{k+1|k}\nabla\bm{\Phi}(\hat{\mathbf{x}}_{k+1|k})^{T}\mathbf{B}^{T}\mathbf{U}^{\phi 2}_{k+1}+\widetilde{\mathbf{R}}_{k+1}\right)^{-1}\\
    &\;\;\;\times\mathbf{U}^{\phi 2}_{k+1}\mathbf{B}\nabla\bm{\Phi}(\hat{\mathbf{x}}_{k+1|k})(\mathbf{U}^{xy}_{k+1})^{T}\bm{\Sigma}_{k+1|k}.
\end{align*}
\normalsize

Define $V_{k}(\widetilde{\mathbf{x}}_{k|k-1})=\widetilde{\mathbf{x}}_{k|k-1}^{T}\bm{\Sigma}_{k|k-1}^{-1}\widetilde{\mathbf{x}}_{k|k-1}$. Then, using the independence of noise terms and zero mean, we have
\par\noindent\small
\begin{align}
    &\mathbb{E}[\mathbf{V}_{k+1}(\widetilde{\mathbf{x}}_{k+1|k})|\widetilde{\mathbf{x}}_{k|k-1}]\nonumber\\
    &=\widetilde{\mathbf{x}}_{k|k-1}(\mathbf{I}-\mathbf{K}^{1}_{k}\mathbf{U}^{\phi 2}_{k}\mathbf{B}\nabla\bm{\Phi}(\hat{\mathbf{x}}_{k|k-1}))^{T}\nabla\bm{\Phi}(\hat{\mathbf{x}}_{k|k})^{T}\mathbf{A}^{T}\mathbf{U}^{\phi 1}_{k}\bm{\Sigma}_{k+1|k}^{-1}\nonumber\\
    &\;\;\;\times\mathbf{U}^{\phi 1}_{k}\mathbf{A}\nabla\bm{\Phi}(\hat{\mathbf{x}}_{k|k})(\mathbf{I}-\mathbf{K}^{1}_{k}\mathbf{U}^{\phi 2}_{k}\mathbf{B}\nabla\bm{\Phi}(\hat{\mathbf{x}}_{k|k-1}))\widetilde{\mathbf{x}}_{k|k-1}\nonumber\\
    &+\mathbb{E}[\mathbf{v}_{k}^{T}(\mathbf{K}^{1}_{k})^{T}\nabla\bm{\Phi}(\hat{\mathbf{x}}_{k|k})^{T}\mathbf{A}^{T}\mathbf{U}^{\phi 1}_{k}\bm{\Sigma}_{k+1|k}^{-1}\mathbf{U}^{\phi 1}_{k}\mathbf{A}\nabla\bm{\Phi}(\hat{\mathbf{x}}_{k|k})\nonumber\\
    &\;\;\;\times\mathbf{K}^{1}_{k}\mathbf{v}_{k}|\widetilde{\mathbf{x}}_{k|k-1}]\nonumber\\
    &+\mathbb{E}[\mathbf{w}_{k}^{T}\bm{\Sigma}^{-1}_{k+1|k}\mathbf{w}_{k}|\widetilde{\mathbf{x}}_{k|k-1}]+r_{k}+s_{k}+q_{k}+u_{k}+b_{k}\nonumber\\
    &+\mathbb{E}[\bm{\eta}_{k}^{T}\bm{\Sigma}_{k+1|k}^{-1}\bm{\eta}_{k}|\widetilde{\mathbf{x}}_{k|k-1}]+\mathbb{E}[2\mathbf{w}_{k}^{T}\bm{\Sigma}_{k+1|k}^{-1}\bm{\eta}_{k}|\widetilde{\mathbf{x}}_{k|k-1}]\label{eqn: RKHS-UKF stability Vk expression}
\end{align}
\normalsize
where
\par\noindent\small
\begin{align*}
    &r_{k}=\delta_{f}(\mathbf{x}_{k})-\mathbf{U}^{\phi 1}_{k}\mathbf{A}\nabla\bm{\Phi}(\hat{\mathbf{x}}_{k|k})\mathbf{K}^{1}_{k}\delta_{h}(\mathbf{x}_{k}))\\
    &+2\widetilde{\mathbf{x}}_{k|k-1}(\mathbf{I}-\mathbf{K}^{1}_{k}\mathbf{U}^{\phi 2}_{k}\mathbf{B}\nabla\bm{\Phi}(\hat{\mathbf{x}}_{k|k-1}))^{T}\nabla\bm{\Phi}(\hat{\mathbf{x}}_{k|k})^{T}\mathbf{A}^{T}\mathbf{U}^{\phi 1}_{k}\bm{\Sigma}_{k+1|k}^{-1}\\
    &\times((\mathbf{A}-\hat{\mathbf{A}}_{k})\bm{\Phi}(\hat{\mathbf{x}}_{k|k})-\mathbf{U}^{\phi 1}_{k}\mathbf{A}\nabla\bm{\Phi}(\hat{\mathbf{x}}_{k|k})\mathbf{K}^{1}_{k}(\mathbf{B}-\hat{\mathbf{B}}_{k-1})\bm{\Phi}(\hat{\mathbf{x}}_{k|k-1}),\\
    &s_{k}=\bm{\Phi}(\hat{\mathbf{x}}_{k|k})^{T}(\mathbf{A}-\hat{\mathbf{A}}_{k})^{T}\bm{\Sigma}_{k+1|k}^{-1}((\mathbf{A}-\hat{\mathbf{A}}_{k})\bm{\Phi}(\hat{\mathbf{x}}_{k|k})\\
    &-2\mathbf{U}^{\phi 1}_{k}\mathbf{A}\nabla\bm{\Phi}(\hat{\mathbf{x}}_{k|k})\mathbf{K}^{1}_{k}(\mathbf{B}-\hat{\mathbf{B}}_{k-1})\bm{\Phi}(\hat{\mathbf{x}}_{k|k-1})+2\delta_{f}(\mathbf{x}_{k})\\
    &-2\mathbf{U}^{\phi 1}_{k}\mathbf{A}\nabla\bm{\Phi}(\hat{\mathbf{x}}_{k|k})\mathbf{K}^{1}_{k}\delta_{h}(\mathbf{x}_{k})),\\
    &q_{k}=\bm{\Phi}(\hat{\mathbf{x}}_{k|k-1})^{T}(\mathbf{B}-\hat{\mathbf{B}}_{k-1})^{T}(\mathbf{K}^{1}_{k})^{T}\nabla\bm{\Phi}(\hat{\mathbf{x}}_{k|k})^{T}\mathbf{A}^{T}\mathbf{U}^{\phi 1}_{k}\bm{\Sigma}_{k+1|k}^{-1}\\
    &\times(\mathbf{U}^{\phi 1}_{k}\mathbf{A}\nabla\bm{\Phi}(\hat{\mathbf{x}}_{k|k})\mathbf{K}^{1}_{k}(\mathbf{B}-\hat{\mathbf{B}}_{k-1})\bm{\Phi}(\hat{\mathbf{x}}_{k|k-1})-2\delta_{f}(\mathbf{x}_{k})\\
    &+2\mathbf{U}^{\phi 1}_{k}\mathbf{A}\nabla\bm{\Phi}(\hat{\mathbf{x}}_{k|k})\mathbf{K}^{1}_{k}\delta_{h}(\mathbf{x}_{k})),\\
    &u_{k}=\delta_{f}(\mathbf{x}_{k})^{T}\bm{\Sigma}_{k+1|k}^{-1}(\delta_{f}(\mathbf{x}_{k})-2\mathbf{U}^{\phi 1}_{k}\mathbf{A}\nabla\bm{\Phi}(\hat{\mathbf{x}}_{k|k})\mathbf{K}^{1}_{k}\delta_{h}(\mathbf{x}_{k})),\\
    &b_{k}=\delta_{h}(\mathbf{x}_{k})^{T}(\mathbf{K}^{1}_{k})^{T}\nabla\bm{\Phi}(\hat{\mathbf{x}}_{k|k})^{T}\mathbf{A}^{T}\mathbf{U}^{\phi 1}_{k}\bm{\Sigma}_{k+1|k}^{-1}\mathbf{U}^{\phi 1}_{k}\mathbf{A}\nabla\bm{\Phi}(\hat{\mathbf{x}}_{k|k})\\
    &\;\;\;\times\mathbf{K}^{1}_{k}\delta_{h}(\mathbf{x}_{k}).
\end{align*}
\normalsize

The proof involves appropriately bounding all the terms in the $\mathbb{E}[\mathbf{V}_{k+1}(\widetilde{\mathbf{x}}_{k+1|k})|\widetilde{\mathbf{x}}_{k|k-1}]$ expression such that both the conditions of Lemma~\ref{lemma:exponential boundedness} are satisfied. Following similar steps as in proof of \cite[Theorem~2]{singh2022inverse_part1}, we can show that under the assumptions of Theorem~\ref{theorem:RKHS-UKF stability}, the following bounds hold.
\par\noindent\small
\begin{align}
    &(\mathbf{U}^{\phi 1}_{k}\mathbf{A}\nabla\bm{\Phi}(\hat{\mathbf{x}}_{k|k})(\mathbf{I}-\mathbf{K}^{1}_{k}\mathbf{U}^{\phi 2}_{k}\mathbf{B}\nabla\bm{\Phi}(\hat{\mathbf{x}}_{k|k-1})))^{T}\bm{\Sigma}_{k+1|k}^{-1}\mathbf{U}^{\phi 1}_{k}\mathbf{A}\nonumber\\
    &\times\nabla\bm{\Phi}(\hat{\mathbf{x}}_{k|k})(\mathbf{I}-\mathbf{K}^{1}_{k}\mathbf{U}^{\phi 2}_{k}\mathbf{B}\nabla\bm{\Phi}(\hat{\mathbf{x}}_{k|k}))\preceq(1-\lambda)\bm{\Sigma}_{k|k-1}^{-1},\label{eqn:RKHS-UKF lambda bound}\\
    &\mathbb{E}[\mathbf{w}_{k}^{T}\bm{\Sigma}_{k+1|k}^{-1}\mathbf{w}_{k}|\widetilde{\mathbf{x}}_{k|k-1}]\leq c_{1},\label{eqn:RKHS-UKF wk bound}\\
    &\mathbb{E}[\mathbf{v}_{k}^{T}(\mathbf{K}^{1}_{k})^{T}\nabla\bm{\Phi}(\hat{\mathbf{x}}_{k|k})^{T}\mathbf{A}^{T}\mathbf{U}^{\phi 1}_{k}\bm{\Sigma}_{k+1|k}^{-1}\mathbf{U}^{\phi 1}_{k}\mathbf{A}\nabla\bm{\Phi}(\hat{\mathbf{x}}_{k|k})\nonumber\\
    &\;\;\;\times\mathbf{K}^{1}_{k}\mathbf{v}_{k}|\widetilde{\mathbf{x}}_{k|k-1}]\leq c_{2},\label{eqn:RKHS-UKF vk bound}
\end{align}
\normalsize
where $1-\lambda=\left(1+\frac{\widetilde{q}}{\overline{\sigma}(\overline{\alpha}\overline{a}\overline{\phi}(1+\overline{k\beta b\phi})^{2})}\right)^{-1}$ with $\overline{k}=\overline{\sigma\gamma\phi\beta b}/\widetilde{r}$ such that $0<\lambda<1$, $c_{1}=\overline{q}n_{x}/\underline{\sigma}$ and $c_{2}=\overline{k}^{2}\overline{\phi}^{2}\overline{\alpha}^{2}\overline{a}^{2}\overline{r}n_{y}/\underline{\sigma}$.

We will now bound the remaining terms in the right-hand side of \eqref{eqn: RKHS-UKF stability Vk expression}.
\begin{claim}\label{claim:RKHS-UKF Vk terms}
    There exist constants $c_{3}$, $c_{4}$, $c_{5}$, $c_{6}$ and $c_{7}$ satisfying $r_{k}\leq c_{3}$, $s_{k}\leq c_{4}$, $q_{k}\leq c_{5}$, $u_{k}\leq c_{6}$ and $b_{k}\leq c_{7}$ for all $k\geq 0$.
\end{claim}
\begin{claimproof}
    First, we upper-bound $r_{k}$. The Gaussian kernel function $K(\cdot,\cdot)$ has maximum value $1$ such that $\|\bm{\Phi}(\cdot)\|_{2}\leq\sqrt{L}$. Also, due to the projection of the coefficient matrix estimates, $\|\hat{\mathbf{A}}_{k}\|\leq\overline{a}$ and $\|\hat{\mathbf{B}}_{k}\|\leq\overline{b}$ for all $k\geq 0$. Then, using these bounds along with the Assumption~\textbf{A4} in Theorem~\ref{theorem:RKHS-UKF stability}, we have $ \|(\mathbf{A}-\hat{\mathbf{A}}_{k})\bm{\Phi}(\hat{\mathbf{x}}_{k|k})\|_{2}\leq 2\overline{a}\sqrt{L}$ and $\|(\mathbf{B}-\hat{\mathbf{B}}_{k-1})\bm{\Phi}(\hat{\mathbf{x}}_{k|k-1})\|_{2}\leq\overline{b}\sqrt{L}$.

    Next, we need an upper-bound on $\|\widetilde{\mathbf{x}}_{k|k-1}\|_{2}$. Since true state $\mathbf{x}_{k}$ lies in $\mathcal{X}$, we have $\|\mathbf{x}_{k}\|_{2}\leq\epsilon$ for some $\epsilon>0$. Also, $\hat{\mathbf{x}}_{k|k-1}=\sum_{i=0}^{2n_{z}}\omega_{i}\hat{\mathbf{A}}_{k-1}\bm{\Phi}([\mathbf{s}_{i,k-1}]_{1:n_{x}})$ as used to obtain \eqref{eqn:RKHS-UKF state prediction error} as well. Since $\hat{\mathbf{A}}_{k-1}$ and $\bm{\Phi}(\cdot)$ are bounded, we have $\|\widetilde{\mathbf{x}}_{k|k-1}\|_{2}\leq\|\mathbf{x}_{k}\|_{2}+\sum_{i=0}^{2n_{z}}\omega_{i}\|\hat{\mathbf{A}}_{k-1}\bm{\Phi}([\mathbf{s}_{i,k-1}]_{1:n_{x}})\|_{2}\leq\epsilon+\overline{a}\sqrt{L}$. With these bounds and other bounds assumed on various matrices, we obtain the required bound on $r_{k}$ with $c_{3}=2(\epsilon+\overline{a}\sqrt{L})(1+\overline{k\beta b\phi})\overline{\phi a\alpha}(2\overline{a}\sqrt{L}+\overline{\alpha a\phi k}(2\overline{b}\sqrt{L})+\overline{f}+\overline{\alpha a\phi kh})/\underline{\sigma}$.

    Following similar steps, we can obtain the other bounds of the claim with $c_{4}=4\overline{a}\sqrt{L}(\overline{a}\sqrt{L}+2\overline{\alpha a\phi kb}\sqrt{L}+\overline{f}+\overline{\alpha a\phi kh})/\underline{\sigma}$, $c_{5}=4\overline{b}\overline{k\phi a\alpha}\sqrt{L}(\overline{\alpha a\phi k}\overline{b}\sqrt{L}+\overline{f}+\overline{\alpha a\phi kh})/\underline{\sigma}$, $c_{6}=\overline{f}(\overline{f}+2\overline{\alpha a\phi kh})/\underline{\sigma}$ and $c_{7}=\overline{\alpha}^{2}\overline{a}^{2}\overline{\phi}^{2}\overline{k}^{2}\overline{h}^{2}/\underline{\sigma}$.
\end{claimproof}
\begin{claim}\label{claim:RKHS-UKF exp terms}
For the projection error $\bm{\eta}_{k}$, we have $\mathbb{E}[\bm{\eta}_{k}^{T}\bm{\Sigma}_{k+1|k}^{-1}\bm{\eta}_{k}|\widetilde{\mathbf{x}}_{k|k-1}]\leq c_{1}$ and $\mathbb{E}[2\mathbf{w}_{k}^{T}\bm{\Sigma}_{k+1|k}^{-1}\bm{\eta}_{k}|\widetilde{\mathbf{x}}_{k|k-1}]\leq 2c_{1}$ where $c_{1}$ is as defined in \eqref{eqn:RKHS-UKF wk bound}.
\end{claim}
\begin{claimproof}
Using the bound on $\bm{\Sigma}_{k+1|k}$, we have $\bm{\eta}_{k}^{T}\bm{\Sigma}_{k+1|k}^{-1}\bm{\eta}_{k}\leq\frac{1}{\underline{\sigma}}\bm{\eta}_{k}^{T}\bm{\eta}_{k}$. Now, by the definition of projection $\Gamma(\cdot)$, we have $\mathbf{x}_{k+1}=\textrm{argmin}_{\mathbf{x}\in\mathcal{X}}{\|f(\mathbf{x}_{k})+\mathbf{w}_{k}-\mathbf{x}\|_{2}}$ which implies $ \|f(\mathbf{x}_{k})+\mathbf{w}_{k}-\mathbf{x}_{k+1}\|_{2}\leq\|f(\mathbf{x}_{k})+\mathbf{w}_{k}-\mathbf{x}\|_{2}$ for all $\mathbf{x}\in\mathcal{X}$. In particular, $f(\mathbf{x}_{k})\in\mathcal{X}$ and the projection error $\bm{\eta}_{k}=\mathbf{x}_{k+1}-(f(\mathbf{x}_{k})+\mathbf{w}_{k})$ such that $\|-\bm{\eta}_{k}\|_{2}\leq\|\mathbf{w}_{k}\|_{2}$ which implies $\bm{\eta}_{k}^{T}\bm{\eta}_{k}\leq\mathbf{w}_{k}^{T}\mathbf{w}_{k}$. Hence, $\bm{\eta}_{k}^{T}\bm{\Sigma}_{k+1|k}^{-1}\bm{\eta}_{k}\leq\frac{1}{\underline{\sigma}}\mathbf{w}_{k}^{T}\mathbf{w}_{k}$. Taking expectation and using the upper bound on noise covariance $\mathbf{Q}_{k}$ similar to the proof of \eqref{eqn:RKHS-UKF wk bound} (in \cite[Theorem~2]{singh2022inverse_part1}), we get the first bound of the claim.

Now, $\mathbb{E}[\mathbf{w}_{k}^{T}\bm{\Sigma}_{k+1|k}^{-1}\bm{\eta}_{k}|\widetilde{\mathbf{x}}_{k|k-1}]\leq\frac{1}{\underline{\sigma}}\mathbb{E}[\mathbf{w}_{k}^{T}\bm{\eta}_{k}|\widetilde{\mathbf{x}}_{k|k-1}]$ which on using Cauchy-Schwartz inequality yields $\mathbb{E}[\mathbf{w}_{k}^{T}\bm{\Sigma}_{k+1|k}^{-1}\bm{\eta}_{k}|\widetilde{\mathbf{x}}_{k|k-1}]\\\leq\frac{1}{\underline{\sigma}}\sqrt{\mathbb{E}[\mathbf{w}_{k}^{T}\mathbf{w}_{k}]}\sqrt{\mathbb{E}[\bm{\eta}_{k}^{T}\bm{\eta}_{k}|\widetilde{\mathbf{x}}_{k|k-1}]}$. Again, using $\bm{\eta}_{k}^{T}\bm{\eta}_{k}\leq\mathbf{w}_{k}^{T}\mathbf{w}_{k}$, we have $\mathbb{E}[\mathbf{w}_{k}^{T}\bm{\Sigma}_{k+1|k}^{-1}\bm{\eta}_{k}|\widetilde{\mathbf{x}}_{k|k-1}]\leq\frac{1}{\underline{\sigma}}\mathbb{E}[\mathbf{w}_{k}^{T}\mathbf{w}_{k}]$ such that the second bound of the claim can be obtained using $\mathbf{Q}_{k}\preceq\overline{q}\mathbf{I}$.
\end{claimproof}

Using the bounds \eqref{eqn:RKHS-UKF lambda bound}-\eqref{eqn:RKHS-UKF vk bound} and Claims~\ref{claim:RKHS-UKF Vk terms} and \ref{claim:RKHS-UKF exp terms}, Theorem~\ref{theorem:RKHS-UKF stability} can be proved using Lemma~\ref{lemma:exponential boundedness} similar to the proof of \cite[Theorem~2]{singh2022inverse_part1}.

\bibliographystyle{IEEEtran}
\bibliography{references}
\end{document}